\documentclass[11pt]{article}
\usepackage{amssymb,amsmath,amsfonts,amsthm,mathrsfs}
\usepackage{graphicx,graphics,psfrag,epsfig,calrsfs,cancel}
\usepackage[colorlinks=true, pdfstartview=FitV, linkcolor=blue, citecolor=red, urlcolor=blue]{hyperref}

\usepackage{caption,subfig,color}
\usepackage{tikz}
\usetikzlibrary{calc}

\usepackage{geometry}
\newtheorem{assumption}{Assumption}

\newtheorem{lemma}{Lemma}
\newtheorem{proposition}{Proposition}
\newtheorem{theorem}{Theorem}
\newtheorem{corollary}{Corollary}

\newcommand{\N}{{\mathbb N}}
\newcommand{\Z}{{\mathbb Z}}
\newcommand{\R}{{\mathbb R}}
\newcommand{\C}{{\mathbb C}}

\newcommand{\Gjn}{{\mathcal G}_j^{\, n}}
\newcommand{\GGjn}{{\mathbb G}_j^n}
\newcommand{\Hjn}{{\mathcal H}_j^{\, n}}

\begin{document}

\title{The Green's function of the Lax-Wendroff \\
and Beam-Warming schemes}

\author{Jean-Fran\c{c}ois {\sc Coulombel}\thanks{Institut de Math\'ematiques de Toulouse - UMR 5219, Universit\'e de Toulouse ;
CNRS, Universit\'e Paul Sabatier, 118 route de Narbonne, 31062 Toulouse Cedex 9 , France. Research of J.-F. C. was supported
by ANR project Nabuco, ANR-17-CE40-0025. Email: {\tt jean-francois.coulombel@math.univ-toulouse.fr}}}
\date{\today}
\maketitle

\begin{abstract}
We prove a sharp uniform generalized Gaussian bound for the Green's function of the Lax-Wendroff and Beam-Warming schemes.
Our bound highlights the spatial region that leads to the well-known (rather weak) instability of these schemes in the maximum norm.
We also recover uniform bounds in the maximum norm when these schemes are applied to initial data of bounded variation.
\end{abstract}
\bigskip


\bigskip
\bigskip


For $1 \le q<+\infty$, we let $\ell^{\, q}(\Z;\C)$ denote the Banach space of complex valued sequences indexed by $\Z$ and such that the norm:
$$
\| \, u \, \|_{\ell^{\, q}} \, := \, \left( \, \sum_{j \in \Z} \, | \, u_j \, |^{\, q} \, \right)^{1/q} \, ,
$$
is finite. We also let $\ell^{\, \infty}(\Z;\C)$ denote the Banach space of bounded complex valued sequences indexed by $\Z$ and equipped with
the norm:
$$
\| \, u \, \|_{\ell^{\, \infty}} \, := \, \sup_{j \in \Z} \, | \, u_j \, | \, .
$$
We let $\N^*$ denote the set $\{ 1,2,3,\dots\}$ of positive integers. The letter $C$, resp. $c$, denotes some large, resp. small, positive constant
that may vary from one line to the other and possibly within the same line (for instance, we use the conventions $C+C=C$, $C/c=C$ and so on).
The dependence of the constants on the various involved parameters is made precise in the statement of the results and in the proofs.

\section{Introduction and main result}
\label{section1}

\subsection{Introduction}

The Lax-Wendroff and Beam-Warming schemes are two second order finite difference approximations of the transport equation:
\begin{equation}
\label{transport}
\begin{cases}
\partial_ t u \, + \, a \, \partial_x u \, = \, 0 \, ,& t \ge 0 \, ,\quad x \in \R \, ,\\
u|_{t=0} \, = \, u_0 \, .&
\end{cases}
\end{equation}
For simplicity, the velocity $a$ is a constant here which we choose to be positive (otherwise the Beam-Warming scheme defined in \eqref{BW} below
should be upwinded the other way round). The solution to \eqref{transport} corresponds to a transport along the characteristics which implies that any
$L^p$ norm of the initial condition $u_0$ is preserved:
$$
\forall \, t \ge 0 \, ,\quad \forall \, p \in [1,+\infty] \, ,\quad \| \, u(t,\cdot) \, \|_{L^p(\R)} \, = \, \| \, u_0 \, \|_{L^p(\R)} \, .
$$
A desirable feature of any finite difference approximation of \eqref{transport} is to satisfy a similar stability property in the discrete setting.

Finite difference approximations of the solution to \eqref{transport} amount to replacing $u$ by a piecewise constant function with respect to both space
and time. We thus introduce time and space steps $\Delta t>0$ and $\Delta x>0$ which we choose such that the ratio $\Delta t/\Delta x$ is a fixed positive
constant. From now on, we let $\lambda>0$ denote the fixed constant $a \, \Delta t/\Delta x$. The solution $u$ to \eqref{transport} is then approximated by
a sequence of functions defined as:
$$
\forall \, (n,j) \in \N \times \Z \, ,\quad \forall \, (t,x) \in [n \, \Delta t,(n+1) \, \Delta t) \times [j \, \Delta x,(j+1) \, \Delta x) \, ,\quad u_\Delta (t,x) \, := \, u_j^n \, ,
$$
where it remains to define inductively the sequence $(u_j^n)_{(n,j) \in \N \times \Z}$. For the linear transport equation \eqref{transport}, the so-called Lax-Wendroff
scheme reads:
\begin{equation}
\label{LW}
\begin{cases}
u_j^{n+1} \, = \, \dfrac{\lambda \, + \, \lambda^{\, 2}}{2} \, u_{j-1}^n \, + \, (1 \, - \, \lambda^{\, 2}) \, u_j^n \, + \, \dfrac{\lambda^{\, 2}  \, - \, \lambda}{2} \, u_{j+1}^n \, ,&
n \in \N \, ,\quad j \in \Z \, ,\\
u_j^0 \, = \, f_j \, ,& j \in \Z \, ,
\end{cases}
\end{equation}
and the Beam-Warming scheme reads\footnote{The choice $a>0$ is crucial here. For negative velocities, the stencil of the Beam-Warming should be shifted to
the points $j,j+1,j+2$.}:
\begin{equation}
\label{BW}
\begin{cases}
u_j^{n+1} \, = \, \dfrac{\lambda^{\, 2} \, - \, \lambda}{2} \, u_{j-2}^n \, + \, \lambda \, (2 \, - \, \lambda) \, u_{j-1}^n \, + \,
\dfrac{(1 \, - \, \lambda) \, (2 \, - \, \lambda)}{2} \, u_j^n \, ,& n \in \N \, ,\quad j \in \Z \, ,\\
u_j^0 \, = \, f_j \, ,& j \in \Z \, .
\end{cases}
\end{equation}
In both \eqref{LW} and \eqref{BW}, the sequence $(f_j)_{j \in \Z}$ is meant to provide with an approximation of the initial condition $u_0$ for \eqref{transport}.
A typical choice is:
$$
\forall \, j \in \Z \, ,\quad f_j \, := \, \dfrac{1}{\Delta x} \, \int_{j \, \Delta x}^{(j+1) \, \Delta x} \, u_0(x) \, {\rm d}x \, .
$$
We recall that in \eqref{LW} and \eqref{BW}, the fixed constant $\lambda$ stands for $a \, \Delta t/\Delta x$. Both schemes \eqref{LW} and \eqref{BW} are known
to lead (at least formally) to second order approximations of the solution to \eqref{transport} and to exhibit a \emph{dispersive} behavior. This phenomenon is
evidenced in Figure \ref{fig:dispersion} where we show the propagation of the step function:
$$
u_0(x) \, := \, \begin{cases}
1 \, ,& \text{\rm if } | \, x \, | \, \le 1/2 \, ,\\
0 \, ,& \text{\rm otherwise,}
\end{cases}
$$
at time $t=2$ by the numerical schemes \eqref{LW} and \eqref{BW} (we choose $a=1$ in the computations). Oscillating wavetrains are generated either behind
or ahead of the discontinuities in the solution. Our goal in this article is to give an accurate description of the so-called Green's function (a.k.a. the fundamental
solution) for the recurrence relations \eqref{LW} and \eqref{BW}. Our result gives a quantitative description of the oscillating wavetrains displayed in Figure
\ref{fig:dispersion}. Moreover it also gives a detailed analysis of the (rather weak) instability of the schemes \eqref{LW} and \eqref{BW} in $\ell^{\, \infty}(\Z;\C)$.
We now briefly recall some bibliographic references on this instability phenomenon and state our main result.

\begin{figure}\centering
\includegraphics[scale=0.25]{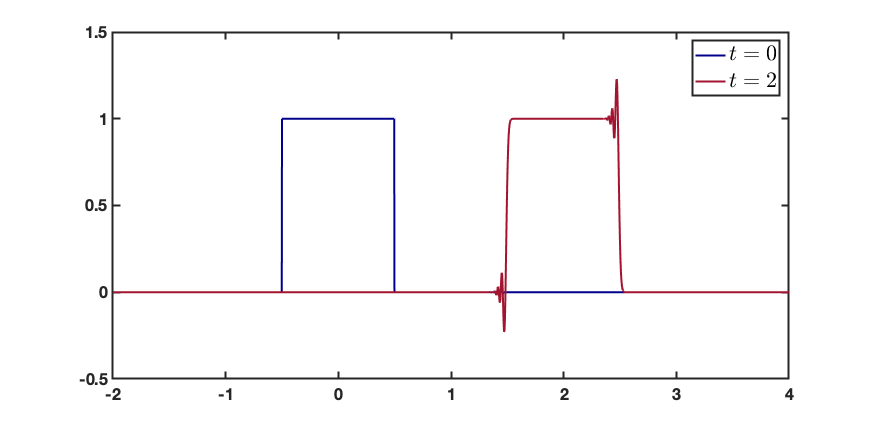}
\includegraphics[scale=0.25]{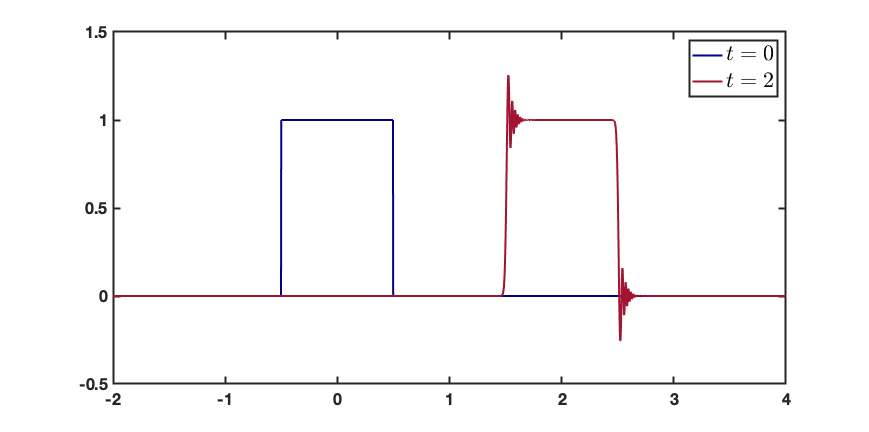}
\caption{Propagation of a step function by the Lax-Wendroff scheme \eqref{LW} (left) and Beam-Warming scheme \eqref{BW} (right) schemes with $\lambda=3/4$.
The initial condition is depicted in blue and the numerical solution at time $t=2$ is depicted in red.}
\label{fig:dispersion}
\end{figure}

\subsection{A reminder on the instability of dispersive schemes in the maximum norm}

Let us recall a few facts about Laurent (or convolution) operators. If $a \in \ell^{\, 1}(\Z;\C)$, we let $L_a$ denote the so-called Laurent operator
associated with the sequence $a$ \cite{TE,Nikolski}, which is defined by:
\begin{equation}
\label{laurent}
L_a \quad : \quad \big( u_j \big)_{j \in \Z} \, \longmapsto \, \left( \, \sum_{\ell \in \Z} \, a_\ell \, u_{j-\ell} \right)_{j \in \Z} \, = \, a \star u \, ,
\end{equation}
whenever the defining formula \eqref{laurent} for the sequence $L_a \, u$ makes sense. For instance, the scheme \eqref{LW} corresponds to the (finitely
supported) sequence:
$$
a_\ell \, := \, \begin{cases}
(\lambda \, + \, \lambda^{\, 2}) \, / \, 2 \, ,& \text{\rm if } \ell \, = \, 1 \, ,\\
1 \, - \, \lambda^{\, 2} \, ,& \text{\rm if } \ell \, = \, 0 \, ,\\
- \, (\lambda \, - \, \lambda^{\, 2}) \, / \, 2 \, ,& \text{\rm if } \ell \, = \, -1 \, ,\\
0 \, ,& \text{\rm otherwise,}
\end{cases}
$$
and the scheme \eqref{BW} corresponds to the (finitely supported) sequence:
$$
a_\ell \, := \, \begin{cases}
(1 \, - \, \lambda) \, (2 \, - \lambda) \, / \, 2 \, ,& \text{\rm if } \ell \, = \, 0 \, ,\\
\lambda \, (2 \, - \, \lambda) \, ,& \text{\rm if } \ell \, = \, 1 \, ,\\
- \, (\lambda \, - \, \lambda^{\, 2}) \, / \, 2 \, ,& \text{\rm if } \ell \, = \, 2 \, ,\\
0 \, ,& \text{\rm otherwise.}
\end{cases}
$$

We always consider a sequence $a \in \ell^{\, 1}(\Z;\C)$ in what follows, and we shall even assume for our main result that the sequence $a$ is finitely
supported. Young's inequality shows that the operator $L_a$ acts boundedly on $\ell^{\, q}(\Z;\C)$ for any $q \in [1,+\infty]$. The spectrum of $L_a$
is well-understood since the so-called L\'evy-Wiener Theorem \cite{newman} characterizes the invertible elements of $\ell^{\, 1}(\Z;\C)$ for the convolution
product (and we have the morphism property $L_a \circ L_b=L_{a \star b}$, where $\star$ denotes the convolution product on $\ell^{\, 1}(\Z;\C)$).
Namely, the spectrum of $L_a$ as an operator acting on $\ell^{\, q}(\Z;\C)$ does not depend on $q$ and is nothing but the image of the Fourier
transform of the sequence $a$, see \cite{TE}:
$$
\sigma \, (L_a) \, = \, \left\{ \, \sum_{\ell \in \Z} \, a_\ell \, {\rm e}^{\, \mathbf{i} \, \ell \, \theta} \, / \, \theta \in \R \right\} \, .
$$
Since $a$ belongs to $\ell^{\, 1}(\Z;\C)$, its Fourier transform is continuous on $\R$. It actually belongs to the so-called Wiener algebra, see \cite{kahane}.

Let us now take a closer look at the norm of $L_a$ when acting on either $\ell^{\, 2}(\Z;\C)$ or $\ell^{\, \infty}(\Z;\C)$. For future use, we let $\widehat{F}_a$
denote the Fourier transform of the sequence $a$:
\begin{equation}
\label{defFourier}
\forall \, \theta \in \R \, ,\quad \widehat{F}_a(\theta) \, := \, \sum_{\ell \in \Z} \, a_\ell \, {\rm e}^{\, \mathbf{i} \, \ell \, \theta} \, .
\end{equation}
By Fourier analysis, see \cite{gko,strikwerda-wade,TE}, we have:
$$
\| \, L_a \, \|_{\ell^{\, 2} \rightarrow \ell^{\, 2}} \, = \, \| \, \widehat{F}_a \, \|_{L^\infty(\R)} \, .
$$
Since the maps $(a \, \mapsto \, L_a)$ and $(a \, \mapsto \, \widehat{F}_a)$ are morphisms, this immediately gives:
$$
\forall \, n \in \N \, ,\quad \| \, (L_a)^{\, n} \, \|_{\ell^{\, 2} \rightarrow \ell^{\, 2}} \, = \, \| \, (\widehat{F}_a)^{\, n} \, \|_{L^\infty(\R)}
\, = \, \| \, \widehat{F}_a \, \|_{L^\infty(\R)}^{\, n} \, ,
$$
which gives the well-known von Neumann necessary and sufficient condition for $\ell^{\, 2}$-stability (see \cite{RM,gko,strikwerda-wade}):
$$
\sup_{n \in \N} \, \| \, (L_a)^{\, n} \, \|_{\ell^{\, 2} \rightarrow \ell^{\, 2}} \, < \, + \, \infty \quad \Longleftrightarrow \quad
\| \, \widehat{F}_a \, \|_{L^\infty(\R)} \, \le \, 1 \, .
$$
In the context of finite difference schemes, $\widehat{F}_a$ is usually referred to as the amplification factor. Under the von Neumann condition, the operator
$L_a$ is a contraction on $\ell^{\, 2}(\Z;\C)$ as is any of its powers. For the Lax-Wendroff \eqref{LW} and Beam-Warming \eqref{BW} schemes, we obtain (with
rather self-explanatory notation):
\begin{subequations}
\begin{align}
\forall \, \theta \in \R \, ,\quad
\widehat{F}_{\rm LW}(\theta) \, &= \, 1 \, - \, 2 \, \lambda^{\, 2} \, \sin^{\, 2} \, \dfrac{\theta}{2} \, + \, \mathbf{i} \, \lambda \, \sin \theta \, ,\label{defFLW} \\
\widehat{F}_{\rm BW}(\theta) \, &= \, 1 \, - \, 2 \, \lambda^{\, 2} \, \sin^{\, 2} \, \dfrac{\theta}{2} \,  \, - \, 4 \, \lambda \, (1 \, - \, \lambda) \, \sin^{\, 4} \, \dfrac{\theta}{2} \,
+ \, \mathbf{i} \, \lambda \, \sin \theta \, \left( 1 \, + \, 2 \, (1 \, - \, \lambda) \, \sin^{\, 2} \, \dfrac{\theta}{2} \right) \, .\label{defFBW}
\end{align}
\end{subequations}
We then compute :
\begin{align*}
\forall \, \theta \in \R \, ,\quad
\Big| \, \widehat{F}_{\rm LW}(\theta) \, \Big|^{\, 2} \, &= \, 1 \, - \, 4 \, \lambda^{\, 2} \, (1 \, - \, \lambda^{\, 2}) \, \sin^{\, 4} \, \dfrac{\theta}{2} \, , \\
\Big| \, \widehat{F}_{\rm BW}(\theta) \, \Big|^{\, 2} \, &= \, 1 \, - \, 4 \, \lambda \, (1 \, - \, \lambda)^{\, 2} \, (2-\lambda) \, \sin^{\, 4} \, \dfrac{\theta}{2} \, ,
\end{align*}
and we thus find that the Lax-Wendroff scheme is $\ell^{\, 2}$-stable for $\lambda \in (0,1]$ and that the Beam-Warming scheme is $\ell^{\, 2}$-stable for
$\lambda \in (0,2]$. In what follows, we avoid the trivial cases $\lambda=1$ in \eqref{LW} and $\lambda \in \{ 1,2 \}$ in \eqref{BW} for which the schemes
reduce to simple shift operators. We thus consider from now on :
\begin{itemize}
 \item the Lax-Wendroff scheme \eqref{LW} with $\lambda \in (0,1)$,
 \item or the Beam-Warming scheme \eqref{BW} with $\lambda \in (0,1) \cup (1,2)$.
\end{itemize}
In particular, there holds the so-called \emph{dissipation} condition:
\begin{equation}
\label{dissipation}
\forall \, \theta \in [ - \, \pi,\pi] \setminus \{ 0 \} \, ,\quad
\max \, \Big( \, \big| \, \widehat{F}_{\rm LW}(\theta) \, \big|,\big| \, \widehat{F}_{\rm BW}(\theta) \, \big| \, \Big) \, < \, 1 \, ,
\end{equation}
as well as the conservativity condition $\widehat{F}_{\rm LW}(0)=\widehat{F}_{\rm BW}(0)=1$.

The stability of the Laurent operator $L_a$ on $\ell^{\, \infty}(\Z;\C)$ is less easy to study since the norm\footnote{The norm of $L_a$ as an operator acting
on either $\ell^{\, 1}(\Z;\C)$ or $\ell^{\, \infty}(\Z;\C)$ coincides with the norm of $\widehat{F}_a$ in the Wiener algebra.}:
$$
\| \, L_a \, \|_{\ell^{\, \infty} \rightarrow \ell^{\, \infty}} \, = \, \sum_{\ell \in \Z} \, | \, a_\ell \, | \, ,
$$
is usually larger than $1$. The only favorable case is that when the coefficients $a_\ell$ are nonnegative real numbers, which corresponds to \emph{monotone}
schemes or equivalently to the case where the $a_\ell$'s are the jump probabilities of a random walk on $\Z$. There is no obvious reason why, even under a
dissipation condition as in \eqref{dissipation} for the Fourier transform $\widehat{F}_a$, the operator $L_a$ should be \emph{power bounded} when acting on
$\ell^{\, \infty}(\Z;\C)$. This stability property in the maximum norm has been thoroughly investigated in the fundamental contribution \cite{Thomee}, see also
\cite{Despres1,Diaconis-SaloffCoste} for some recent developments. The main result in \cite{Thomee} shows that stability in the maximum norm is related to
the Taylor expansion of the Fourier transform $\widehat{F}_a$ at any point where the modulus of $\widehat{F}_a$ attains the value $1$. For the Lax-Wendroff
and Beam-Warming schemes, the modulus of the Fourier transform equals $1$ only at the frequency $0$, see \eqref{dissipation}, and we compute the Taylor
expansions:
\begin{subequations}
\begin{align}
\widehat{F}_{\rm LW}(\theta) \, &= \, \exp \, \left ( \mathbf{i} \, \lambda \, \theta \, - \, \mathbf{i} \, \dfrac{\lambda \, (1 \, - \, \lambda^{\, 2})}{6} \, \theta^{\, 3}
\, - \, \dfrac{\lambda^{\, 2} \, (1 \, - \, \lambda^{\, 2})}{8} \, \theta^{\, 4} \, + \, O(\theta^{\, 5}) \right) \, ,\label{TaylorLW} \\
\widehat{F}_{\rm BW}(\theta) \, &= \, \exp \, \left ( \mathbf{i} \, \lambda \, \theta \, + \, \mathbf{i} \, \dfrac{\lambda \, (1 \, - \, \lambda) \, (2 \, - \, \lambda)}{6} \, \theta^{\, 3}
 \, - \, \dfrac{\lambda \, (1 \, - \, \lambda)^{\, 2} \, (2 \, - \lambda)}{8} \, \theta^{\, 4} \, + \, O(\theta^{\, 5}) \right) \, ,\label{TaylorBW}
\end{align}
\end{subequations}
as $\theta$ tends to $0$. Since a purely imaginary term of the form $\mathbf{i} \, \theta^{\, 3}$ arises in \eqref{TaylorLW} and \eqref{TaylorBW}
(with a power $3$ that is less than the first even power with a coefficient of negative real part), \cite[Theorem 3]{Thomee} shows that the associated
convolution operators are \emph{not} power bounded on $\ell^{\, \infty}(\Z;\C)$. In other words, the schemes \eqref{LW} and \eqref{BW} are not stable
in $\ell^{\, \infty}(\Z;\C)$ uniformly in time. Namely, there holds (with, again, quite self-explanatory notation):
$$
\sup_{n \in \N} \, \| \, (L_{\rm LW})^{\, n} \, \|_{\ell^{\, \infty} \rightarrow \ell^{\, \infty}} \, = \,
\sup_{n \in \N} \, \| \, (L_{\rm BW})^{\, n} \, \|_{\ell^{\, \infty} \rightarrow \ell^{\, \infty}} \, = \, + \, \infty \, .
$$
Actually, \cite[Theorem 3]{Thomee} gives the lower bounds:
$$
\| \, (L_{\rm LW})^{\, n} \, \|_{\ell^{\, \infty} \rightarrow \ell^{\, \infty}} \, \ge \, c \, n^{\, 1/12} \, ,\quad
\| \, (L_{\rm BW})^{\, n} \, \|_{\ell^{\, \infty} \rightarrow \ell^{\, \infty}} \, \ge \, c \, n^{\, 1/12} \, ,
$$
for a suitable constant $c>0$ (that only depends on $\lambda$). The \emph{sharp} growth rate:
$$
\| \, (L_{\rm LW})^{\, n} \, \|_{\ell^{\, \infty} \rightarrow \ell^{\, \infty}} \, \sim \, n^{\, 1/8} \, ,\quad
\| \, (L_{\rm BW})^{\, n} \, \|_{\ell^{\, \infty} \rightarrow \ell^{\, \infty}} \, \sim \, n^{\, 1/8} \, ,
$$
as $n$ tends to infinity, was proved in \cite{hedstrom1} (see also \cite{hedstrom2} for a general classification). We shall recover this growth rate later on by
making it even more precise (see Section \ref{section4} for more details).

\subsection{Main result}

From now on, we consider a convolution operator $L_a$ where $a \in \ell^{\, 1}(\Z;\C)$ is a given sequence whose associated Fourier transform
is denoted $\widehat{F}_a$, see \eqref{defFourier}. We are interested in giving an accurate description of the powers $L_a^{\, n}$, as the integer
$n$ becomes large, in a case that exhibits a \emph{dispersive} behavior as reported in Figure \ref{fig:dispersion}. We thus make the following
assumption.

\begin{assumption}
\label{hyp:1}
The complex valued sequence $a$ is finitely supported and sums to $1$:
$$
\sum_{\ell \in \Z} \, a_\ell \, = \, 1 \, .
$$
Its Fourier transform $\widehat{F}_a$ (that is a trigonometric polynomial) satisfies the dissipation condition:
\begin{equation}
\label{hyp:dissipation}
\forall \, \theta \in [ - \, \pi,\pi] \setminus \{ 0 \} \, ,\quad \big| \, \widehat{F}_a(\theta) \, \big| \, < \, 1 \, .
\end{equation}
Moreover, there exist a real number $\alpha$, a nonzero real number $c_3$, a positive real number $c_4$ and a holomorphic function $\varphi$
defined on a neighborhood of $0$ such that, as $\theta \in \C$ tends to zero, there holds:
\begin{equation}
\label{hyp:stabilite2}
\widehat{F}_a(\theta) \, = \, \exp \left( \, \mathbf{i} \, \alpha \, \theta \, - \, \mathbf{i} \, c_3 \, \theta^{\, 3} \, - \, c_4 \, \theta^{\, 4} \,
+ \, \theta^{\, 5} \, \varphi (\theta) \, \right) \, .
\end{equation}
\end{assumption}

\noindent Let us observe that since $\widehat{F}_a$ is a trigonometric polynomial, the definition \eqref{defFourier} shows that it extends to a holomorphic
function on the whole complex plane. Hence the holomorphy of the remainder $\varphi$ in \eqref{hyp:stabilite2} is automatic.

In order to state our main result, we let from now on $\boldsymbol{\delta}$ denote the discrete Dirac mass defined by:
$$
\forall \, j \in \Z \, ,\quad \boldsymbol{\delta}_j \, := \, \begin{cases}
1 \, ,& \text{\rm if } j \, = \, 0 \, ,\\
0 \, ,& \text{\rm otherwise.}
\end{cases}
$$
For any $n \in \N$ and $j \in \Z$, we then use the notation $\Gjn$ for the so-called Green's function associated with the operator $L_a$, that is:
\begin{equation}
\label{def:green}
\forall \, n \in \N \, ,\quad \forall \, j \in \Z \, ,\quad \Gjn \, := \, \big( \, L_a^{\, n} \, \boldsymbol{\delta} \, \big)_j \, .
\end{equation}
Since $L_a$ is a convolution operator, the Green's function $\mathcal{G}^{\, n}$ is nothing but the sequence $a$ convolved with itself $n-1$
times: $\mathcal{G}^{\, 0} =\boldsymbol{\delta}$, $\mathcal{G}^{\, 1} =a$, $\mathcal{G}^{\, 2} =a \star a$ and so on. Equivalently, we have:
$$
\forall \, n \in \N \, ,\quad (L_a)^{\, n} \, = \, L_{\, \mathcal{G}^{\, n}} \, .
$$
Our main result is the following.

\begin{theorem}
\label{thm:1}
Assume that the constant $c_3$ in Assumption \ref{hyp:1} is positive (which holds for the Lax-Wendroff scheme with $\lambda \in (0,1)$ and
for the Beam-Warming scheme with $\lambda \in (1,2)$). Then there exist two constants $C>0$ and $c>0$ such that the Green's function
$(\Gjn)_{(n,j) \in \N \times \Z}$ satisfies the uniform bounds:
\begin{multline}
\label{bound1}
\forall \, n \in \N^* \, ,\quad \forall \, j \in \Z \, ,\\
\big| \, \Gjn \, \big| \, \le \, \dfrac{C}{n^{\, 1/3}} \, \min \left( 1,\left( \dfrac{j \, - \, \alpha \, n}{n^{\, 1/3}} \right)^{-\, 1/4} \right) \,
\exp \left( - \, c \, \left( \dfrac{j \, - \, \alpha \, n}{n^{\, 1/3}} \right)^{3/2} \, \right) \, ,\quad \text{\rm if } j \, - \, \alpha \, n \, \ge \, 0 \, ,
\end{multline}
and:
\begin{multline}
\label{bound2}
\forall \, n \in \N^* \, ,\quad \forall \, j \in \Z \, ,\\
\big| \, \Gjn \, - \, \GGjn \, \big| \, \le \, \dfrac{C}{n^{\, 1/3}} \, \min \left( 1,\left( \dfrac{|j \, - \, \alpha \, n|}{n^{\, 1/3}} \right)^{-\, 1} \right) \,
\exp \left( - \, c \, \left( \dfrac{|j \, - \, \alpha \, n|}{n^{\, 1/3}} \right)^{3/2} \, \right) \, ,\quad \text{\rm if } j \, - \, \alpha \, n \, < \, 0 \, ,
\end{multline}
where $\GGjn$ is defined for $n \in \N^*$ and $j \in \Z$ as:
\begin{equation}
\label{defprincipal}
\GGjn \, := \, \dfrac{1}{\pi} \, \exp \left( - \, \dfrac{c_4 \, (j \, - \, \alpha \, n)^{\, 2}}{9 \, c_3^{\, 2} \, n} \, \right) \,
\cos \left( \dfrac{2 \, |j \, - \, \alpha \, n|^{\, 3/2}}{3 \, \sqrt{3 \, c_3 \, n}} \, - \, \dfrac{\pi}{4} \right) \,
\int_{-\sqrt{\frac{2 \, |j \, - \, \alpha \, n|}{3\, c_3 \, n}}}^{\sqrt{\frac{2 \, |j \, - \, \alpha \, n|}{3\, c_3 \, n}}} \,
{\rm e}^{\, - \, \sqrt{3 \, c_3 \, n \, |j \, - \, \alpha \, n|} \, u^{\, 2}} \, {\rm d}u \, .
\end{equation}

If $c_3$ is negative (which holds for the Beam-Warming scheme with $\lambda \in (0,1)$), the bounds depending on the sign of $j \, - \, \alpha \, n$
should be switched and $c_3$ should be replaced by $| \, c_3 \, |$ in \eqref{defprincipal}.
\end{theorem}

\noindent In what follows, only the values of $\GGjn$ for $j \, - \, \alpha \, n \, < \, 0$ will matter, but there is no problem defining $\GGjn$ for any
$(n,j) \in \N^* \times \Z$. An immediate consequence of Theorem \ref{thm:1} is the following corollary.

\begin{corollary}
\label{coro:1}
Assume that the constant $c_3$ in Assumption \ref{hyp:1} is positive. Then there exists a constant $C>0$ such that for any $n \in \N^*$, the Green's function
$(\Gjn)_{j \in \Z}$ satisfies:
\begin{equation}
\label{boundcoro1-1}
\sum_{j \in \Z \, / \, j \, - \, \alpha \, n \, \ge \, 0} \, \big| \, \Gjn \, \big| \, \le \, C \, ,
\end{equation}
and
\begin{equation}
\label{boundcoro1-2}
\sum_{j  \in \Z \, / \, j \, - \, \alpha \, n \, < \, 0} \, \big| \, \Gjn \, - \, \GGjn \, \big| \, \le \, C \, ,
\end{equation}
where the approximate Green's function $\GGjn$ is defined in \eqref{defprincipal}.
\end{corollary}

In other words, the reason for the nonuniform integrability of the Green's function $(\Gjn)_{j \in \Z}$ is located in the region $\{ j \, - \, \alpha \, n \, < \, 0 \}$
(assuming $c_3>0$), and the resulting instability is entirely described by means of the approximate Green's function $(\GGjn)_{j \in \Z}$. This observation
will be used systematically in Section \ref{section4} below.

Let us clarify the position of our main result with some recent advances on \emph{local limit theorems} for complex valued sequences. More precisely,
in probability theory (the case where the $a_\ell$'s are nonnegative real numbers), local limit theorems describe the asymptotic behavior of $\Gjn$
as $n$ becomes large, see for instance \cite[Chapter VII]{Petrov}. There have been recent developments of this theory in the case of complex valued
sequences, of which a culminating point is \cite{RSC1} (see \cite{RSC2} for results on multidimensional problems). Going more into the details, when
we apply \cite[Theorem 1.2]{RSC1} to sequences $a$ satisfying Assumption \ref{hyp:1}, we obtain\footnote{The reader may also consult
\cite[Proposition 8.1]{RSC1} that deals with the specific case of the Lax-Wendroff scheme \eqref{LW}.}:
$$
\Gjn \, = \, \dfrac{1}{n^{\, 1/3}} \, H_3^{\, \mathbf{i} \, c_3} \left( \, \dfrac{j \, - \, \alpha \, n}{n^{\, 1/3}} \, \right) \, + \, o(n^{\, - \, 1/3}) \, ,
$$
uniformly with respect to $j \in \Z$, and the function $H_3^{\, \mathbf{i} \, c_3}$ is defined as the oscillatory integral\footnote{The function $H_3^{\, \mathbf{i}/3}$
is the classical Airy function, which is already a strong indication why, for $c_3>0$, $\mathcal{G}_j^{\, n}$ should have strong decay properties for
$j \, - \, \alpha \, n >0$.}:
$$
\forall \, x \in \R \, ,\quad
H_3^{\, \mathbf{i} \, c_3}(x) \, := \, \dfrac{1}{2 \, \pi} \, \int_\R \, {\rm e}^{\, - \, \mathbf{i} \, x \, u} \, {\rm e}^{\, - \, \mathbf{i} \, c_3 \, u^{\, 3}} \, {\rm d}u \, .
$$
Theorem \ref{thm:1} above does not give such an accurate (universal) description of $\Gjn$, but it gives global uniform bounds that cannot be obtained
(at least, not in a straightforward way) from \cite{RSC1}. This difference between Theorem \ref{thm:1} and \cite{RSC1} should be expected since the
analysis in \cite{RSC1} only retains the first two terms in the Taylor expansion \eqref{hyp:stabilite2} while we make here a very strong assumption on
the third term (the $O(\theta^{\, 4})$ term in \eqref{hyp:stabilite2} that produces most of the dissipation mechanism). Last, we observe that the uniform
bound in Theorem \ref{thm:1} for the case $c_3 \, (j  \, - \, \alpha \, n) \, \ge \, 0$ is compatible with the well-known fast decaying behavior of the Airy
function on $\R^+$. Theorem \ref{thm:1} and \cite{RSC1} should therefore be seen as complementary.
\bigskip

We provide in Figure \ref{fig:green} with a representation of the Green's function $(\Gjn)_{j \in \Z}$ for the Lax-Wendroff scheme \eqref{LW} for $\lambda=3/4$
at various time iterations. Entirely similar pictures may be generated for the Beam-Warming scheme \eqref{BW}. The right picture in Figure \ref{fig:green}
compares the exact Green's function $(\Gjn)_{j \in \Z}$, that is depicted in blue, with yet another approximate Green's function from the one given in Theorem
\ref{thm:1}. Namely, we plot in red the approximate Green's function $(\Hjn)_{j \in \Z}$ defined by:
\begin{equation}
\label{defHjn}
\forall (j,n) \in \Z \times \N^* \, ,\quad \Hjn \, := \, \begin{cases}
(3 \, c_3 \, n)^{\, - \, 1/3} \, \, \text{\rm Ai } \left( \dfrac{j \, - \, \alpha \, n}{(3 \, c_3 \, n)^{\, 1/3}} \right) \, ,& \text{\rm if $j \, - \, \alpha \, n \ge 0$,} \\
\dfrac{{\rm e}^{- \, \frac{c_4 \, (j \, - \, \alpha \, n)^{\, 2}}{9 \, c_3^{\, 2} \, n}}}{(3 \, c_3 \, n)^{\, 1/3}} \, \,
\text{\rm Ai } \left( \dfrac{j \, - \, \alpha \, n}{(3 \, c_3 \, n)^{\, 1/3}} \right) \, ,& \text{\rm if $j \, - \, \alpha \, n < 0$,}
\end{cases}
\end{equation}
where $\text{\rm Ai}$ denotes the Airy function. The values of $c_3$ and $c_4$ for the Lax-Wendroff scheme \eqref{LW} are:
$$
c_3 \, = \, \dfrac{\lambda \, (1 \, - \, \lambda^{\, 2})}{6} \, ,\quad c_4 \, = \, \dfrac{\lambda^{\, 2} \, (1 \, - \, \lambda^{\, 2})}{8} \, .
$$
The Gaussian like factor in the definition of $\Hjn$ fits very well with the observed damping in the oscillations of $\Gjn$. With no Gaussian factor, the very
slow decay of the Airy function on $\R^-$ would not fit with the observed fast decaying behavior of $\Gjn$. However, we have not been able so far to obtain
accurate bounds for the difference $\Gjn -\Hjn$. This is postponed to a future work.
\bigskip

\begin{figure}\centering
\includegraphics[scale=0.35]{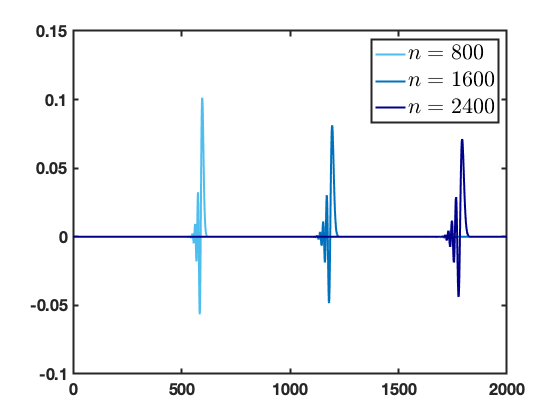}
\includegraphics[scale=0.35]{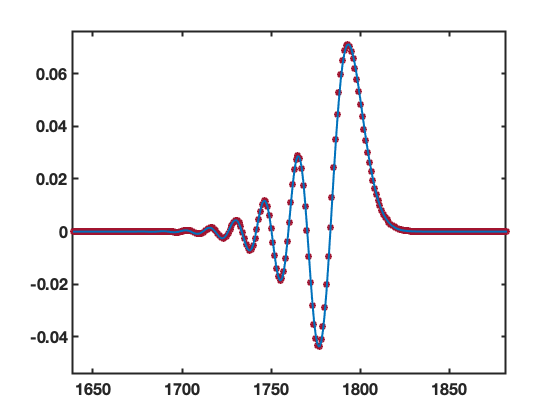}
\caption{The Green's function of the Lax-Wendroff scheme \eqref{LW} at various time iterations (left) and a zoom on the oscillations for $n=2400$
(right). The chosen parameter is $\lambda=3/4$. On the right: the Green's function $\Gjn$ in blue and the approximate Green's function $\Hjn$ defined
in \eqref{defHjn} in red dots.}
\label{fig:green}
\end{figure}

We now give the proof of Theorem \ref{thm:1} in the case $c_3>0$ and leave the case $c_3<0$ to the interested reader. The rest of this article is organized
as follows. In Section \ref{section2}, we prove the sharp bound \eqref{bound1}. We then give the proof of the bound \eqref{bound2} in Section \ref{section3}.
Corollary \ref{coro:1} as well as two other consequences of Theorem \ref{thm:1} are discussed in Section \ref{section4} in connection with \cite{hedstrom1,hedstrom2}
and \cite{ELR}.

\section{Proof of the main result. I. Uniform bound and the fast decaying tail}
\label{section2}

\subsection{Preliminary facts and notation}

First of all, let us fix a constant $\delta_0>0$ such that the function $\varphi$ in Assumption \ref{hyp:1} is holomorphic on the open square
$\{ z \in \C \, / \, \max \, ( \, |\text{\rm Re } z|,|\text{\rm Im } z| \, ) \, < \, 2 \, \delta_0 \}$. We then define the constant $C_0>0$ as:
\begin{equation}
\label{defC0}
C_0 \, := \, \max_{|\text{\rm Re } z| \le \delta_0 \, , \, |\text{\rm Im } z| \le \delta_0} \, | \, \varphi(z) \, | \, .
\end{equation}
In many arguments below, we shall use some contours that are located within the closed square $\{ z \in \C \, / \, \max \, ( \, |\text{\rm Re } z|,|\text{\rm Im } z| \, )
\, \le \, \delta_0 \}$ in order to be able to bound from above the modulus of $\varphi$ by the constant $C_0$. The terms that involve $\varphi$ will always be dealt
with as remainders. The constants $\delta_0$ and $C_0$ are fixed once and for all.

We recall the (Fourier based) formula for the Green's function $\Gjn$. This is the same starting point as in \cite{Thomee,Despres1,Diaconis-SaloffCoste,RSC1}.
By the standard properties of the Fourier transform, we have:
$$
\widehat{F}_{a \star a} (\theta) \, = \, \widehat{F}_a (\theta)^{\, 2} \, ,
$$
and, more generally, since $\mathcal{G}^{\, n}$ is the convolution of $a$ with itself $n-1$ times, we have:
$$
\forall \, n \in \N^* \, ,\quad \forall \, \theta \in \R \, ,\quad \widehat{F}_a (\theta)^n \, = \, \sum_{\ell \in \Z} \, \mathcal{G}_\ell^{\, n} \, {\rm e}^{\, \mathbf{i} \, \ell \, \theta} \, .
$$
We thus obtain the expression:
\begin{equation}
\label{formuleFourier}
\forall \, n \in \N^* \, ,\quad \forall \, j \in \Z \, ,\quad \Gjn \, = \,
\dfrac{1}{2\, \pi} \, \int_{-\pi}^\pi \, {\rm e}^{- \, \mathbf{i} \, j \, \theta} \, \widehat{F}_a (\theta)^n \, {\rm d}\theta
\, = \, \dfrac{1}{2\, \pi} \, \int_{-\pi}^\pi \, {\rm e}^{\, \mathbf{i} \, j \, \theta} \, \widehat{F}_a (- \, \theta)^n \, {\rm d}\theta \, ,
\end{equation}
where the final change of variables has been performed in order to stick as much as possible to the notation in \cite{hedstrom1}. From now, we use the notation:
\begin{equation}
\label{defomega}
\forall \, n \in \N^* \, ,\quad \forall \, j \in \Z \, ,\quad \omega \, := \, \dfrac{j \, - \, \alpha \, n}{n} \, ,
\end{equation}
where the real number $\alpha$ corresponds to the first coefficient in the Taylor expansion \eqref{hyp:stabilite2}. With this definition, we can rewrite
\eqref{formuleFourier} as:
\begin{equation}
\label{formuleGjn}
\forall \, n \in \N^* \, ,\quad \forall \, j \in \Z \, ,\quad \Gjn \, = \,
\dfrac{1}{2\, \pi} \, \int_{-\pi}^\pi \, {\rm e}^{\, \mathbf{i} \, n \, \omega \, \theta} \, \Big( {\rm e}^{\, \mathbf{i} \, \alpha \, \theta} \, \widehat{F}_a (- \, \theta) \Big)^n
\, {\rm d}\theta \, .
\end{equation}

Since the sequence $a$ is finitely supported, the sequence $\mathcal{G}^{\, n}$ is also finitely supported for any $n \in \N$. More precisely, there exists
a positive integer $M \in \N^*$ such that $\Gjn=0$ for any $j \in \Z$ satisfying $|\, j \, | \, > \, M \, n$. The integer $M$ measures the size of the support of
$a$. From now on, we thus only consider the case $|\, j \, | \, \le \, M \, n$. Up to choosing $M$ even larger, we can thus assume that the parameter
$\omega$ in \eqref{defomega} satisfies $|\, \omega \, | \, \le \, 2 \, M$.

The estimate of the Green's function $\Gjn$ is split in several steps, depending on the size of $\omega$. In some regimes, $\omega$ plays the role of
a small parameter in \eqref{formuleGjn}, while $n$ plays the role of a large parameter. In other regimes, $\omega$ is basically treated as a \emph{positive}
or \emph{negative} constant and the integer $n$ is the only (large) parameter in \eqref{formuleGjn}. Of course, our main attention below is to obtain estimates
where all constants are independent of $\omega$ and $n \in \N^*$, which is similar to being independent of $j \in \Z$ and $n \in \N^*$.

\subsection{The uniform bound}

In this paragraph, we recall the argument in \cite{RSC1} that yields a uniform $O(n^{-\, 1/3})$ bound for the Green's function. The result follows
from two lemmas on oscillatory integrals.

\begin{lemma}[Lemma 3.1 in \cite{RSC1}]
\label{lem1}
Let $a \le b$ be two real numbers, let $h \in \mathcal{C}^0([a,b];\C)$ and let $g \in \mathcal{C}^1([a,b];\C)$. Then there holds:
$$
\left| \, \int_a^b \, g(x) \, h(x) \, {\rm d}x \, \right| \, \le \, \left( \, \sup_{x \in [a,b]} \, \left| \, \int_a^x \, h(y) \, {\rm d}y \, \right| \, \right) \,
\Big( \, \| \, g \, \|_{L^\infty([a,b])} \, + \, \| \, g' \, \|_{L^1([a,b])} \, \Big) \, .
$$
\end{lemma}

\noindent The second result is due to van der Corput.

\begin{lemma}[van der Corput]
\label{lem2}
There exists a numerical constant $C>0$ such that for any real numbers $a \le b$, for any real valued function $f \in \mathcal{C}^3([a,b];\R)$,
there holds:
$$
\left| \, \int_a^b \, {\rm e}^{\, \mathbf{i} \, f(x)} \, {\rm d}x \, \right| \, \le \, C \,
\min_{k=1,2,3} \, \dfrac{1}{\Big( \min_{x \in [a,b]} \, |\, f^{(k)}(x) \, | \, \Big)^{\, 1/k}} \, .
$$
\end{lemma}

\noindent Following \cite{RSC1}, the combination of Lemma \ref{lem1} and Lemma \ref{lem2} yields a uniform bound for the Green's function,
as we now recall.

\begin{proposition}
\label{prop1}
Under Assumption \ref{hyp:1}, there exists a constant $C>0$ such that the Green's function in \eqref{def:green} satisfies:
$$
\forall \, n \in \N^* \, ,\quad \forall \, j \in \Z \, ,\quad \big| \, \Gjn \, \big| \, \le \, \dfrac{C}{n^{\, 1/3}} \, .
$$
\end{proposition}

\begin{proof}
With the radius $\delta_0>0$ fixed above and the constant $C_0$ in \eqref{defC0}, we choose $\delta>0$ such that $\delta \le \delta_0$ and
$$
\delta \, \le \, \pi \, , \quad \text{\rm and } \quad \delta \, C_0 \, \le \, \dfrac{c_4}{2} \, ,
$$
where $c_4>0$ is the constant associated with the $\theta^{\, 4}$ term in the Taylor expansion \eqref{hyp:stabilite2}. We then use the expression
\eqref{formuleGjn} and split the integral as:
\begin{equation}
\label{decompositionprop1-1}
 \Gjn \, = \, \varepsilon_j^n \, + \,
\dfrac{1}{2\, \pi} \, \int_{- \, \delta}^\delta \, {\rm e}^{\, \mathbf{i} \, n \, \omega \, \theta} \, \Big( {\rm e}^{\, \mathbf{i} \, \alpha \, \theta} \,
\widehat{F}_a (- \, \theta) \Big)^n \, {\rm d}\theta \, ,
\end{equation}
with:
$$
\varepsilon_j^n \, := \, \dfrac{1}{2\, \pi} \, \int_{-\pi}^{- \, \delta} \,
{\rm e}^{\, \mathbf{i} \, n \, \omega \, \theta} \, \Big( {\rm e}^{\, \mathbf{i} \, \alpha \, \theta} \, \widehat{F}_a (- \, \theta) \Big)^n \, {\rm d}\theta \, + \,
\dfrac{1}{2\, \pi} \, \int_\delta^\pi \,
{\rm e}^{\, \mathbf{i} \, n \, \omega \, \theta} \, \Big( {\rm e}^{\, \mathbf{i} \, \alpha \, \theta} \, \widehat{F}_a (- \, \theta) \Big)^n \, {\rm d}\theta \, .
$$

The estimate of $\varepsilon_j^n$ follows from the dissipation assumption \eqref{hyp:dissipation}, from which we get:
$$
| \, \varepsilon_j^n \, | \, \le \, \Big( \, \max_{\delta \le \theta \le \pi} \, \big| \, \widehat{F}_a (\theta) \, \big| \, \Big)^n \, \le \, \dfrac{C}{n^{\, 1/3}} \, ,
$$
for some suitable constant $C>0$ that is independent of $n \in \N^*$. It remains to focus on the oscillatory integral on the interval $[-\delta,\delta]$.
As in \cite{RSC1}, we use the Taylor expansion \eqref{hyp:stabilite2} (hence the restriction $\delta \le \delta_0$) and write:
$$
\dfrac{1}{2\, \pi} \, \int_{- \, \delta}^\delta \, {\rm e}^{\, \mathbf{i} \, n \, \omega \, \theta} \, \Big( {\rm e}^{\, \mathbf{i} \, \alpha \, \theta} \,
\widehat{F}_a (- \, \theta) \Big)^n \, {\rm d}\theta \, = \, \int_{- \, \delta}^\delta \, h_{n,\omega}(\theta) \, g_n(\theta) \, {\rm d}\theta \, ,
$$
with:
$$
h_{n,\omega}(\theta) \, := \, {\rm e}^{\, \mathbf{i} \, n \, (\omega \, \theta \, + \, c_3 \, \theta^{\, 3})} \, ,\quad
g_n(\theta) \, := \, \dfrac{1}{2\, \pi} \, {\rm e}^{- \, n \, c_4 \, \theta^{\, 4} \, - \, n \, \theta^{\, 5} \, \varphi(-\theta)} \, .
$$

By applying the van der Corput Lemma (Lemma \ref{lem2}), there exists a constant $C>0$ that is independent of $\omega$ and $n$ such that:
$$
\forall \, x \in \, [-\delta,\delta] \, ,\quad \left| \, \int_{- \, \delta}^x \, h_{n,\omega}(\theta) \, {\rm d}\theta \, \right| \, \le \, \dfrac{C}{n^{\, 1/3}} \, .
$$
Furthermore, with our choice for the parameter $\delta$, we have:
$$
\forall \, \theta \in [-\delta,\delta] \, ,\quad | \, g_n(\theta) \, | \, \le \, \dfrac{1}{2\, \pi} \, \exp \left(- \, n \, \dfrac{c_4}{2} \, \theta^{\, 4} \right)
\, \le \, \dfrac{1}{2\, \pi} \, ,
$$
and, differentiating the expression for $g_n(\theta)$, we also get the bound:
$$
\forall \, \theta \in [-\delta,\delta] \, ,\quad | \, g_n'(\theta) \, | \, \le \, C \, n \, | \, \theta \, |^{\, 3} \, \exp \left(- \, n \, \dfrac{c_4}{2} \, \theta^{\, 4} \right) \, ,
$$
for some uniform constant $C$. We thus obtain that the quantity:
$$
\sup_{n \in \N^*} \, \Big( \, \| \, g_n \, \|_{L^\infty([-\delta,\delta])} \, + \, \| \, g_n' \, \|_{L^1([-\delta,\delta])} \, \Big)
$$
is finite. Applying Lemma \ref{lem1}, we get the final estimate:
$$
\left| \dfrac{1}{2\, \pi} \, \int_{- \, \delta}^\delta \, {\rm e}^{\, \mathbf{i} \, n \, \omega \, \theta} \, \Big( {\rm e}^{\, \mathbf{i} \, \alpha \, \theta} \,
\widehat{F}_a (- \, \theta) \Big)^n \, {\rm d}\theta \, \right| \, \le \, \dfrac{C}{n^{\, 1/3}} \, ,
$$
with a constant $C$ that does not depend on $\omega \in \R$ nor on $n \in \N^*$. Going back to the decomposition \eqref{decompositionprop1-1}
of $\Gjn$, the claim of Proposition \ref{prop1} follows.
\end{proof}

\subsection{The fast decaying tail}

In view of Proposition \ref{prop1}, we now consider $n \in \N^*$ and $j \in \Z$ such that $j \, - \, \alpha \, n \ge n^{\, 1/3}$ or, in other words, $\omega \ge n^{- \, 2/3}$.
We are going to use a contour deformation argument in order to prove the generalized Gaussian bound stated in Theorem \ref{thm:1}. The contour is guessed by
following the so-called saddle point method, see \cite{debruijn}. To determine the location of the saddle point as well as the path direction through the saddle point,
we use the (truncated) phase function:
$$
\theta \, \longmapsto \, \mathbf{i} \, n \, \big( \, \omega \, \theta \, + \, c_3 \, \theta^{\, 3} \, \big) \, .
$$
For $\omega>0$, the two saddle points are $\pm \mathbf{i} \, \sqrt{\omega/(3\, c_3)}$ and the one with ``lowest altitude'' is $\mathbf{i} \, \sqrt{\omega/(3\, c_3)}$,
hence the choice made below in the proof of Proposition \ref{prop2}. The negative direction through this saddle point corresponds to the real axis (see Figure
\ref{fig:contour1}). We verify below that this choice of contour deformation, which is associated with the truncated phase, handles well the complete phase in
\eqref{formuleGjn} that includes the $\theta^{\, 4}$ term as well as the $O(\theta^{\, 5})$ remainder. Our first result for the regime $\omega>0$ is the following.

\begin{proposition}
\label{prop2}
Under Assumption \ref{hyp:1}, there exists $\omega_0>0$ and there exist two constants $C>0$ and $c>0$ such that the Green's function in \eqref{def:green} satisfies:
$$
\forall \, n \in \N^* \, ,\quad \forall \, j \in \Z \, ,\quad \big| \, \Gjn \, \big| \, \le \, \dfrac{C}{(j \, - \, \alpha \, n)^{\, 1/4} \, \, n^{\, 1/4}} \,
\exp \left( - \, c \, \left( \dfrac{j \, - \, \alpha \, n}{n^{\, 1/3}} \right)^{3/2} \, \right) \, ,
$$
as long as $n \in \N^*$ and the parameter $\omega$ defined in \eqref{defomega} satisfy $n^{-2/3} \le \omega \le \omega_0$ (hence $j \, - \, \alpha \, n \, > \, 0$).
\end{proposition}

\begin{proof}
In the regime considered in Proposition \ref{prop2}, the parameter $\omega$ in \eqref{formuleGjn} is positive and small (but cannot go arbitrarily close to $0$) and
the integer $n$ is thought as being large (at least large enough so that $n^{-\, 2/3} \le \omega_0$ with $\omega_0>0$ fixed as above). We start from the formula
\eqref{formuleGjn} and split again $\Gjn$ as:
$$
 \Gjn \, = \, \varepsilon_j^n \, + \,
\dfrac{1}{2\, \pi} \, \int_{- \, \delta}^\delta \, {\rm e}^{\, \mathbf{i} \, n \, \omega \, \theta} \, \Big( {\rm e}^{\, \mathbf{i} \, \alpha \, \theta} \, \widehat{F}_a (- \, \theta) \Big)^n
\, {\rm d}\theta \, ,
$$
with $\delta \in (0,\pi)$ to be fixed and:
$$
\varepsilon_j^n \, := \, \dfrac{1}{2\, \pi} \, \int_{-\pi}^{- \, \delta} \,
{\rm e}^{\, \mathbf{i} \, n \, \omega \, \theta} \, \Big( {\rm e}^{\, \mathbf{i} \, \alpha \, \theta} \, \widehat{F}_a (- \, \theta) \Big)^n \, {\rm d}\theta \, + \,
\dfrac{1}{2\, \pi} \, \int_\delta^\pi \,
{\rm e}^{\, \mathbf{i} \, n \, \omega \, \theta} \, \Big( {\rm e}^{\, \mathbf{i} \, \alpha \, \theta} \, \widehat{F}_a (- \, \theta) \Big)^n \, {\rm d}\theta \, .
$$
For reasons that will be made clear in the following lines, we choose the parameters $\delta$ and $\omega_0$ such that the following inequalities
hold\footnote{We recall that the parameter $c_3$ in \eqref{hyp:stabilite2} is assumed to be positive and that the constants $\delta_0$ and $C_0$
are determined by the remainder $\varphi$ in \eqref{hyp:stabilite2}. Hence the choice for $\delta$ and $\omega_0$ in \eqref{restrictionsprop2} is
nonempty.}:
\begin{equation}
\label{restrictionsprop2}
16 \, C_0 \, \delta \, \le \, \dfrac{c_4}{2} \, ,\quad \delta \, \le \, \delta_0 \, ,\quad \omega_0 \, \le \, 3 \, c_3 \, \delta_0^{\, 2} \, ,\quad
16 \, C_0 \, \sqrt{\dfrac{\omega_0}{3 \, c_3}} \, \le \, \dfrac{c_4}{2} \, ,\quad 12 \, c_4 \, \sqrt{\omega_0} \, \le \, (3 \, c_3)^{\, 3/2} \, .
\end{equation}

With this choice of $\delta$ (that is fixed once and for all), we use Assumption \ref{hyp:1} to write $\Gjn$ as:
\begin{equation}
\label{prop2decomposition}
\Gjn \, = \, \varepsilon_j^n \, + \,
\dfrac{1}{2\, \pi} \, \int_{- \, \delta}^\delta \, \exp \, \Big( \, \mathbf{i} \, n \, \big( \omega \, \theta \, + \, c_3 \, \theta^{\, 3} \big) \, - \, n \, c_4 \, \theta^{\, 4} \,
- \, n \, \theta^{\, 5} \, \varphi(-\theta) \, \Big) \, {\rm d}\theta \, ,
\end{equation}
with
\begin{equation}
\label{prop2estim1}
\big| \, \varepsilon_j^n \, \big| \, \le \, C \, {\rm e}^{- \, c \, n} \, ,
\end{equation}
for suitable constants $C>0$ and $c>0$. This is the same first step as in the proof of Proposition \ref{prop1}. For later use, we define:
$$
\Hjn \, := \, \dfrac{1}{2\, \pi} \, \int_{- \, \delta}^\delta \,
\exp \, \Big( \, \mathbf{i} \, n \, \big( \omega \, \theta \, + \, c_3 \, \theta^{\, 3} \big) \, - \, n \, c_4 \, \theta^{\, 4} \, - \, n \, \theta^{\, 5} \, \varphi(-\theta) \, \Big)
\, {\rm d}\theta \, ,
$$
and now focus on this term, which is the second term on the right hand side in the decomposition \eqref{prop2decomposition}. By our choice of $\delta$
and $\omega_0$ in \eqref{restrictionsprop2}, and the restriction $n^{-2/3} \le \omega \le \omega_0$, we can use the contour deformation depicted in
Figure \ref{fig:contour1}. This contour remains within the closed square $[-\delta_0,\delta_0] \times [-\delta_0,\delta_0]$ on which $\varphi$ is a holomorphic
function and we can bound its modulus by $C_0$. Applying Cauchy's formula \cite{rudin}, we thus get:
$$
\Hjn \, = \, \varepsilon_j^n(1) \, + \, \varepsilon_j^n(2) \, + \, \widetilde{\mathcal{H}}_j^{\, n} \, ,
$$
where $\varepsilon_j^n(1)$ corresponds to the integral on the left vertical segment, $\varepsilon_j^n(2)$ corresponds to the integral on the right vertical segment,
and the leading contribution $\widetilde{\mathcal{H}}_j^{\, n}$ corresponds to the integral on the horizontal segment (these contributions are depicted in red in
Figure \ref{fig:contour1}). We obtain the expressions:
\begin{multline*}
\varepsilon_j^n(1) \, = \, \int_0^{\sqrt{\frac{\omega}{3 \, c_3}}} \,
\exp \, \Big( \, \mathbf{i} \, n \, \big( \omega \, (-\delta+\mathbf{i} \, y) \, + \, c_3 \, (-\delta+\mathbf{i} \, y)^{\, 3} \big) \\
- \, n \, c_4 \, (-\delta+\mathbf{i} \, y)^{\, 4} \, - \, n \, (-\delta+\mathbf{i} \, y)^{\, 5} \, \varphi(\delta-\mathbf{i} \, y) \, \Big) \, \mathbf{i} \, {\rm d}y \, ,
\end{multline*}
$$
\varepsilon_j^n(2) \, = \, - \, \int_0^{\sqrt{\frac{\omega}{3 \, c_3}}} \,
\exp \, \Big( \, \mathbf{i} \, n \, \big( \omega \, (\delta+\mathbf{i} \, y) \, + \, c_3 \, (\delta+\mathbf{i} \, y)^{\, 3} \big)
\, - \, n \, c_4 \, (\delta+\mathbf{i} \, y)^{\, 4} \, - \, n \, (\delta+\mathbf{i} \, y)^{\, 5} \, \varphi(-\delta-\mathbf{i} \, y) \, \Big) \, \mathbf{i} \, {\rm d}y \, ,
$$
and
\begin{multline}
\label{prop2Htildejn}
\widetilde{\mathcal{H}}_j^{\, n} \, = \, \dfrac{1}{2\, \pi} \, \int_{- \, \delta}^\delta \, \exp \, \Big( \, \mathbf{i} \, n \, \Big( \omega \, \left( \mathbf{i} \, \sqrt{\frac{\omega}{3 \, c_3}}
+\theta \right) \, + \, c_3 \, \left( \mathbf{i} \, \sqrt{\frac{\omega}{3 \, c_3}} +\theta \right)^{\, 3} \Big) \\
- \, n \, c_4 \, \left( \mathbf{i} \, \sqrt{\frac{\omega}{3 \, c_3}} +\theta \right)^{\, 4} \, - \, n \, \left( \mathbf{i} \, \sqrt{\frac{\omega}{3 \, c_3}} +\theta \right)^{\, 5} \,
\varphi \left( -\, \mathbf{i} \, \sqrt{\frac{\omega}{3 \, c_3}} -\theta \right) \, \Big) \, {\rm d}\theta \, .
\end{multline}

\begin{figure}[h!]
\begin{center}
\begin{tikzpicture}[scale=1.5,>=latex]
\draw[black,->] (-4,0) -- (4,0);
\draw[black,->] (0,-0.5)--(0,3);
\draw[thick,blue,->] (-3,0) -- (-3,1);
\draw[thick,blue] (-3,1) -- (-3,2);
\draw[thick,blue,->] (-3,2) -- (1,2);
\draw[thick,blue] (1,2) -- (3,2);
\draw[thick,blue,->] (3,2) -- (3,0.9);
\draw[thick,blue] (3,0.9) -- (3,0);
\draw (-3,-0.1) node[below]{$-\delta$};
\draw (3,-0.1) node[below]{$\delta$};
\draw (0.1,0) node[below]{$0$};
\draw (-3.5,1.5) node[above]{{\color{red}$\varepsilon_j^n(1)$}};
\draw (3.6,1.5) node[above]{{\color{red}$\varepsilon_j^n(2)$}};
\draw (-1.5,2.5) node[above]{{\color{red}$\widetilde{\mathcal{H}}_j^{\, n}$}};
\draw[thick,red,->] (-3.05,1.1) arc (270:180:0.5) ;
\draw[thick,red,->] (3.05,1.1) arc (270:360:0.5) ;
\draw[thick,red,->] (-0.9,2.1) -- (-1.32,2.65);
\draw (0,1.5) node[right]{$\mathbf{i} \, \sqrt{\dfrac{\omega}{3 \, c_3}}$};
\draw (3.5,2.5) node {$\C$};
\node[red] (centre) at (0,2){$\bullet$};
\node (centre) at (-3,0){$\bullet$};
\node (centre) at (-3,2){$\bullet$};
\node (centre) at (3,0){$\bullet$};
\node (centre) at (3,2){$\bullet$};
\end{tikzpicture}
\caption{The integration contour in the case $n^{-2/3} \le \omega \le \omega_0$. The (approximate) saddle point is represented
with a red bullet. The black bullets represent the end points of the three segments along which we compute the integrals $\varepsilon_j^n(1)$,
$\varepsilon_j^n(2)$ and $\widetilde{\mathcal{H}}_j^{\, n}$.}
\label{fig:contour1}
\end{center}
\end{figure}
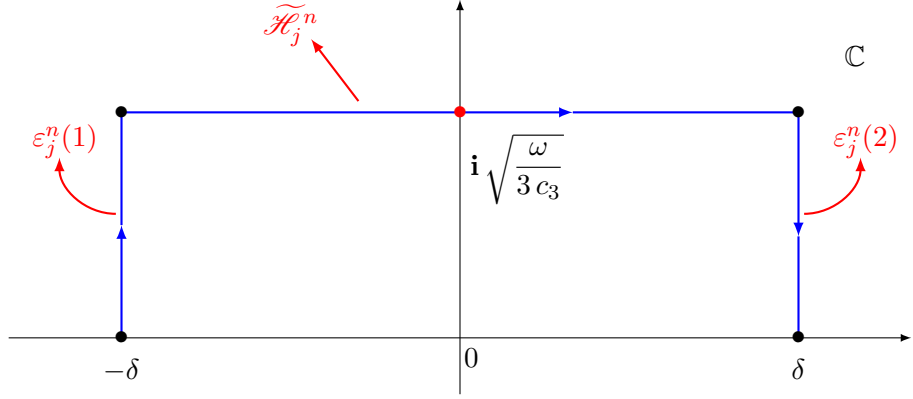

Let us start with $\varepsilon_j^n(1)$. From now on, when we write $\exp ( \, \mathbf{i} \, \cdots)$, the dots always stand for a real number whose expression is
useless since in the end a modulus of this expression will be simply estimated by $1$. Going back to the defining expression for $\varepsilon_j^n(1)$, we expand
the various expressions within the integral and compute:
$$
\varepsilon_j^n(1) \, = \, {\rm e}^{- \, n \, c_4 \, \delta^{\, 4}} \, \int_0^{\sqrt{\frac{\omega}{3 \, c_3}}} \, {\rm e}^{\, \mathbf{i} \, \cdots} \, \,
{\rm e}^{- \, n \, \Big( (\omega \, + \, 3 \, c_3 \, \delta^{\, 2}) \, y \, - \, 6 \, c_4 \, \delta^{\, 2} \, y^{\, 2} \, - \, c_3 \, y^{\, 3} \, + \, c_4 \, y^{\, 4} \Big)} \,
{\rm e}^{\, - \, n \, (-\delta+\mathbf{i} \, y)^{\, 5} \, \varphi(\delta-\mathbf{i} \, y)} \, {\rm d}y \, .
$$
We take the modulus of each side of the equality and apply the triangle inequality to get:
$$
\big| \, \varepsilon_j^n(1) \, \big| \, \le \, {\rm e}^{- \, n \, c_4 \, \delta^{\, 4}} \, \int_0^{\sqrt{\frac{\omega}{3 \, c_3}}} \,
{\rm e}^{- \, n \, \Big( (\omega \, + \, 3 \, c_3 \, \delta^{\, 2}) \, y \, - \, 6 \, c_4 \, \delta^{\, 2} \, y^{\, 2} \, - \, c_3 \, y^{\, 3} \, + \, c_4 \, y^{\, 4} \Big)} \,
{\rm e}^{\, n \, C_0 \, |-\delta+\mathbf{i} \, y \, |^{\, 5}} \, {\rm d}y \, ,
$$
where we used the fact that the integration contour is located within the region where the modulus of $\varphi$ is less than $C_0$. We now apply the H\"older
inequality in $\C^{\, 2}$ to get ($y$ is nonnegative here):
$$
\big| \, - \, \delta \, + \, \mathbf{i} \, y \, \big|^{\, 5} \, \le \, 16 \, \big( \, \delta^{\, 5} \, + \, y^{\, 5} \, \big) \, ,
$$
which gives:
$$
\big| \, \varepsilon_j^n(1) \, \big| \, \le \, {\rm e}^{- \, n \, c_4 \, \delta^{\, 4}} \, {\rm e}^{\, 16 \, n \, C_0 \, \delta^{\, 5}} \, \int_0^{\sqrt{\frac{\omega}{3 \, c_3}}} \,
{\rm e}^{- \, n \, \Big( (\omega \, + \, 3 \, c_3 \, \delta^{\, 2}) \, y \, - \, 6 \, c_4 \, \delta^{\, 2} \, y^{\, 2} \, - \, c_3 \, y^{\, 3} \, + \, c_4 \, y^{\, 4} \Big)} \,
{\rm e}^{\, 16 \, n \, C_0 \, y^{\, 5}} \, {\rm d}y \, .
$$
The restrictions (see \eqref{restrictionsprop2}):
$$
16 \, C_0 \, \delta \, \le \, \dfrac{c_4}{2} \, ,\qquad \omega \, \le \, \omega_0 \, ,\qquad 16 \, C_0 \, \sqrt{\dfrac{\omega_0}{3 \, c_3}} \, \le \, \dfrac{c_4}{2} \, ,
$$
imply that the terms with $\delta^{\, 5}$ and $y^{\, 5}$ can be absorbed by half the ones with the power $4$, namely:
\begin{align*}
\big| \, \varepsilon_j^n(1) \, \big| \, &\le \, {\rm e}^{- \, n \, \frac{c_4}{2} \, \delta^{\, 4}} \, \int_0^{\sqrt{\frac{\omega}{3 \, c_3}}} \,
{\rm e}^{- \, n \, \Big( (\omega \, + \, 3 \, c_3 \, \delta^{\, 2}) \, y \, - \, 6 \, c_4 \, \delta^{\, 2} \, y^{\, 2} \, - \, c_3 \, y^{\, 3} \, + \, \frac{c_4}{2} \, y^{\, 4} \Big)} \, {\rm d}y \\
&\le \, {\rm e}^{- \, n \, \frac{c_4}{2} \, \delta^{\, 4}} \, \int_0^{\sqrt{\frac{\omega}{3 \, c_3}}} \,
{\rm e}^{- \, n \, \Big( (\omega \, + \, 3 \, c_3 \, \delta^{\, 2}) \, y \, - \, 6 \, c_4 \, \delta^{\, 2} \, y^{\, 2} \, - \, c_3 \, y^{\, 3} \Big)} \, {\rm d}y \, .
\end{align*}
Let us now note that on the interval $[0,\sqrt{\omega/(3 \, c_3)}]$, we have:
$$
- \, c_3 \, y^{\, 3} \, \ge \, - \, \dfrac{\omega}{3} \, y \, ,
$$
and we also use the last inequality in \eqref{restrictionsprop2} to get:
\begin{equation}
\label{prop2estim2}
\big| \, \varepsilon_j^n(1) \, \big| \, \le \, {\rm e}^{- \, n \, \frac{c_4}{2} \, \delta^{\, 4}} \, \int_0^{\sqrt{\frac{\omega}{3 \, c_3}}} \,
{\rm e}^{- \, n \, (\frac{2 \, \omega}{3} \, + \, \frac{3}{2} \, c_3 \, \delta^{\, 2}) \, y} \, {\rm d}y \, \le \,
{\rm e}^{- \, n \, \frac{c_4}{2} \, \delta^{\, 4}} \, \int_0^{+ \, \infty} \, {\rm e}^{- \, n \, \frac{3}{2} \, c_3 \, \delta^{\, 2} \, y} \, {\rm d}y \, \le \,
C \, {\rm e}^{- \, c \, n} \, ,
\end{equation}
for suitable constants $C>0$ and $c>0$.

The estimate of the integral $\varepsilon_j^n(2)$ along the right vertical segment is entirely similar. At this stage, we can collect \eqref{prop2estim1} and
\eqref{prop2estim2} to show that for $n^{\, -2/3} \le \omega \le \omega_0$, the Green's function $\Gjn$ satisfies:
\begin{equation}
\label{prop2estim3}
\Big| \, \Gjn \, - \, \widetilde{\mathcal{H}}_j^{\, n} \, \Big| \, \le \, C \, {\rm e}^{- \, c \, n} \, ,
\end{equation}
where the expression of the (presumably leading) contribution $\widetilde{\mathcal{H}}_j^{\, n}$ is given in \eqref{prop2Htildejn}. Let us therefore turn to the
study of $\widetilde{\mathcal{H}}_j^{\, n}$.
\bigskip

We expand the expressions within the integral on the right hand side of \eqref{prop2Htildejn} and obtain:
\begin{multline*}
\widetilde{\mathcal{H}}_j^{\, n} \, = \, \dfrac{{\rm e}^{- \, \frac{2}{3 \, \sqrt{3 \, c_3}} \, n \, \omega^{\, 3/2}} \, {\rm e}^{- \, \frac{c_4}{9 \, c_3^{\, 2}} \, n \, \omega^{\, 2}}}{2 \, \pi}
\, \int_{- \, \delta}^\delta  \, {\rm e}^{\, \mathbf{i} \, \cdots} \, {\rm e}^{- \, n \, \Big( \sqrt{3 \, c_3 \, \omega} \, - \, \frac{2 \, c_4}{c_3} \, \omega \Big) \, \theta^{\, 2}} \,
{\rm e}^{- \, n \, c_4 \, \theta^{\, 4}} \, \times \\
\exp \, \left( - \, n \, \left( \mathbf{i} \, \sqrt{\frac{\omega}{3 \, c_3}} +\theta \right)^{\, 5} \, \varphi \left( -\mathbf{i} \, \sqrt{\frac{\omega}{3 \, c_3}} -\theta \right) \, \right)
\, {\rm d}\theta \, .
\end{multline*}
We take the modulus on each side of the equality and apply the triangle inequality to get (the same H\"older inequality as above is used to deal with the
remainder term on the second line):
\begin{multline*}
\Big| \, \widetilde{\mathcal{H}}_j^{\, n} \, \Big| \, \le \,
\dfrac{{\rm e}^{- \, \frac{2}{3 \, \sqrt{3 \, c_3}} \, n \, \omega^{\, 3/2}} \, {\rm e}^{- \, \frac{c_4}{9 \, c_3^{\, 2}} \, n \, \omega^{\, 2}}}{2 \, \pi} \, \int_{- \, \delta}^\delta \,
{\rm e}^{- \, n \, \Big( \sqrt{3 \, c_3 \, \omega} \, - \, \frac{2 \, c_4}{c_3} \, \omega \Big) \, \theta^{\, 2}} \, {\rm e}^{- \, n \, c_4 \, \theta^{\, 4}} \, \times \\
\exp \, \left( 16 \, n \, C_0 \, \left( \frac{\omega^{\, 5/2}}{(3 \, c_3)^{\, 5/2}}  \, + \, | \, \theta \, |^{\, 5} \right) \, \right) \, {\rm d}\theta \, .
\end{multline*}
Again, our restrictions on $\omega_0$ and $\delta$ in \eqref{restrictionsprop2} imply that the final remainder terms can be absorbed by half of some
already arising with a ``good'' sign, and we get the estimate:
\begin{align*}
\Big| \, \widetilde{\mathcal{H}}_j^{\, n} \, \Big| \, & \le \,
\dfrac{{\rm e}^{- \, \frac{2}{3 \, \sqrt{3 \, c_3}} \, n \, \omega^{\, 3/2}} \, {\rm e}^{- \, \frac{c_4}{18 \, c_3^{\, 2}} \, n \, \omega^{\, 2}}}{2 \, \pi}
\, \int_{- \, \delta}^\delta \, {\rm e}^{- \, n \, \Big( \sqrt{3 \, c_3 \, \omega} \, - \, \frac{2 \, c_4}{c_3} \, \omega \Big) \, \theta^{\, 2}} \,
{\rm e}^{- \, n \, \frac{c_4}{2} \, \theta^{\, 4}} \, {\rm d}\theta \\
& \le \, \dfrac{{\rm e}^{- \, \frac{2}{3 \, \sqrt{3 \, c_3}} \, n \, \omega^{\, 3/2}}}{2 \, \pi}
\, \int_{- \, \delta}^\delta \, {\rm e}^{- \, n \, \Big( \sqrt{3 \, c_3 \, \omega} \, - \, \frac{2 \, c_4}{c_3} \, \omega \Big) \, \theta^{\, 2}} \, {\rm d}\theta \, .
\end{align*}
At last, we use the bound from above (see \eqref{restrictionsprop2}):
$$
\dfrac{2 \, c_4}{c_3} \, \omega \, \le \, \sqrt{\omega} \, \, \dfrac{2 \, c_4}{c_3} \, \sqrt{\omega_0} \, \le \, \dfrac{\sqrt{3 \, c_3 \, \omega}}{2} \, ,
$$
to get:
$$
\Big| \, \widetilde{\mathcal{H}}_j^{\, n} \, \Big| \, \le \, \dfrac{{\rm e}^{- \, \frac{2}{3 \, \sqrt{3 \, c_3}} \, n \, \omega^{\, 3/2}}}{2 \, \pi} \,
\int_{- \, \delta}^\delta \, {\rm e}^{- \, n \, \frac{\sqrt{3 \, c_3 \, \omega}}{2} \, \theta^{\, 2}} \, {\rm d}\theta
 \, \le \, \dfrac{{\rm e}^{- \, \frac{2}{3 \, \sqrt{3 \, c_3}} \, n \, \omega^{\, 3/2}}}{2 \, \pi} \,
\int_\R \, {\rm e}^{- \, n \, \frac{\sqrt{3 \, c_3 \, \omega}}{2} \, \theta^{\, 2}} \, {\rm d}\theta \, ,
$$
and we therefore end up with our final estimate:
\begin{equation}
\label{prop2estim4}
\Big| \, \widetilde{\mathcal{H}}_j^{\, n} \, \Big| \, \le \, \dfrac{C}{n^{\, 1/2} \, \omega^{\, 1/4}} \, \exp \Big( - \, c \, n \, \omega^{\, 3/2} \Big) \, .
\end{equation}

We now combine \eqref{prop2estim3} and \eqref{prop2estim4} to get:
$$
\big| \, \Gjn \, \big| \, \le \, C \, \exp (- \, c \, n) \, + \, \dfrac{C}{(j \, - \, \alpha \, n)^{\, 1/4} \, n^{\, 1/4}} \,
\exp \left( - \, c \, \left( \dfrac{j \, - \, \alpha \, n}{n^{\, 1/3}} \right)^{3/2} \, \right) \, ,
$$
as long as $j$ and $n$ satisfy $n^{\, -2/3} \le \omega \le \omega_0$. For such integers, it can be easily seen that the leading contribution on the right hand
side of this last inequality is the second one. Namely, given $C$, $c$ and $\omega_0$ positive, we can always find other positive constants $C'$ and $c'$
such that for $0 \, < \, j \, - \alpha \, n \, \le \, \omega_0 \, n$, there holds:
$$
C \, \exp (- \, c \, n) \, \le \, \dfrac{C'}{(j \, - \, \alpha \, n)^{\, 1/4} \, n^{\, 1/4}} \, \exp \left( - \, c' \, \left( \dfrac{j \, - \, \alpha \, n}{n^{\, 1/3}} \right)^{3/2} \, \right) \, .
$$
We have thus obtained the estimate of the Green's function $\Gjn$ as claimed in Proposition \ref{prop2}.
\end{proof}

\noindent The final case to deal with in this section is when the parameter $\omega$ belongs to the interval $[\omega_0,2\, M]$.

\begin{proposition}
\label{prop3}
Under Assumption \ref{hyp:1}, with the same $\omega_0>0$ as in Proposition \ref{prop2}, there exist two constants $C>0$ and $c>0$ such that the Green's
function in \eqref{def:green} satisfies:
$$
\forall \, n \in \N^* \, ,\quad \forall \, j \in \Z \, ,\quad \big| \, \Gjn \, \big| \, \le \, \dfrac{C}{(j \, - \, \alpha \, n)^{\, 1/4} \, \, n^{\, 1/4}} \,
\exp \left( - \, c \, \left( \dfrac{j \, - \, \alpha \, n}{n^{\, 1/3}} \right)^{3/2} \, \right) \, ,
$$
as long as $n \in \N^*$ and the parameter $\omega$ defined in \eqref{defomega} satisfy $\omega_0 \le \omega \le 2 \, M$ (hence $j \, - \, \alpha \, n \, > \, 0$).
\end{proposition}

\begin{proof}
The proof follows similar lines as that of Proposition \ref{prop2}. The difference is that we can no longer choose the same contour as in Figure \ref{fig:contour1}
since for too large values of $\omega$, there is no reason why the contour would remain within the holomorphy region of $\varphi$. Nevertheless, we still fix
the parameters $\delta$ and $\omega_0$ as in \eqref{restrictionsprop2} and decompose $\Gjn$ as in \eqref{prop2decomposition}, where the remainder
$\varepsilon_j^n$ is uniformly exponentially small, see \eqref{prop2estim1}. We now use the contour depicted in Figure \ref{fig:contour2} where the ``height''
has been kept fixed equal to $\sqrt{\omega_0/(3\, c_3)}$, independently of $\omega$, in order to remain within the region where $\varphi$ is holomorphic
and bounded by $C_0$.

\begin{figure}[h!]
\begin{center}
\begin{tikzpicture}[scale=1.5,>=latex]
\draw[black,->] (-4,0) -- (4,0);
\draw[black,->] (0,-0.5)--(0,3);
\draw[thick,blue,->] (-3,0) -- (-3,1);
\draw[thick,blue] (-3,1) -- (-3,2);
\draw[thick,blue,->] (-3,2) -- (1,2);
\draw[thick,blue] (1,2) -- (3,2);
\draw[thick,blue,->] (3,2) -- (3,0.8);
\draw[thick,blue] (3,0.8) -- (3,0);
\draw (-3,-0.1) node[below]{$-\delta$};
\draw (3,-0.1) node[below]{$\delta$};
\draw (0.1,0) node[below]{$0$};
\draw (-3.5,1.5) node[above]{{\color{red}$\varepsilon_j^n(1)$}};
\draw (3.6,1.5) node[above]{{\color{red}$\varepsilon_j^n(2)$}};
\draw (-1.5,2.5) node[above]{{\color{red}$\widetilde{\mathcal{H}}_j^{\, n}$}};
\draw[thick,red,->] (-3.05,1.1) arc (270:180:0.5) ;
\draw[thick,red,->] (3.05,1.1) arc (270:360:0.5) ;
\draw[thick,red,->] (-0.9,2.1) -- (-1.32,2.65);
\draw (0,1.5) node[right]{$\mathbf{i} \, \sqrt{\dfrac{\omega_0}{3 \, c_3}}$};
\draw (3.5,2.5) node {$\C$};
\node[red] (centre) at (0,2){$\bullet$};
\node (centre) at (-3,0){$\bullet$};
\node (centre) at (-3,2){$\bullet$};
\node (centre) at (3,0){$\bullet$};
\node (centre) at (3,2){$\bullet$};
\end{tikzpicture}
\caption{The integration contour in the case $\omega_0 \le \omega \le 2 \, M$. The (approximate) saddle point is represented
with a red bullet. The black bullets represent the end points of the three segments along which we compute the integrals}
\label{fig:contour2}
\end{center}
\end{figure}
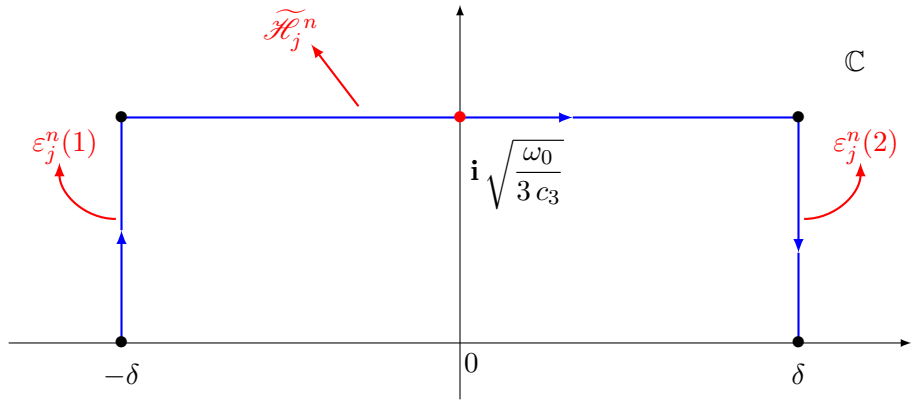

Keeping the same notation as in the proof of Proposition \ref{prop2} (see Figure \ref{fig:contour2}), we have thus decomposed the Green's function $\Gjn$ as:
$$
\Gjn \, = \, \varepsilon_j^n \, + \, \varepsilon_j^n(1) \, + \, \varepsilon_j^n(2) \, + \, \widetilde{\mathcal{H}}_j^{\, n} \, ,
$$
where the first remainder $\varepsilon_j^n$ is estimated as in \eqref{prop2estim1}, and the other terms are given by\footnote{For $\varepsilon_j^n(1)$ and
$\varepsilon_j^n(2)$, the difference with the proof of Proposition \ref{prop2} is in the upper bound of the interval over which we integrate.}:
\begin{multline*}
\varepsilon_j^n(1) \, = \, \int_0^{\sqrt{\frac{\omega_0}{3 \, c_3}}} \,
\exp \, \Big( \, \mathbf{i} \, n \, \big( \omega \, (-\delta+\mathbf{i} \, y) \, + \, c_3 \, (-\delta+\mathbf{i} \, y)^{\, 3} \big) \\
- \, n \, c_4 \, (-\delta+\mathbf{i} \, y)^{\, 4} \, - \, n \, (-\delta+\mathbf{i} \, y)^{\, 5} \, \varphi(\delta-\mathbf{i} \, y) \, \Big) \, \mathbf{i} \, {\rm d}y \, ,
\end{multline*}
$$
\varepsilon_j^n(2) \, = \, - \, \int_0^{\sqrt{\frac{\omega_0}{3 \, c_3}}} \,
\exp \, \Big( \, \mathbf{i} \, n \, \big( \omega \, (\delta+\mathbf{i} \, y) \, + \, c_3 \, (\delta+\mathbf{i} \, y)^{\, 3} \big)
\, - \, n \, c_4 \, (\delta+\mathbf{i} \, y)^{\, 4} \, - \, n \, (\delta+\mathbf{i} \, y)^{\, 5} \, \varphi(-\delta-\mathbf{i} \, y) \, \Big) \, \mathbf{i} \, {\rm d}y \, ,
$$
and
\begin{multline}
\label{prop3Htildejn}
\widetilde{\mathcal{H}}_j^{\, n} \, = \, \dfrac{1}{2\, \pi} \, \int_{- \, \delta}^\delta \, \exp \, \Big( \, \mathbf{i} \, n \, \Big( \omega \, \left( \mathbf{i} \, \sqrt{\frac{\omega_0}{3 \, c_3}}
+\theta \right) \, + \, c_3 \, \left( \mathbf{i} \, \sqrt{\frac{\omega_0}{3 \, c_3}} +\theta \right)^{\, 3} \Big) \\
- \, n \, c_4 \, \left( \mathbf{i} \, \sqrt{\frac{\omega_0}{3 \, c_3}} +\theta \right)^{\, 4} \, - \, n \, \left( \mathbf{i} \, \sqrt{\frac{\omega_0}{3 \, c_3}} +\theta \right)^{\, 5} \,
\varphi \left( -\, \mathbf{i} \, \sqrt{\frac{\omega_0}{3 \, c_3}} -\theta \right) \, \Big) \, {\rm d}\theta \, .
\end{multline}

Let us start with the remainder term $\varepsilon_j^n(1)$. As in the proof of Proposition \ref{prop2}, we expand the quantities within the integral and then
apply the triangle inequality to get:
$$
\big| \, \varepsilon_j^n(1) \, \big| \, \le \, {\rm e}^{- \, n \, c_4 \, \delta^{\, 4}} \, {\rm e}^{\, 16 \, n \, C_0 \, \delta^{\, 5}} \, \int_0^{\sqrt{\frac{\omega_0}{3 \, c_3}}} \,
{\rm e}^{- \, n \, \Big( (\omega \, + \, 3 \, c_3 \, \delta^{\, 2}) \, y \, - \, 6 \, c_4 \, \delta^{\, 2} \, y^{\, 2} \, - \, c_3 \, y^{\, 3} \, + \, c_4 \, y^{\, 4} \Big)} \,
{\rm e}^{\, 16 \, n \, C_0 \, y^{\, 5}} \, {\rm d}y \, .
$$
By using the restrictions \eqref{restrictionsprop2} on $\delta$ and $\omega_0$, we can still absorb the $\delta^{\, 5}$ and $y^{\, 5}$ terms to get:
\begin{align*}
\big| \, \varepsilon_j^n(1) \, \big| \, & \le \, {\rm e}^{- \, n \, \frac{c_4}{2} \, \delta^{\, 4}} \, \int_0^{\sqrt{\frac{\omega_0}{3 \, c_3}}} \,
{\rm e}^{- \, n \, \Big( (\omega \, + \, 3 \, c_3 \, \delta^{\, 2}) \, y \, - \, 6 \, c_4 \, \delta^{\, 2} \, y^{\, 2} \, - \, c_3 \, y^{\, 3} \, + \, \frac{c_4}{2} \, y^{\, 4} \Big)} \, {\rm d}y \\
& \le \, {\rm e}^{- \, n \, \frac{c_4}{2} \, \delta^{\, 4}} \, \int_0^{\sqrt{\frac{\omega_0}{3 \, c_3}}} \,
{\rm e}^{- \, n \, \Big( (\omega \, + \, 3 \, c_3 \, \delta^{\, 2}) \, y \, - \, 6 \, c_4 \, \delta^{\, 2} \, \sqrt{\frac{\omega_0}{3 \, c_3}} \, y \, - \, \frac{\omega_0}{3} \, y \Big)}
\, {\rm d}y \, .
\end{align*}
By using again \eqref{restrictionsprop2} and $\omega \ge \omega_0 \ge \omega_0/3$, we end up with:
$$
\big| \, \varepsilon_j^n(1) \, \big| \, \le \, {\rm e}^{- \, n \, \frac{c_4}{2} \, \delta^{\, 4}} \, \int_0^{\sqrt{\frac{\omega_0}{3 \, c_3}}} \,
{\rm e}^{- \, n \, \frac{3}{2} \, c_3 \, \delta^{\, 2} \, y} \, {\rm d}y \, \le \, C \, {\rm e}^{- \, c \, n} \, .
$$
The estimate of the other remainder term $\varepsilon_j^n(2)$ is similar and we still get:
\begin{equation}
\label{prop3estim1}
\Big| \, \Gjn \, - \, \widetilde{\mathcal{H}}_j^{\, n} \, \Big| \, \le \, C \, {\rm e}^{- \, c \, n} \, ,
\end{equation}
where $\widetilde{\mathcal{H}}_j^{\, n}$ is now given by \eqref{prop3Htildejn}.

We expand the various terms in \eqref{prop3Htildejn} and obtain the expression:
\begin{multline*}
\widetilde{\mathcal{H}}_j^{\, n} \, = \,
\dfrac{{\rm e}^{- \, n \, \left( \omega - \frac{\omega_0}{3} \right) \, \sqrt{\frac{\omega_0}{3 \, \, c_3}}} \, {\rm e}^{- \, \frac{c_4}{9 \, c_3^{\, 2}} \, n \, \omega_0^{\, 2}}}{2 \, \pi}
\, \int_{- \, \delta}^\delta \, {\rm e}^{\, \mathbf{i} \, \cdots} \, {\rm e}^{- \, n \, \Big( \sqrt{3 \, c_3 \, \omega_0} \, - \, \frac{2 \, c_4}{c_3} \, \omega_0 \Big) \, \theta^{\, 2}} \,
{\rm e}^{- \, n \, c_4 \, \theta^{\, 4}} \, \times \\
\exp \, \left( - \, n \, \left( \mathbf{i} \, \sqrt{\frac{\omega_0}{3 \, c_3}} +\theta \right)^{\, 5} \, \varphi \left( -\mathbf{i} \, \sqrt{\frac{\omega_0}{3 \, c_3}} -\theta \right) \, \right)
\, {\rm d}\theta \, .
\end{multline*}
We take the modulus on each side of the inequality and apply the already used H\"older inequality to absorb the final remainder, which yields:
$$
\Big| \, \widetilde{\mathcal{H}}_j^{\, n} \, \Big| \, \le \, \dfrac{{\rm e}^{- \, n \, \left( \omega - \frac{\omega_0}{3} \right) \, \sqrt{\frac{\omega_0}{3 \, \, c_3}}}}{2 \, \pi}
\, \int_{- \, \delta}^\delta \, {\rm e}^{- \, n \, \Big( \sqrt{3 \, c_3 \, \omega_0} \, - \, \frac{2 \, c_4}{c_3} \, \omega_0 \Big) \, \theta^{\, 2}} \, {\rm d}\theta \, .
$$
We now use $\omega \ge \omega_0$ as well as the restriction \eqref{restrictionsprop2} on $\omega_0$ to get:
$$
\Big| \, \widetilde{\mathcal{H}}_j^{\, n} \, \Big| \, \le \, \dfrac{C}{n^{\, 1/2} \, \omega_0^{\, 1/4}} \, \exp \Big( - \, c \, n \, \omega_0^{\, 3/2} \Big) \, ,
$$
where we recall that $\omega_0>0$ has been fixed. Combining with \eqref{prop3estim1}, we have thus obtained the uniform exponential bound:
$$
\Big| \, \Gjn \, \Big| \, \le \, C \, {\rm e}^{- \, c \, n} \, ,
$$
for $\omega \ge \omega_0$, and we can convert this bound into:
$$
\Big| \, \Gjn \, \Big| \, \le \, \dfrac{C}{n^{\, 1/2} \, \omega^{\, 1/4}} \, \exp \Big( - \, c \, n \, \omega^{\, 3/2} \Big) \, ,
$$
for $\omega \in [\omega_0,2\, M]$ where $\omega_0$ and $M$ have already been fixed. Going back to the definition of $\omega$, this gives the result
of Proposition \ref{prop3}.
\end{proof}

\subsection{Conclusion}

In this short paragraph, we explain why the above preliminary results imply the validity of \eqref{bound1}. If $n \ge 1$ and $n^{\, 1/3} \le j \, - \, \alpha \, n
\le 2 \, M \, n$, we use Propositions \ref{prop2} and \ref{prop3} to obtain the existence of positive constants $C_\sharp$ and $c_\sharp$ (independent of
$j$ and $n$) such that:
\begin{align*}
\big| \, \Gjn \, \big| \, &\le \, \dfrac{C_\sharp}{(j \, - \, \alpha \, n)^{\, 1/4} \, \, n^{\, 1/4}} \,
\exp \left( - \, c_\sharp \, \left( \dfrac{j \, - \, \alpha \, n}{n^{\, 1/3}} \right)^{3/2} \, \right) \\
&\le \, \dfrac{C_\sharp}{n^{\, 1/3}} \, \left( \dfrac{j \, - \, \alpha \, n}{n^{\, 1/3}} \right)^{\, - \, 1/4} \,
\exp \left( - \, c_\sharp \, \left( \dfrac{j \, - \, \alpha \, n}{n^{\, 1/3}} \right)^{3/2} \, \right) \, .
\end{align*}
This proves the validity of \eqref{bound1} for $n^{\, 1/3} \le j \, - \, \alpha \, n \le 2 \, M \, n$. The constants $C_\sharp$ and $c_\sharp$ are now fixed.

For $2 \, M \, n < j \, - \, \alpha \, n$, the validity of \eqref{bound1} is even more clear since $\Gjn$ is zero. It therefore remains to treat the case
$0 \le j \, - \, \alpha \, n \le n^{\, 1/3}$. We use Proposition \ref{prop1} to obtain:
$$
| \, \Gjn \, | \, \le \, \dfrac{C_\flat}{n^{\, 1/3}} \, ,
$$
for some other constant $C_\flat$ (possibly larger than the above constant $C_\sharp$), and \eqref{bound1} follows by using the inequalities:
$$
\dfrac{C_\flat}{n^{\, 1/3}} \, \le \, \dfrac{C_\flat \, {\rm e}^{\, c_\sharp}}{n^{\, 1/3}} \, {\rm e}^{\, - \, c_\sharp}
\, \le \, \dfrac{C_\flat \, {\rm e}^{\, c_\sharp}}{n^{\, 1/3}} \, \exp \left( - \, c_\sharp \, \left( \dfrac{j \, - \, \alpha \, n}{n^{\, 1/3}} \right)^{3/2} \, \right) \, ,
$$
for  $0 \le j \, - \, \alpha \, n \le n^{\, 1/3}$. The bound \eqref{bound1} follows by choosing $c \, := \, c_\sharp$ and $C \, := \, \max (C_\sharp,
C_\flat \, {\rm e}^{\, c_\sharp})$. We now turn to the case $j \, - \, \alpha \, n < 0$ which is where the oscillations in the Green's function will arise.

\section{Proof of the main result. II. The oscillations}
\label{section3}

\subsection{The oscillations}

We are now interested in the regime $\omega<0$ and start with the case where $\omega$ is small. This is the most difficult region where the Green's
function exhibits oscillations. This is also the unique region that is the cause for the $\ell^{\, \infty}$ instability phenomenon which we have recalled in
the introduction. As in Section \ref{section2}, the regime $-n^{-\, 2/3} \le \omega \le 0$ will de dealt with by Proposition \ref{prop1}, so we consider from
now on $\omega \le -n^{-\, 2/3}$, that is $j \, - \, \alpha \, n \le - \, n^{\, 1/3}$.

\begin{proposition}
\label{prop4}
Under Assumption \ref{hyp:1}, there exist $\omega_0>0$ and there exist two constants $C>0$ and $c>0$ such that for any $j \in \Z$ and $n \in \N^*$,
with $\GGjn$ defined as in \eqref{defprincipal}, there holds:
\begin{equation}
\label{estimprop4}
\big| \, \Gjn \, - \, \GGjn \, \big| \, \le \, \dfrac{C}{|j \, - \, \alpha \, n|} \, \exp \left( - \, c \, \left( \dfrac{|j \, - \, \alpha \, n|}{n^{\, 1/3}} \right)^{\, 3/2} \, \right) \, ,
\end{equation}
as long as $n \in \N^*$ and the parameter $\omega$ defined in \eqref{defomega} satisfy $-\omega_0 \le \omega \le -n^{\, -2/3}$ (hence $j \, - \, \alpha \, n \, < \, 0$).
\end{proposition}

\begin{proof}
Following the proof of Proposition \ref{prop2}, we introduce a parameter $\delta  \in (0,\pi)$ to be fixed later on and split the quantity $\Gjn$ as:
$$
 \Gjn \, = \, \varepsilon_j^n \, + \,
\dfrac{1}{2\, \pi} \, \int_{- \, \delta}^\delta \,
{\rm e}^{\, \mathbf{i} \, n \, \omega \, \theta} \, \Big( {\rm e}^{\, \mathbf{i} \, \alpha \, \theta} \, \widehat{F}_a (- \, \theta) \Big)^n \, {\rm d}\theta \, ,
$$
with:
$$
\varepsilon_j^n \, := \, \dfrac{1}{2\, \pi} \, \int_{-\pi}^{- \, \delta} \,
{\rm e}^{\, \mathbf{i} \, n \, \omega \, \theta} \, \Big( {\rm e}^{\, \mathbf{i} \, \alpha \, \theta} \, \widehat{F}_a (- \, \theta) \Big)^n \, {\rm d}\theta \, + \,
\dfrac{1}{2\, \pi} \, \int_\delta^\pi \,
{\rm e}^{\, \mathbf{i} \, n \, \omega \, \theta} \, \Big( {\rm e}^{\, \mathbf{i} \, \alpha \, \theta} \, \widehat{F}_a (- \, \theta) \Big)^n \, {\rm d}\theta \, .
$$
For reasons that will be made clear in the following lines, we choose some parameters $\delta$ and $\omega_\star$ such that the following inequalities hold:
\begin{equation}
\label{restrictionsprop4}
16 \, C_0 \, \delta \, \le \, \dfrac{c_4}{2} \, ,\quad \delta \, \le \, \delta_0 \, ,\quad
\sqrt{\dfrac{\omega_\star}{3 \, c_3}} \, + \, \dfrac{2 \, c_4}{9 \, c_3^{\, 2}} \, \omega_\star \, \le \, \dfrac{\delta}{2} \, ,\quad
16 \, C_0 \, \sqrt{\dfrac{\omega_\star}{3 \, c_3}} \, \le \, \dfrac{c_4}{2} \, ,\quad 12 \, c_4 \, \sqrt{\omega_\star} \, \le \, (3 \, c_3)^{\, 3/2} \, .
\end{equation}
The parameter $\delta$ in the above decomposition of $\Gjn$ is fixed once and for all, and we consider $\omega \in [- \, \omega_\star,-\, n^{-\, 2/3}]$. We
shall further need to restrict the possible values of $\omega$ later on but the restrictions \eqref{restrictionsprop4} are a starting point for several terms that
arise below.

With the choice \eqref{restrictionsprop4} for $\delta$, we use Assumption \ref{hyp:1} to write $\Gjn$ as ($\omega$ is negative here):
\begin{equation}
\label{prop4decomposition}
\Gjn \, = \, \varepsilon_j^n \, + \,
\dfrac{1}{2\, \pi} \, \int_{- \, \delta}^\delta \, \exp \, \Big( \, \mathbf{i} \, n \, \big( - \, | \, \omega \, | \, \theta \, + \, c_3 \, \theta^{\, 3} \big) \, - \, n \, c_4 \, \theta^{\, 4} \,
- \, n \, \theta^{\, 5} \, \varphi(-\theta) \, \Big) \, {\rm d}\theta \, ,
\end{equation}
with
\begin{equation}
\label{prop4estim1}
\big| \, \varepsilon_j^n \, \big| \, \le \, C \, {\rm e}^{- \, c \, n} \, ,
\end{equation}
for suitable constants $C>0$ and $c>0$. For the integral on the right hand side of \eqref{prop4decomposition}, we use Cauchy's formula and choose the
contour depicted in Figure \ref{fig:contour3} which consists in:
\begin{itemize}
 \item A vertical segment from $-\delta$ to $-\delta +\mathbf{i} \, \sqrt{| \, \omega \, | / (3\, c_3)}$,
 \item A horizontal segment from $-\delta +\mathbf{i} \, \sqrt{| \, \omega \, | / (3\, c_3)}$ to $-2 \, \sqrt{| \, \omega \, | / (3\, c_3)} \, - \, (2\, c_4 / (9 \, c_3^{\, 2})) \,
 | \, \omega \, | +\mathbf{i} \, \sqrt{| \, \omega \, | / (3\, c_3)}$,
 \item A segment (with slope $-\pi/4$) from the point $-2 \, \sqrt{| \, \omega \, | / (3\, c_3)} \, - \, (2\, c_4 / (9 \, c_3^{\, 2})) \, | \, \omega \, |
 +\mathbf{i} \, \sqrt{| \, \omega \, | / (3\, c_3)}$ to $-\mathbf{i} \, (\sqrt{| \, \omega \, | / (3 \, c_3)} \, + \, (2 \, c_4 / (9 \, c_3^{\, 2})) \, | \, \omega \, |)$,
 \item A segment (with slope $\pi/4$) from the point $-\mathbf{i} \, (\sqrt{| \, \omega \, | / (3 \, c_3)} \, + \, (2 \, c_4 / (9 \, c_3^{\, 2})) \, | \, \omega \, |)$ to
 $2 \, \sqrt{| \, \omega \, | / (3\, c_3)} \, + \, (2\, c_4 / (9 \, c_3^{\, 2})) \, | \, \omega \, | +\mathbf{i} \, \sqrt{| \, \omega \, | / (3\, c_3)}$,
 \item A horizontal segment from $2 \, \sqrt{| \, \omega \, | / (3\, c_3)} \, + \, (2\, c_4 / (9 \, c_3^{\, 2})) \, | \, \omega \, | +\mathbf{i} \, \sqrt{| \, \omega \, | / (3\, c_3)}$
 to $\delta +\mathbf{i} \, \sqrt{| \, \omega \, | / (3\, c_3)}$,
 \item A final vertical segment from $\delta +\mathbf{i} \, \sqrt{| \, \omega \, | / (3\, c_3)}$ to $\delta$.
\end{itemize}

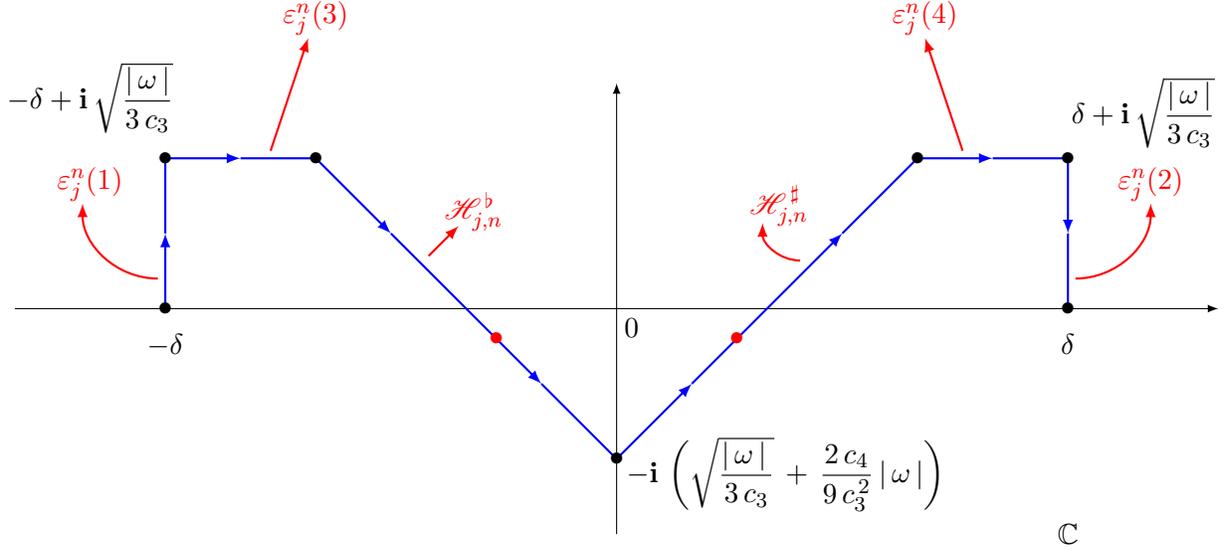
\begin{figure}[h!]
\begin{center}
\begin{tikzpicture}[scale=2,>=latex]
\draw[black,->] (-4,0) -- (4,0);
\draw[black,->] (0,-1.5)--(0,1.5);
\draw[thick,blue,->] (-3,0) -- (-3,0.5);
\draw[thick,blue] (-3,0.5) -- (-3,1);
\draw[thick,blue,->] (-3,1) -- (-2.5,1);
\draw[thick,blue] (-2.5,1) -- (-2,1);
\draw[thick,blue,->] (-2,1) -- (-1.5,0.5);
\draw[thick,blue] (-1.5,0.5) -- (-1,0);
\draw[thick,blue,->] (-1,0) -- (-0.5,-0.5);
\draw[thick,blue] (-0.5,-0.5) -- (0,-1);
\draw[thick,blue,->] (0,-1) -- (0.5,-0.5);
\draw[thick,blue] (0.5,-0.5) -- (1,0);
\draw[thick,blue,->] (1,0) -- (1.5,0.5);
\draw[thick,blue] (1.5,0.5) -- (2,1);
\draw[thick,blue,->] (2,1) -- (2.5,1);
\draw[thick,blue] (2.5,1) -- (3,1);
\draw[thick,blue,->] (3,1) -- (3,0.5);
\draw[thick,blue] (3,0.5) -- (3,0);
\draw (-3,-0.1) node[below]{$-\delta$};
\draw (3,-0.1) node[below]{$\delta$};
\draw (-3.5,1.1) node[above]{$-\delta+\mathbf{i} \, \sqrt{\dfrac{| \, \omega \, |}{3 \, c_3}}$};
\draw (3.5,1) node[above]{$\delta+\mathbf{i} \, \sqrt{\dfrac{| \, \omega \, |}{3 \, c_3}}$};
\draw (0,-1.1) node[right]{$-\mathbf{i} \, \left( \sqrt{\dfrac{| \, \omega \, | \, }{3 \, c_3}} \, + \, \dfrac{2 \, c_4}{9 \, c_3^{\, 2}} \, | \, \omega \, | \right)$};
\draw (0.1,0) node[below]{$0$};
\draw (3,-1.5) node {$\C$};
\node (centre) at (-3,1){$\bullet$};
\node (centre) at (3,1){$\bullet$};
\node (centre) at (-2,1){$\bullet$};
\node (centre) at (2,1){$\bullet$};
\node (centre) at (0,-1){$\bullet$};
\node[red] (centre) at (-0.8,-0.2){$\bullet$};
\node[red] (centre) at (0.8,-0.2){$\bullet$};
\node (centre) at (-3,0){$\bullet$};
\node (centre) at (3,0){$\bullet$};
\draw (-3.5,0.65) node[above]{{\color{red}$\varepsilon_j^n(1)$}};
\draw (3.55,0.65) node[above]{{\color{red}$\varepsilon_j^n(2)$}};
\draw (-2,1.75) node[above]{{\color{red}$\varepsilon_j^n(3)$}};
\draw (2.05,1.75) node[above]{{\color{red}$\varepsilon_j^n(4)$}};
\draw (-0.95,0.45) node[above]{{\color{red}$\mathcal{H}_{j,n}^{\, \flat}$}};
\draw (1.08,0.48) node[above]{{\color{red}$\mathcal{H}_{j,n}^{\, \sharp}$}};
\draw[thick,red,->] (-3.05,0.2) arc (270:180:0.5);
\draw[thick,red,->] (3.05,0.2) arc (270:360:0.5);
\draw[thick,red,->] (-2.3,1.05) -- (-2.05,1.8);
\draw[thick,red,->] (2.3,1.05) -- (2.05,1.8);
\draw[thick,red,->] (-1.25,0.35) -- (-1.05,0.55);
\draw[thick,red,->] (1.22,0.32) arc (270:180:0.25);
\end{tikzpicture}
\caption{The integration contour in the case $-\omega_0 \le \omega \le -n^{-2/3}$. The two red bullets represent the approximate saddle points and
the black bullets represent the end points of the six segments along which we compute the contributions $\varepsilon_j^n(1)$, $\varepsilon_j^n(3)$,
$\mathcal{H}_{j,n}^{\, \flat}$, $\mathcal{H}_{j,n}^{\, \sharp}$, $\varepsilon_j^n(4)$ and $\varepsilon_j^n(2)$.}
\label{fig:contour3}
\end{center}
\end{figure}

Thanks to the restrictions \eqref{restrictionsprop4}, this contour is included in the closed square $[-\delta_0,\delta_0] \times [-\delta_0,\delta_0]$ on
which $\varphi$ is a holomorphic function and we can bound its modulus by $C_0$. According to our choice of contour, we decompose $\Gjn$ in
\eqref{prop4decomposition} as:
\begin{equation}
\label{prop4decomposition'}
\Gjn \, = \, \varepsilon_j^n \, + \, \varepsilon_j^n(1) \, + \, \varepsilon_j^n(3) \, + \, \mathcal{H}_{j,n}^{\, \flat} \, + \, \mathcal{H}_{j,n}^{\, \sharp} \, + \,
\varepsilon_j^n(4) \, + \, \varepsilon_j^n(2) \, ,
\end{equation}
where the six contributions $\varepsilon_j^n(1),\varepsilon_j^n(3),\mathcal{H}_{j,n}^{\, \flat},\mathcal{H}_{j,n}^{\, \sharp},\varepsilon_j^n(4),\varepsilon_j^n(2)$
correspond to the integral of the holomorphic function:
$$
\theta \, \longmapsto \, \dfrac{1}{2\, \pi} \, \exp \,
\Big( \, \mathbf{i} \, n \, \big( - \, | \, \omega \, | \, \theta \, + \, c_3 \, \theta^{\, 3} \big) \, - \, n \, c_4 \, \theta^{\, 4} \, - \, n \, \theta^{\, 5} \, \varphi(-\theta) \, \Big) \, ,
$$
along each of the six segments that make the above defined contour (we refer to Firgure \ref{fig:contour3} for an illustration).

Let us start with the contribution $\varepsilon_j^n(1)$ that corresponds to the integral along the first vertical segment:
\begin{multline*}
\varepsilon_j^n(1) \, = \, \int_0^{\sqrt{\frac{| \, \omega \, |}{3 \, c_3}}} \,
\exp \, \Big( \, \mathbf{i} \, n \, \big( - \, | \, \omega \, | \, (-\delta+\mathbf{i} \, y) \, + \, c_3 \, (-\delta+\mathbf{i} \, y)^{\, 3} \big) \\
- \, n \, c_4 \, (-\delta+\mathbf{i} \, y)^{\, 4} \, - \, n \, (-\delta+\mathbf{i} \, y)^{\, 5} \, \varphi(\delta-\mathbf{i} \, y) \, \Big) \, \mathbf{i} \, {\rm d}y \, .
\end{multline*}
Expanding the terms within the integral, we get:
$$
\varepsilon_j^n(1) \, = \, {\rm e}^{- \, n \, c_4 \, \delta^{\, 4}} \, \int_0^{\sqrt{\frac{| \, \omega \, |}{3 \, c_3}}} \, {\rm e}^{\, \mathbf{i} \, \cdots} \, \,
{\rm e}^{- \, n \, \Big( (- \, | \, \omega \, | \, + \, 3 \, c_3 \, \delta^{\, 2}) \, y \, - \, 6 \, c_4 \, \delta^{\, 2} \, y^{\, 2} \, - \, c_3 \, y^{\, 3} \, + \, c_4 \, y^{\, 4} \Big)} \,
{\rm e}^{\, - \, n \, (-\delta+\mathbf{i} \, y)^{\, 5} \, \varphi(\delta-\mathbf{i} \, y)} \, {\rm d}y \, .
$$
We take the modulus on each side of the equality and apply the triangle inequality as well as the (already used) H\"older inequality for the remainder
term to obtain (the arguments here are the same as in the proof of Proposition \ref{prop2}):
$$
\big| \, \varepsilon_j^n(1) \, \big| \, \le \, {\rm e}^{- \, n \, \frac{c_4}{2} \, \delta^{\, 4}} \, \int_0^{\sqrt{\frac{| \, \omega \, |}{3 \, c_3}}} \,
{\rm e}^{- \, n \, \Big( (-\, | \, \omega \, | \, + \, 3 \, c_3 \, \delta^{\, 2}) \, y \, - \, 6 \, c_4 \, \delta^{\, 2} \, y^{\, 2} \, - \, c_3 \, y^{\, 3} \Big)} \, {\rm d}y \, .
$$
Let us now note that on the interval $[0,\sqrt{| \, \omega \, |/(3 \, c_3)}]$, we have:
$$
- \, c_3 \, y^{\, 3} \, \ge \, - \, \dfrac{| \, \omega \, |}{3} \, y \, ,
$$
and we also use the inequality (see \eqref{restrictionsprop4} and use $| \, \omega \, | \le \omega_\star$):
$$
\dfrac{4 \, | \, \omega \, |}{3} \, \le \, c_3 \, \delta^{\, 2} \, ,
$$
to get:
\begin{equation}
\label{prop4estim2}
\big| \, \varepsilon_j^n(1) \, \big| \, \le \, {\rm e}^{- \, n \, \frac{c_4}{2} \, \delta^{\, 4}} \, \int_0^{\sqrt{\frac{| \, \omega \, |}{3 \, c_3}}} \,
{\rm e}^{- \, n \, (\frac{-4 \, | \, \omega \, |}{3} \, + \, \frac{3}{2} \, c_3 \, \delta^{\, 2}) \, y} \, {\rm d}y \, \le \,
{\rm e}^{- \, n \, \frac{c_4}{2} \, \delta^{\, 4}} \, \int_0^{+ \, \infty} \, {\rm e}^{- \, n \, \frac{c_3}{2} \, \delta^{\, 2} \, y} \, {\rm d}y \, \le \, C \, {\rm e}^{- \, c \, n} \, ,
\end{equation}
for suitable constants $C>0$ and $c>0$ that do not depend on $\omega$ and $n$. The estimate of the integral $\varepsilon_j^n(2)$ along the final vertical
segment is entirely similar so we can collect \eqref{prop4estim1} and \eqref{prop4estim2} to obtain:
\begin{equation}
\label{prop4estim3}
\big| \, \varepsilon_j^n \, + \, \varepsilon_j^n(1) \, + \, \varepsilon_j^n(2) \, \big| \, \le \, C \, {\rm e}^{- \, c \, n} \, .
\end{equation}

We now turn to the contribution $\varepsilon_j^n(3)$ which corresponds to the integral along the horizontal segment from $-\delta +\mathbf{i} \, \sqrt{| \, \omega \, | / (3\, c_3)}$
to $-2 \, \sqrt{| \, \omega \, | / (3\, c_3)} \, - \, (2\, c_4 / (9 \, c_3^{\, 2})) \, | \, \omega \, | +\mathbf{i} \, \sqrt{| \, \omega \, | / (3\, c_3)}$. We compute:
\begin{multline*}
\varepsilon_j^n(3) \, = \, \dfrac{1}{2\, \pi} \, \int_{- \, \delta}^{\Xi(\omega)} \, \exp \, \left( \, \mathbf{i} \, n \, \left( - \, | \, \omega \, | \,
\left( \mathbf{i} \, \sqrt{\frac{| \, \omega \, |}{3 \, c_3}} +\theta \right) \, + \, c_3 \, \left( \mathbf{i} \, \sqrt{\frac{| \, \omega \, |}{3 \, c_3}} +\theta \right)^{\, 3} \right) \right. \\
\left. - \, n \, c_4 \, \left( \mathbf{i} \, \sqrt{\frac{| \, \omega \, |}{3 \, c_3}} +\theta \right)^{\, 4} \, - \, n \, \left( \mathbf{i} \, \sqrt{\frac{| \, \omega \, |}{3 \, c_3}} +\theta \right)^{\, 5} \,
\varphi \left( -\, \mathbf{i} \, \sqrt{\frac{| \, \omega \, |}{3 \, c_3}} -\theta \right) \, \right) \, {\rm d}\theta \, ,
\end{multline*}
where the upper bound $\Xi(\omega)$ in the integral is defined as:
$$
\Xi(\omega) \, := \, -2 \, \sqrt{\dfrac{| \, \omega \, |}{3\, c_3}} \, - \, \dfrac{2\, c_4}{9 \, c_3^{\, 2}} \, | \, \omega \, | \, .
$$
We collect the real and purely imaginary contributions in the exponential functions of the expression for $\varepsilon_j^n(3)$, which yields:
\begin{multline*}
\varepsilon_j^n(3) \, = \, \dfrac{{\rm e}^{\frac{4}{3 \, \sqrt{3 \, c_3}} \, n \, | \, \omega \, |^{\, 3/2}} \, {\rm e}^{- \, \frac{c_4}{9 \, c_3^{\, 2}} \, n \, \omega^{\, 2}}}{2 \, \pi}
\, \int_{- \, \delta}^{\Xi(\omega)} \, {\rm e}^{\, \mathbf{i} \, \cdots} \, {\rm e}^{- \, n \, \Big( \sqrt{3 \, c_3 \, | \, \omega \, |} \, - \, \frac{2 \, c_4}{c_3} \, | \, \omega \, | \Big) \, \theta^{\, 2}} \,
{\rm e}^{- \, n \, c_4 \, \theta^{\, 4}} \, \times \\
\exp \, \left( - \, n \, \left( \mathbf{i} \, \sqrt{\frac{| \, \omega \, |}{3 \, c_3}} +\theta \right)^{\, 5} \, \varphi \left( -\mathbf{i} \, \sqrt{\frac{| \, \omega \, |}{3 \, c_3}} -\theta \right) \, \right)
\, {\rm d}\theta \, .
\end{multline*}
The situation is less favorable than what occurred in the proof of Proposition \ref{prop2} because we have a large factor in front of the integral (since $\omega$
is small, the $| \, \omega \, |^{\, 3/2}$ term is dominant with respect to the second term which scales as $\omega^{\, 2}$). Applying the triangle inequality as well
as H\"older's inequality, we use \eqref{restrictionsprop4} and get:
$$
\big| \, \varepsilon_j^n(3) \, \big| \, \le \, \dfrac{{\rm e}^{\frac{4}{3 \, \sqrt{3 \, c_3}} \, n \, | \, \omega \, |^{\, 3/2}} \, {\rm e}^{- \, \frac{c_4}{18 \, c_3^{\, 2}} \, n \, \omega^{\, 2}}}{2 \, \pi}
\, \int_{- \, \delta}^{\Xi(\omega)} \, {\rm e}^{- \, n \, \Big( \sqrt{3 \, c_3 \, | \, \omega \, |} \, - \, \frac{2 \, c_4}{c_3} \, | \, \omega \, | \Big) \, \theta^{\, 2}} \, {\rm d}\theta \, .
$$
Using \eqref{restrictionsprop4} again, we have:
$$
\dfrac{2 \, c_4}{c_3} \, | \, \omega \, | \, \le \, \dfrac{1}{2} \, \sqrt{3 \, c_3 \, | \, \omega \, |} \, ,
$$
and we thus obtain:
\begin{align*}
\big| \, \varepsilon_j^n(3) \, \big| \, &\le \, \dfrac{{\rm e}^{\frac{4}{3 \, \sqrt{3 \, c_3}} \, n \, | \, \omega \, |^{\, 3/2}} \, {\rm e}^{- \, \frac{c_4}{18 \, c_3^{\, 2}} \, n \, \omega^{\, 2}}}{2 \, \pi}
\, \int_{- \, \delta}^{\Xi(\omega)} \, {\rm e}^{- \, \frac{n}{2} \, \sqrt{3 \, c_3 \, | \, \omega \, |} \, \theta^{\, 2}} \, {\rm d}\theta \\
&\le \, \dfrac{{\rm e}^{\frac{4}{3 \, \sqrt{3 \, c_3}} \, n \, | \, \omega \, |^{\, 3/2}} \, {\rm e}^{- \, \frac{c_4}{18 \, c_3^{\, 2}} \, n \, \omega^{\, 2}}}{2 \, \pi}
\, \int_{- \, \infty}^{- \, 2 \, \sqrt{\frac{| \, \omega \, |}{3\, c_3}}} \, {\rm e}^{- \, \frac{n}{2} \, \sqrt{3 \, c_3 \, | \, \omega \, |} \, \theta^{\, 2}} \, {\rm d}\theta \, ,
\end{align*}
where we used:
$$
\Xi(\omega) \, \le \, - \, 2 \, \sqrt{\dfrac{| \, \omega \, |}{3\, c_3}} \, .
$$
We now use the following inequality which is valid for any couple $(a,X)$ of positive numbers:
$$
\int_{-\, \infty}^{-\, X} \, {\rm e}^{- \, a \, y^{\, 2}} \, {\rm d}y \, \le \, \dfrac{1}{2\, a \, X} \, {\rm e}^{- \, a \, X^{\, 2}} \, ,
$$
and obtain:
$$
\big| \, \varepsilon_j^n(3) \, \big| \, \le \,
\dfrac{{\rm e}^{- \, \frac{2}{3 \, \sqrt{3 \, c_3}} \, n \, | \, \omega \, |^{\, 3/2}} \, {\rm e}^{- \, \frac{c_4}{18 \, c_3^{\, 2}} \, n \, \omega^{\, 2}}}{4 \, \pi \, n \, | \, \omega \, |}
\, \le \, \dfrac{{\rm e}^{- \, \frac{2}{3 \, \sqrt{3 \, c_3}} \, n \, | \, \omega \, |^{\, 3/2}}}{4 \, \pi \, n \, | \, \omega \, |} \, .
$$
The estimate of the contribution $\varepsilon_j^n(4)$ is entirely similar and we have thus obtained, for suitable constants $C$ and $c$:
\begin{equation}
\label{prop4estim4}
\big| \, \varepsilon_j^n(3) \, + \, \varepsilon_j^n(4) \, \big| \, \le \, \dfrac{C}{n \, | \, \omega \, |} \, {\rm e}^{- \, c \, n \, | \, \omega \, |^{\, 3/2}} \, .
\end{equation}
Comparing with the estimate in \eqref{prop4estim3}, we see that the upper bound ${\rm e}^{- \, c \, n}$ in \eqref{prop4estim3} is always smaller
than the right hand side in \eqref{prop4estim4} (up to choosing appropriate constants $C$ and $c$). We can thus add \eqref{prop4estim3} and
\eqref{prop4estim4} and obtain, for new constants $C$ and $c$:
\begin{equation}
\label{prop4estim5}
\big| \, \varepsilon_j^n \, + \, \varepsilon_j^n(1) \, + \, \varepsilon_j^n(2) \, + \, \varepsilon_j^n(3) \, + \, \varepsilon_j^n(4) \, \big|
\, \le \, \dfrac{C}{n \, | \, \omega \, |} \, {\rm e}^{- \, c \, n \, | \, \omega \, |^{\, 3/2}} \, ,
\end{equation}
for $-\omega_\star \le \omega \le -n^{-2/3}$ and $\omega_\star$ satisfying \eqref{restrictionsprop4}.
\bigskip

We now focus on the integrals $\mathcal{H}_{j,n}^{\, \flat}$ and $\mathcal{H}_{j,n}^{\, \sharp}$ which correspond to the leading contributions arising from
the segments that pass through the (approximate) saddle points. These contributions will give rise to the term $\GGjn$ defined in \eqref{defprincipal}. In
what follows, we shall further need to restrict the possible values of $\omega$ since the condition $|\omega| \le \omega_\star$ will not be sufficient to
absorb several terms that arise during the analysis.

Let us start with $\mathcal{H}_{j,n}^{\, \flat}$. We parametrize the corresponding segment by:
\begin{equation}
\label{parametrisation}
t \in \left[ \, - \, \sqrt{2} \, \left( \sqrt{\dfrac{| \, \omega \, |}{3\, c_3}} \, + \, \dfrac{2\, c_4}{9 \, c_3^{\, 2}} \, | \, \omega \, | \right) \, , \,
\sqrt{2} \, \sqrt{\dfrac{| \, \omega \, |}{3\, c_3}} \, \right] \, \longmapsto \, \Theta(t) \, := \,
- \, \sqrt{\dfrac{| \, \omega \, |}{3\, c_3}} \, - \, \mathbf{i} \, \dfrac{2\, c_4}{9 \, c_3^{\, 2}} \, | \, \omega \, | \, + \, t \, {\rm e}^{\, - \, \mathbf{i} \, \pi / 4} \, ,
\end{equation}
and we therefore obtain the expression:
$$
\mathcal{H}_{j,n}^{\, \flat} \, = \, \dfrac{{\rm e}^{\, - \, \mathbf{i} \, \pi / 4}}{2 \, \pi} \, \int_{t_{\rm min}}^{t_{\rm max}} \,
\exp \, \Big( \, \mathbf{i} \, n \, \big( - \, | \, \omega \, | \, \Theta(t) \, + \, c_3 \, \Theta(t)^{\, 3} - \, n \, c_4 \, \Theta(t)^{\, 4} \, - \,
n \, \Theta(t)^{\, 5} \, \varphi(- \Theta(t)) \, \Big) \, {\rm d}t \, ,
$$
where $\Theta(t)$ is given in \eqref{parametrisation} and the interval bounds $t_{\rm min},t_{\rm max}$ correspond to the end points of the interval
in \eqref{parametrisation}:
\begin{equation}
\label{tmintmax}
t_{\rm min} \, := \, - \, \sqrt{2} \, \left( \sqrt{\dfrac{| \, \omega \, |}{3\, c_3}} \, + \, \dfrac{2\, c_4}{9 \, c_3^{\, 2}} \, | \, \omega \, | \right) \quad
t_{\rm max} \, := \, \sqrt{2} \, \sqrt{\dfrac{| \, \omega \, |}{3\, c_3}} \, .
\end{equation}
We use the expression of $\Theta(t)$ in \eqref{parametrisation} to compute:
$$
- \, \mathbf{i} \, | \, \omega \, | \, \Theta(t) \, + \, \mathbf{i} \, c_3 \, \Theta(t)^{\, 3} - \, c_4 \, \Theta(t)^{\, 4} \, = \,
p_0(\omega) \, + \, p_1(\omega) \, t \, + \, \cdots \, + \, p_4(\omega) \, t^{\, 4} \, ,
$$
where the complex valued functions $p_0,\dots,p_4$ depend on $\omega$ only, are continuous on $\R$ and satisfy the following expansions as the
real variable $\omega$ tends to zero:
\begin{subequations}
\label{propertiespj}
\begin{align}
\text{\rm Re } \, p_0(\omega) \, &= \, - \, \dfrac{c_4}{9 \, c_3^{\, 2}} \, \omega^{\, 2} \, + \, O(| \, \omega \, |^{\, 3}) \, ,\label{partiereellep0} \\
\text{\rm Im } \, p_0(\omega) \, &= \, \dfrac{2}{3 \, \sqrt{3 \, c_3}} \, | \, \omega \, |^{\, 3/2} \, + \, O(| \, \omega \, |^{\, 5/2}) \, ,\label{partieimagp0} \\
p_1(\omega) \, &= \, O(\omega^{\, 2}) \, ,\label{p1} \\
\text{\rm Re } \, p_2(\omega) \, &= \, - \, \sqrt{3 \, c_3 \, | \, \omega \, |} \, + \, O(| \, \omega \, |^{\, 3/2}) \, ,\label{partiereellep2} \\
\text{\rm Im } \, p_2(\omega) \, &= \, O(| \, \omega \, |) \, ,\label{partieimagp2} \\
p_3(\omega) \, &= \, c_3 \, {\rm e}^{\, - \, \mathbf{i} \, \pi / 4} \, + \, O(| \, \omega \, |^{\, 1/2}) \, ,\label{p3} \\
p_4(\omega) \, &= \, c_4 \, .\label{p4}
\end{align}
\end{subequations}
At this stage, we have written $\mathcal{H}_{j,n}^{\, \flat}$ under the form:
\begin{equation}
\label{expression1Hjn}
\mathcal{H}_{j,n}^{\, \flat} \, = \, \dfrac{{\rm e}^{\, - \, \mathbf{i} \, \pi / 4 \, + \, n \, p_0(\omega)}}{2 \, \pi} \, \int_{t_{\rm min}}^{t_{\rm max}}
\, {\rm e}^{\, n \, \big( \, p_1(\omega) \, t \, + \, \cdots \, + \, p_4(\omega) \, t^{\, 4} \, \big) \, - \, n \, \Theta(t)^{\, 5} \, \varphi(-\Theta(t))} \, {\rm d}t \, .
\end{equation}
In what follows, we use the expansions \eqref{propertiespj} to simplify the expression of $\mathcal{H}_{j,n}^{\, \flat}$ by getting rid of the terms
in the exponential that contribute for remainders only. It is important to understand that here, a remainder is meant to be a quantity $\upsilon_{j,n}$
that is uniformly (in $n$) summable with respect to $j$. Up to now, we have assumed that $\omega$ is negative and satisfies $n^{\, -2/3} \le
| \, \omega \, | \le \omega_\star$ where $\omega_\star>0$ has been fixed in order to meet the restrictions \eqref{restrictionsprop4}. The further
restrictions on $\omega$ will help us absorb several terms. Namely, with $\omega_\star$ fixed as in \eqref{restrictionsprop4}, there exists a
constant $C>0$ such that, for any $\omega \in [-\omega_\star,0)$, there holds (see \eqref{propertiespj}):
$$
\dfrac{| \, p_1(\omega) \, |}{\omega^{\, 2}} \, + \, \dfrac{\big| \, \text{\rm Re } \, p_2(\omega) \, + \, \sqrt{3 \, c_3 \, | \, \omega \, |} \, \big|}{| \, \omega \, |^{\, 3/2}}
\, + \, \dfrac{| \, \text{\rm Im } \, p_2(\omega) \, |}{| \, \omega \, |} \, + \,
\dfrac{\big| \, p_3(\omega) \, - \, c_3 \, {\rm e}^{\, - \, \mathbf{i} \, \pi / 4} \, \big|}{| \, \omega \, |^{\, 1/2}} \, \le \, C \, .
$$

Using the inequality:
$$
\forall \, z \in \C \, ,\quad | \, {\rm e}^{\, z} \, - \, 1 \, | \, \le \, | \, z \, | \, {\rm e}^{\, | \, z \, |} \, ,
$$
as well as the estimate (see \eqref{parametrisation} and \eqref{tmintmax}):
$$
| \, t_{\rm min} \, | \, + \, | \, t_{\rm max} \, |  \, + \, | \, \Theta(t) \, | \, \le \, C \, | \, \omega \, |^{\, 1/2} \, ,
$$
for a suitable constant $C>0$ that does not depend on $\omega \in [-\omega_\star,0)$, we obtain the estimate (for another suitable constant $C$
independent of $\omega$ and $n$, at least for $| \, \omega \, | \le \omega_\star$):
\begin{multline}
\label{oscillestim1}
\left| \, \mathcal{H}_{j,n}^{\, \flat} \, - \, \dfrac{{\rm e}^{\, - \, \mathbf{i} \, \pi / 4 \, + \, n \, p_0(\omega)}}{2 \, \pi} \, \int_{t_{\rm min}}^{t_{\rm max}}
\, {\rm e}^{\, - \, n \, \sqrt{3 \, c_3 \, | \, \omega \, |} \, t^{\, 2} \, + \, n \, c_3 \, {\rm e}^{\, - \, \mathbf{i} \, \pi / 4} \, t^{\, 3}} \, {\rm d}t \, \right| \\
\le \, C \, {\rm e}^{\, n \, \text{\rm Re} \, p_0(\omega) \, + \, C \, n \, | \, \omega \, |^{\, 5/2}} \, \int_{t_{\rm min}}^{t_{\rm max}}
{\rm e}^{\, - \, n \, \sqrt{3 \, c_3 \, | \, \omega \, |} \, t^{\, 2} \, + \, n \, \big( c_3 /\sqrt{2} \big) \, t^{\, 3}} \, \Big( n \, | \, \omega \, | \, t^{\, 2} \, + \, n \, | \, \omega \, |^{\, 5/2} \Big) \,
{\rm e}^{\, C \, n \, | \, \omega \, | \, t^{\, 2}} \, {\rm d}t \, .
\end{multline}

The crucial observation is now the following. In the integral on the right hand side of \eqref{oscillestim1}, either $t$ is negative and $t^{\, 3}$ is also negative,
so we have:
$$
- \, n \, \sqrt{3 \, c_3 \, | \, \omega \, |} \, t^{\, 2} \, + \, n \, \big( c_3 /\sqrt{2} \big) \, t^{\, 3} \, \le \, - \, n \, \sqrt{3 \, c_3 \, | \, \omega \, |} \, t^{\, 2} \, .
$$
Or $t$ is nonnegative and we have:
$$
- \, n \, \sqrt{3 \, c_3 \, | \, \omega \, |} \, t^{\, 2} \, + \, n \, \big( c_3 /\sqrt{2} \big) \, t^{\, 3} \, \le \, - \, n \, \sqrt{3 \, c_3 \, | \, \omega \, |} \, t^{\, 2}
\, + \, n \, \big( c_3 /\sqrt{2} \big) \, t_{\rm max} \, t^{\, 2} \, = \, - \, \dfrac{2 \, n}{3} \, \sqrt{3 \, c_3 \, | \, \omega \, |} \, t^{\, 2} \, .
$$
In both cases, we can use a bound of the form:
$$
{\rm e}^{\, - \, n \, \sqrt{3 \, c_3 \, | \, \omega \, |} \, t^{\, 2} \, + \, n \, \big( c_3 /\sqrt{2} \big) \, t^{\, 3}} \, \le \,
{\rm e}^{\, - \, c \, n \, | \, \omega \, |^{\, 1/2} \, t^{\, 2}} \, ,
$$
for a suitable constant $c>0$. Going back to \eqref{oscillestim1}, we can thus choose $\omega_0>0$ such that, for any $\omega \in [-\, \omega_0,0)$, there holds:
\begin{multline*}
\left| \, \mathcal{H}_{j,n}^{\, \flat} \, - \, \dfrac{{\rm e}^{\, - \, \mathbf{i} \, \pi / 4 \, + \, n \, p_0(\omega)}}{2 \, \pi} \, \int_{t_{\rm min}}^{t_{\rm max}}
\, {\rm e}^{\, - \, n \, \sqrt{3 \, c_3 \, | \, \omega \, |} \, t^{\, 2} \, + \, n \, c_3 \, {\rm e}^{\, - \, \mathbf{i} \, \pi / 4} \, t^{\, 3}} \, {\rm d}t \, \right| \\
\le \, C \, {\rm e}^{\, n \, \text{\rm Re} \, p_0(\omega) \, + \, C \, n \, | \, \omega \, |^{\, 5/2}} \, \int_{t_{\rm min}}^{t_{\rm max}}
{\rm e}^{\, - \, c \, n \, | \, \omega \, |^{\, 1/2} \, t^{\, 2}} \, \Big( n \, | \, \omega \, | \, t^{\, 2} \, + \, n \, | \, \omega \, |^{\, 5/2} \Big) \, {\rm d}t \, .
\end{multline*}
Using \eqref{partiereellep0}, we find that the real part of $p_0(\omega)$ can absorb the $O(| \, \omega \, |^{\, 5/2})$ remainder term, and we are eventually led
to the estimate:
\begin{multline*}
\left| \, \mathcal{H}_{j,n}^{\, \flat} \, - \, \dfrac{{\rm e}^{\, - \, \mathbf{i} \, \pi / 4 \, + \, n \, p_0(\omega)}}{2 \, \pi} \, \int_{t_{\rm min}}^{t_{\rm max}}
\, {\rm e}^{\, - \, n \, \sqrt{3 \, c_3 \, | \, \omega \, |} \, t^{\, 2} \, + \, n \, c_3 \, {\rm e}^{\, - \, \mathbf{i} \, \pi / 4} \, t^{\, 3}} \, {\rm d}t \, \right| \\
\le \, C \, {\rm e}^{\, - \, c \, n \, \omega^{\, 2}} \, \int_{t_{\rm min}}^{t_{\rm max}} {\rm e}^{\, - \, c \, n \, | \, \omega \, |^{\, 1/2} \, t^{\, 2}} \,
\Big( n \, | \, \omega \, | \, t^{\, 2} \, + \, n \, | \, \omega \, |^{\, 5/2} \Big) \, {\rm d}t \, .
\end{multline*}
It now remains to compute the integral on the right hand side and we obtain our first main simplification:
\begin{multline}
\label{oscillestim3}
\left| \, \mathcal{H}_{j,n}^{\, \flat} \, - \, \dfrac{{\rm e}^{\, - \, \mathbf{i} \, \pi / 4 \, + \, n \, p_0(\omega)}}{2 \, \pi} \, \int_{t_{\rm min}}^{t_{\rm max}}
\, {\rm e}^{\, - \, n \, \sqrt{3 \, c_3 \, | \, \omega \, |} \, t^{\, 2} \, + \, n \, c_3 \, {\rm e}^{\, - \, \mathbf{i} \, \pi / 4} \, t^{\, 3}} \, {\rm d}t \, \right| \\
\le \, C \, {\rm e}^{\, - \, c \, n \, \omega^{\, 2}} \, \left( \dfrac{1}{\sqrt{n}} \, | \, \omega \, |^{\, 1/4} \, + \, \sqrt{n} \, | \, \omega \, |^{\, 9/4} \right)
\, \le \, \dfrac{C}{\sqrt{n}} \, {\rm e}^{\, - \, c \, n \, \omega^{\, 2}} \, \left( 1 \, + \, n \, | \, \omega \, |^{\, 2} \right) \, \le \,
\dfrac{C}{\sqrt{n}} \, {\rm e}^{\, - \, c \, n \, \omega^{\, 2}} \, .
\end{multline}

We now simplify the term $p_0(\omega)$ in the left hand side of \eqref{oscillestim3} by using \eqref{partiereellep0} and \eqref{partieimagp0}.
Namely, by using the triangle inequality and \eqref{oscillestim3}, we get:
\begin{multline*}
\left| \, \mathcal{H}_{j,n}^{\, \flat} \, - \, \dfrac{{\rm e}^{\, - n \, \frac{c_4}{9 \, c_3^{\, 2}} \, \omega^{\, 2}} \,
{\rm e}^{\, \mathbf{i} \, n \, \frac{2}{3 \, \sqrt{3 \, c_3}} \, | \, \omega \, |^{\, 3/2} \, - \, \mathbf{i} \, \pi / 4}}{2 \, \pi} \, \int_{t_{\rm min}}^{t_{\rm max}}
\, {\rm e}^{\, - \, n \, \sqrt{3 \, c_3 \, | \, \omega \, |} \, t^{\, 2} \, + \, n \, c_3 \, {\rm e}^{\, - \, \mathbf{i} \, \pi / 4} \, t^{\, 3}} \, {\rm d}t \, \right| \\
\le \, \dfrac{C}{\sqrt{n}} \, {\rm e}^{\, - \, c \, n \, \omega^{\, 2}} \, + \, C \, n \, | \, \omega \, |^{\, 5/2} \, {\rm e}^{\, - c \, n \, \omega^{\, 2}} \,
\int_{t_{\rm min}}^{t_{\rm max}} {\rm e}^{\, - \, n \, \sqrt{3 \, c_3 \, | \, \omega \, |} \, t^{\, 2} \, + \, n \, \big( c_3 / \sqrt{2} \big) \, t^{\, 3}} \, {\rm d}t \, ,
\end{multline*}
for suitable constants $C>0$ and $c>0$ and $| \, \omega \, | \le \omega_0$ for a sufficiently small constant $\omega_0>0$. The final integral is
dealt with as above, meaning that we can absorb the $t^{\, 3}$ term within the $O(t^{\, 2})$ on the interval $[t_{\rm min},t_{\rm max}]$ and we
end up with the estimate:
\begin{multline}
\label{oscillestim4}
\left| \, \mathcal{H}_{j,n}^{\, \flat} \, - \, \dfrac{{\rm e}^{\, - n \, \frac{c_4}{9 \, c_3^{\, 2}} \, \omega^{\, 2}} \,
{\rm e}^{\, \mathbf{i} \, n \, \frac{2}{3 \, \sqrt{3 \, c_3}} \, | \, \omega \, |^{\, 3/2} \, - \, \mathbf{i} \, \pi / 4}}{2 \, \pi} \, \int_{t_{\rm min}}^{t_{\rm max}}
\, {\rm e}^{\, - \, n \, \sqrt{3 \, c_3 \, | \, \omega \, |} \, t^{\, 2} \, + \, n \, c_3 \, {\rm e}^{\, - \, \mathbf{i} \, \pi / 4} \, t^{\, 3}} \, {\rm d}t \, \right| \\
\le \, \dfrac{C}{\sqrt{n}} \, {\rm e}^{\, - \, c \, n \, \omega^{\, 2}} \, + \, C \, \sqrt{n} \, | \, \omega \, |^{\, 9/4} \, {\rm e}^{\, - c \, n \, \omega^{\, 2}}
\, \le \, \dfrac{C}{\sqrt{n}} \, {\rm e}^{\, - \, c \, n \, \omega^{\, 2}} \, ,
\end{multline}
for suitable constants $C>0$ and $c>0$ that are independent of $\omega$ and $n$, and for $| \, \omega \, | \le \omega_0$ for a sufficiently
small constant $\omega_0>0$.

There are still two steps to further simplify the contribution $\mathcal{H}_{j,n}^{\, \flat}$. The first step consists in restricting to a symmetric interval
for the integral. We recall once again the definition \eqref{tmintmax} of the interval bounds, and use again \eqref{oscillestim4} with the triangle
inequality to obtain:
\begin{multline*}
\left| \, \mathcal{H}_{j,n}^{\, \flat} \, - \, \dfrac{{\rm e}^{\, - \, n \, \frac{c_4}{9 \, c_3^{\, 2}} \, \omega^{\, 2}} \,
{\rm e}^{\, \mathbf{i} \, n \, \frac{2}{3 \, \sqrt{3 \, c_3}} \, | \, \omega \, |^{\, 3/2} \, - \, \mathbf{i} \, \pi / 4}}{2 \, \pi} \, \int_{-t_{\rm max}}^{t_{\rm max}}
\, {\rm e}^{\, - \, n \, \sqrt{3 \, c_3 \, | \, \omega \, |} \, t^{\, 2} \, + \, n \, c_3 \, {\rm e}^{\, - \, \mathbf{i} \, \pi / 4} \, t^{\, 3}} \, {\rm d}t \, \right| \\
\le \, \dfrac{C}{\sqrt{n}} \, {\rm e}^{\, - \, c \, n \, \omega^{\, 2}} \, + \, C \, {\rm e}^{\, - \, c \, n \, \omega^{\, 2}} \, \int_{t_{\rm min}}^{-t_{\rm max}}
\, {\rm e}^{\, - \, c \, n \, | \, \omega \, |^{\, 1/2} \, t^{\, 2}} \, {\rm d}t \, .
\end{multline*}
The length of the interval $[t_{\rm min},-t_{\rm max}]$ is $O(| \, \omega \, |)$, see \eqref{tmintmax}, and the function that is integrated is increasing
with respect to $t$ on the considered interval so we get:
\begin{multline}
\label{oscillestim5}
\left| \, \mathcal{H}_{j,n}^{\, \flat} \, - \, \dfrac{{\rm e}^{\, - \, n \, \frac{c_4}{9 \, c_3^{\, 2}} \, \omega^{\, 2}} \,
{\rm e}^{\, \mathbf{i} \, n \, \frac{2}{3 \, \sqrt{3 \, c_3}} \, | \, \omega \, |^{\, 3/2} \, - \, \mathbf{i} \, \pi / 4}}{2 \, \pi} \, \int_{-t_{\rm max}}^{t_{\rm max}}
\, {\rm e}^{\, - \, n \, \sqrt{3 \, c_3 \, | \, \omega \, |} \, t^{\, 2} \, + \, n \, c_3 \, {\rm e}^{\, - \, \mathbf{i} \, \pi / 4} \, t^{\, 3}} \, {\rm d}t \, \right| \\
\le \, \dfrac{C}{\sqrt{n}} \, {\rm e}^{\, - \, c \, n \, \omega^{\, 2}} \, + \,
C \, {\rm e}^{\, - \, c \, n \, \omega^{\, 2}} \, | \, \omega \, | \, {\rm e}^{\, - \, c \, n \, | \, \omega \, |^{\, 3/2}} \,
\, \le \, \dfrac{C}{\sqrt{n}} \, {\rm e}^{\, - \, c \, n \, \omega^{\, 2}} \, + \, \dfrac{C}{n^{\, 2/3}} \, {\rm e}^{\, - \, c \, n \, \omega^{\, 2}}
\, \le \, \dfrac{C}{\sqrt{n}} \, {\rm e}^{\, - \, c \, n \, \omega^{\, 2}} \, .
\end{multline}

Once we have restricted to a symmetric interval in $t$, the final step consists in getting rid of the $O(t^{\, 3})$ term in the integrated function.
To achieve this, we use the inequality:
$$
\forall \, z \in \C \, ,\quad | \, {\rm e}^{\, z} \, - \, 1 \, - \, z \, | \, \le \, \dfrac{| \, z \, |^{\, 2}}{2} \, {\rm e}^{\, | \, z \, |} \, ,
$$
and the fact that the function:
$$
t \in \R \quad \longmapsto t^{\, 3} \, {\rm e}^{\, - \, n \, \sqrt{3 \, c_3 \, | \, \omega \, |} \, t^{\, 2}} \, ,
$$
is odd to get:
\begin{multline}
\label{oscillestim6}
\left| \, \mathcal{H}_{j,n}^{\, \flat} \, - \, \dfrac{{\rm e}^{\, - \, n \, \frac{c_4}{9 \, c_3^{\, 2}} \, \omega^{\, 2}} \,
{\rm e}^{\, \mathbf{i} \, n \, \frac{2}{3 \, \sqrt{3 \, c_3}} \, | \, \omega \, |^{\, 3/2} \, - \, \mathbf{i} \, \pi / 4}}{2 \, \pi} \, \int_{-t_{\rm max}}^{t_{\rm max}}
\, {\rm e}^{\, - \, n \, \sqrt{3 \, c_3 \, | \, \omega \, |} \, t^{\, 2}} \, {\rm d}t \, \right| \\
\le \, \dfrac{C}{\sqrt{n}} \, {\rm e}^{\, - \, c \, n \, \omega^{\, 2}} \, + \, C \, {\rm e}^{\, - \, c \, n \, \omega^{\, 2}} \, \int_{-t_{\rm max}}^{t_{\rm max}}
\, {\rm e}^{\, - \, n \, \sqrt{3 \, c_3 \, | \, \omega \, |} \, t^{\, 2}} \, {\rm e}^{\, n \, c_3 \, | \, t \, |^{\, 3}} \, n^{\, 2} \, t^{\, 6} \, {\rm d}t \, .
\end{multline}
For $| \, t \, | \le t_{\rm max}$, we use the bound:
$$
- \, n \, \sqrt{3 \, c_3 \, | \, \omega \, |} \, t^{\, 2} \, + \, n \, c_3 \, | \, t \, |^{\, 3} \, \le \, - \, n \, \sqrt{3 \, c_3 \, | \, \omega \, |} \, t^{\, 2} \, + \, n \, c_3 \, t_{\rm max} \, t^{\, 2}
\, = \, - \, \left( 1 \, - \, \dfrac{\sqrt{2}}{3} \right) \, n \, \sqrt{3 \, c_3 \, | \, \omega \, |} \, t^{\, 2} \, ,
$$
and we can thus use \eqref{oscillestim6} to get:
\begin{multline*}
\left| \, \mathcal{H}_{j,n}^{\, \flat} \, - \, \dfrac{{\rm e}^{\, - \, n \, \frac{c_4}{9 \, c_3^{\, 2}} \, \omega^{\, 2}} \,
{\rm e}^{\, \mathbf{i} \, n \, \frac{2}{3 \, \sqrt{3 \, c_3}} \, | \, \omega \, |^{\, 3/2} \, - \, \mathbf{i} \, \pi / 4}}{2 \, \pi} \, \int_{-t_{\rm max}}^{t_{\rm max}}
\, {\rm e}^{\, - \, n \, \sqrt{3 \, c_3 \, | \, \omega \, |} \, t^{\, 2}} \, {\rm d}t \, \right| \\
\le \, \dfrac{C}{\sqrt{n}} \, {\rm e}^{\, - \, c \, n \, \omega^{\, 2}} \, + \, C \, {\rm e}^{\, - \, c \, n \, \omega^{\, 2}} \, \int_{-t_{\rm max}}^{t_{\rm max}}
\, n^{\, 2} \, t^{\, 6} \, {\rm e}^{\, - \, c \, n \, | \, \omega \, |^{\, 1/2} \, t^{\, 2}} \, {\rm d}t \\
\le \, \dfrac{C}{\sqrt{n}} \, {\rm e}^{\, - \, c \, n \, \omega^{\, 2}} \, + \, \dfrac{C}{n^{\, 3/2} \, | \, \omega \, |^{\, 7/4}} \, {\rm e}^{\, - \, c \, n \, \omega^{\, 2}} \, .
\end{multline*}
Writing $| \, \omega \, |^{\, 7/4} = | \, \omega \, | \, | \, \omega \, |^{\, 3/4}$ and bounding from below $| \, \omega \, |^{\, 3/4} \ge n^{-\, 1/2}$, we end up with
our final estimate:
\begin{equation}
\label{oscillestim7}
\left| \, \mathcal{H}_{j,n}^{\, \flat} \, - \, \dfrac{{\rm e}^{\, - \, n \, \frac{c_4}{9 \, c_3^{\, 2}} \, \omega^{\, 2}} \,
{\rm e}^{\, \mathbf{i} \, n \, \frac{2}{3 \, \sqrt{3 \, c_3}} \, | \, \omega \, |^{\, 3/2} \, - \, \mathbf{i} \, \pi / 4}}{2 \, \pi} \, \int_{-t_{\rm max}}^{t_{\rm max}}
\, {\rm e}^{\, - \, n \, \sqrt{3 \, c_3 \, | \, \omega \, |} \, t^{\, 2}} \, {\rm d}t \, \right| \, \le \,
\dfrac{C}{\sqrt{n}} \, {\rm e}^{\, - \, c \, n \, \omega^{\, 2}} \, + \, \dfrac{C}{n \, | \, \omega \, |} \, {\rm e}^{\, - \, c \, n \, \omega^{\, 2}} \, ,
\end{equation}
where \eqref{oscillestim7} holds for any $\omega \in [-\, \omega_0,-n^{- \, 2/3}]$ and $\omega_0>0$ is a sufficiently small constant.

We can now use exactly the same arguments to deal with the last contribution $\mathcal{H}_{j,n}^{\, \sharp}$. Leaving the details to the interested reader,
we end up with the estimate:
\begin{equation}
\label{oscillestim8}
\left| \, \mathcal{H}_{j,n}^{\, \sharp} \, - \, \dfrac{{\rm e}^{\, - \, n \, \frac{c_4}{9 \, c_3^{\, 2}} \, \omega^{\, 2}} \,
{\rm e}^{\, - \, \mathbf{i} \, n \, \frac{2}{3 \, \sqrt{3 \, c_3}} \, | \, \omega \, |^{\, 3/2} \, + \, \mathbf{i} \, \pi / 4}}{2 \, \pi} \, \int_{-t_{\rm max}}^{t_{\rm max}}
\, {\rm e}^{\, - \, n \, \sqrt{3 \, c_3 \, | \, \omega \, |} \, t^{\, 2}} \, {\rm d}t \, \right| \, \le \,
\dfrac{C}{\sqrt{n}} \, {\rm e}^{\, - \, c \, n \, \omega^{\, 2}} \, + \, \dfrac{C}{n \, | \, \omega \, |} \, {\rm e}^{\, - \, c \, n \, \omega^{\, 2}} \, ,
\end{equation}
that is entirely similar to \eqref{oscillestim7}. Eventually, adding \eqref{oscillestim7} and \eqref{oscillestim8} and using the triangle inequality again,
we end up with:
\begin{equation}
\label{oscillestim9}
\left| \, \mathcal{H}_{j,n}^{\, \flat} \, + \, \mathcal{H}_{j,n}^{\, \sharp} \, - \, \GGjn \, \right| \, \le \,
\dfrac{C}{\sqrt{n}} \, {\rm e}^{\, - \, c \, n \, \omega^{\, 2}} \, + \, \dfrac{C}{n \, | \, \omega \, |} \, {\rm e}^{\, - \, c \, n \, \omega^{\, 2}} \, ,
\end{equation}
where we recall the definition \eqref{defprincipal} of $\GGjn$. Adding \eqref{oscillestim9} with \eqref{prop4estim5}, we see that the largest term on the right
hand side is that of \eqref{prop4estim5}, which gives:
$$
\left| \, \Gjn \, - \, \GGjn \, \right| \, \le \, \dfrac{C}{n \, | \, \omega \, |} \, {\rm e}^{\, - \, c \, n \, | \, \omega \, |^{\, 3/2}} \, ,
$$
and the proof of Proposition \ref{prop4} is thus complete.
\end{proof}

\subsection{The tail}

It remains to consider the regime where the parameter $\omega$ in \eqref{formuleGjn} is negative and not small. Our result is summarized in the following
Proposition.

\begin{proposition}
\label{prop5}
Under Assumption \ref{hyp:1}, with the same $\omega_0>0$ as in Proposition \ref{prop4}, there exist two constants $C>0$ and $c>0$ such that the Green's
function in \eqref{def:green} satisfies:
$$
\forall \, n \in \N^* \, ,\quad \forall \, j \in \Z \, ,\quad \big| \, \Gjn \, \big| \, \le \, C \, {\rm e}^{- \, c \, n} \, ,
$$
as long as $n \in \N^*$ and the parameter $\omega$ defined in \eqref{defomega} satisfy $-2 \, M \le \omega \le -\omega_0$ (hence $j \, - \, \alpha \, n \, < \, 0$).
\end{proposition}

\begin{proof}
As in the proof of Proposition \ref{prop2}, we decompose the Green's function into:
$$
\Gjn \, = \, \varepsilon_j^n \, + \,
\dfrac{1}{2\, \pi} \, \int_{- \, \delta}^\delta \, {\rm e}^{\, \mathbf{i} \, n \, \omega \, \theta} \, \Big( {\rm e}^{\, \mathbf{i} \, \alpha \, \theta} \, \widehat{F}_a (- \, \theta) \Big)^n
\, {\rm d}\theta \, ,
$$
with $\delta \in (0,\pi)$ to be fixed and:
$$
\varepsilon_j^n \, := \, \dfrac{1}{2\, \pi} \, \int_{-\pi}^{- \, \delta} \,
{\rm e}^{\, \mathbf{i} \, n \, \omega \, \theta} \, \Big( {\rm e}^{\, \mathbf{i} \, \alpha \, \theta} \, \widehat{F}_a (- \, \theta) \Big)^n \, {\rm d}\theta \, + \,
\dfrac{1}{2\, \pi} \, \int_\delta^\pi \,
{\rm e}^{\, \mathbf{i} \, n \, \omega \, \theta} \, \Big( {\rm e}^{\, \mathbf{i} \, \alpha \, \theta} \, \widehat{F}_a (- \, \theta) \Big)^n \, {\rm d}\theta \, .
$$
The parameter $\delta \in (0,\pi)$ is fixed here in order to satisfy\footnote{Recall that $\omega_0>0$ is given here by the result of Proposition \ref{prop4} so
we may not modify it.}:
\begin{equation}
\label{restrictionsprop5}
\delta \, \le \, \delta_0 \, ,\qquad C_0 \, \delta \, \le \, \dfrac{c_4}{2} \, ,\qquad 3 \, c_3 \, \delta^{\, 2} \, \le \, \dfrac{\omega_0}{3} \, ,\qquad
6 \, c_4 \, \delta^3 \, \le \, \dfrac{\omega_0}{3} \, .
\end{equation}

From Assumption \ref{hyp:1}, we have the uniform bound:
$$
| \, \varepsilon_j^n \, | \, \le \, C \, {\rm e}^{\, - \, c \, n} \, ,
$$
so we focus on the second term in the above decomposition of $\Gjn$, which we denote $\mathcal{H}_j^n$ as in the proof of Proposition \ref{prop2}.
Thanks to our choice for $\delta$, we can use Cauchy's formula and use the segments $[-\, \delta,-\mathbf{i} \, \delta] \cup [-\mathbf{i} \, \delta,\delta]$
rather than the interval $[-\, \delta,\delta]$ as an integration contour (see Figure \ref{fig:contour4}). Those two segments are included in the closed square
$\{ z \in \C \, / \, \max \, ( \, |\text{\rm Re } z|,|\text{\rm Im } z| \, ) \, \le \, \delta_0 \}$ on which $\varphi$ is holomorphic and its modulus is bounded by $C_0$.
We thus obtain a decomposition:
$$
\mathcal{H}_j^n \, = \, \mathcal{H}_j^n(1) \, + \, \mathcal{H}_j^n (2) \, ,
$$
with ($\omega$ is negative here):
\begin{multline*}
\mathcal{H}_j^n(1) \, := \, \dfrac{\delta \, (1-\mathbf{i})}{2 \, \pi} \, \int_0^1 \, \exp \, \Big( \, - \, \mathbf{i} \, n \, | \, \omega \, | \, \big( -\delta \, (1-t) \, - \, \mathbf{i} \, \delta \, t \big)
\, + \, \mathbf{i} \, n \, c_3 \, \big( -\delta \, (1-t) \, - \, \mathbf{i} \, \delta \, t \big)^{\, 3} \\
- \, n \, c_4 \, \big( -\delta \, (1-t) \, - \, \mathbf{i} \, \delta \, t \big)^{\, 4} \, - \, n \, \big( -\delta \, (1-t) \, - \, \mathbf{i} \, \delta \, t \big)^{\, 5} \,
\varphi \big( \delta \, (1-t) \, + \, \mathbf{i} \, \delta \, t \big) \, \Big) \, {\rm d}t \, ,
\end{multline*}
and
\begin{multline*}
\mathcal{H}_j^n(2) \, := \, \dfrac{\delta \, (1+\mathbf{i})}{2 \, \pi} \, \int_0^1 \, \exp \, \Big( \, - \, \mathbf{i} \, n \, | \, \omega \, | \, \big( \delta \, t \, - \, \mathbf{i} \, \delta \, (1-t) \big)
\, + \, \mathbf{i} \, n \, c_3 \, \big( \delta \, t \, - \, \mathbf{i} \, \delta \, (1-t) \big)^{\, 3} \\
- \, n \, c_4 \, \big( \delta \, t \, - \, \mathbf{i} \, \delta \, (1-t) \big)^{\, 4} \, - \, n \, \big( \delta \, t \, - \, \mathbf{i} \, \delta \, (1-t) \big)^{\, 5} \,
\varphi \big( - \, \delta \, t \, + \, \mathbf{i} \, \delta \, (1-t) \big) \, \Big) \, {\rm d}t \, .
\end{multline*}

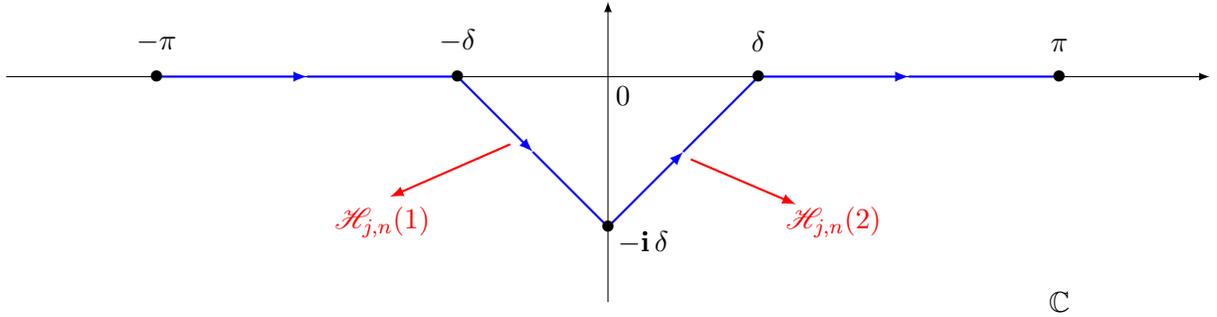
\begin{figure}[h!]
\begin{center}
\begin{tikzpicture}[scale=2,>=latex]
\draw[black,->] (-4,0) -- (4,0);
\draw[black,->] (0,-1.5)--(0,0.5);
\draw[thick,blue,->] (-3,0) -- (-2,0);
\draw[thick,blue] (-2,0) -- (-1,0);
\draw[thick,blue,->] (-1,0) -- (-0.5,-0.5);
\draw[thick,blue] (-0.5,-0.5) -- (0,-1);
\draw[thick,blue,->] (0,-1) -- (0.5,-0.5);
\draw[thick,blue] (0.5,-0.5) -- (1,0);
\draw[thick,blue,->] (1,0) -- (2,0);
\draw[thick,blue] (2,0) -- (3,0);
\draw (-3,0.1) node[above]{$-\pi$};
\draw (3,0.1) node[above]{$\pi$};
\draw (-1,0.1) node[above]{$-\delta$};
\draw (1,0.1) node[above]{$\delta$};
\draw (0,-1.1) node[right]{$-\mathbf{i} \, \delta$};
\draw (-1.5,-0.8) node[below]{{\color{red}$\mathcal{H}_{j,n}(1)$}};
\draw (1.5,-0.8) node[below]{{\color{red}$\mathcal{H}_{j,n}(2)$}};
\draw[thick,red,->] (-0.65,-0.45) -- (-1.45,-0.8);
\draw[thick,red,->] (0.55,-0.55) -- (1.25,-0.85);
\draw (0.1,0) node[below]{$0$};
\draw (3,-1.5) node {$\C$};
\node (centre) at (0,-1){$\bullet$};
\node (centre) at (-1,0){$\bullet$};
\node (centre) at (1,0){$\bullet$};
\node (centre) at (-3,0){$\bullet$};
\node (centre) at (3,0){$\bullet$};
\end{tikzpicture}
\caption{The integration contour in the case $-2 \, M \le \omega \le -\omega_0$.}
\label{fig:contour4}
\end{center}
\end{figure}

We deal with the estimate of the first integral $\mathcal{H}_j^n(1)$ and leave the similar estimate of $\mathcal{H}_j^n(2)$ to the interested reader.
Expanding the terms within the exponential function, we compute:
\begin{multline*}
\mathcal{H}_j^n(1) \, = \, \dfrac{\delta \, (1-\mathbf{i})}{2 \, \pi} \, \int_0^1 \, {\rm e}^{\, \mathbf{i} \, \cdots} \,
\exp \, \Big( \, - \, n \, | \, \omega \, | \, \delta \, t \, - \, n \, c_3 \, \delta^{\, 3} \, t^{\, 3} \, + \, 3 \, n \, c_3 \, \delta^{\, 3} \, (1-t)^{\, 2} \, t \, \Big) \, \times \\
\exp \, \Big( \, - \, n \, c_4 \, \delta^{\, 4} \, (1-t)^{\, 4} \, + \, 6 \, n \, c_4 \, \delta^{\, 4} \, (1-t)^{\, 2} \, t^{\, 2} \, - \, n \, c_4 \, \delta^{\, 4} \, t^{\, 4} \Big) \, \times \\
\exp \, \Big( \, - \, n \, \big( -\delta \, (1-t) \, - \, \mathbf{i} \, \delta \, t \big)^{\, 5} \, \varphi \big( \delta \, (1-t) \, + \, \mathbf{i} \, \delta \, t \big) \, \Big) \, {\rm d}t \, .
\end{multline*}
We take the modulus on each side of the equality sign and use the uniform bound on $\varphi$ to get:
\begin{multline*}
| \, \mathcal{H}_j^n(1) \, | \, \le \, C \, {\rm e}^{\, n \, C_0 \, \delta^{\, 5}} \, \int_0^1 \,
\exp \, \Big( \, - \, n \, | \, \omega \, | \, \delta \, t \, - \, n \, c_3 \, \delta^{\, 3} \, t^{\, 3} \, + \, 3 \, n \, c_3 \, \delta^{\, 3} \, (1-t)^{\, 2} \, t \, \Big) \, \times \\
\exp \, \Big( \, - \, n \, c_4 \, \delta^{\, 4} \, (1-t)^{\, 4} \, + \, 6 \, n \, c_4 \, \delta^{\, 4} \, (1-t)^{\, 2} \, t^{\, 2} \, - \, n \, c_4 \, \delta^{\, 4} \, t^{\, 4} \, \Big) \, {\rm d}t \, .
\end{multline*}
Using $| \, \omega \, | \ge \omega_0$ and \eqref{restrictionsprop5}, we have:
$$
3 \, n \, c_3 \, \delta^{\, 3} \, (1-t)^{\, 2} \, t \, \le \, \dfrac{\omega_0}{3} \, \delta \, t \, ,\qquad \text{\rm and} \qquad
6 \, n \, c_4 \, \delta^{\, 4} \, (1-t)^{\, 2} \, t^{\, 2} \, \le \, \dfrac{\omega_0}{3} \, \delta \, t \, .
$$
We then get:
$$
| \, \mathcal{H}_j^n(1) \, | \, \le \, C \, {\rm e}^{\, n \, C_0 \, \delta^{\, 5}} \, \int_0^1 \,
\exp \, \Big( \, - \, n \, \dfrac{\omega_0}{3} \, \delta \, t \, - \, n \, c_4 \, \delta^{\, 4} \, (1-t)^{\, 4} \, \Big) \, {\rm d}t \, .
$$
With the restriction \eqref{restrictionsprop5}, we can see that the function:
$$
t \in [0,1] \, \longmapsto \, \dfrac{\omega_0}{3} \, \delta \, t \, + \, c_4 \, \delta^{\, 4} \, (1-t)^{\, 4} \, ,
$$
achieves its minimum at $t=0$, and we thus get:
$$
| \, \mathcal{H}_j^n(1) \, | \, \le \, C \, {\rm e}^{\, n \, C_0 \, \delta^{\, 5}} \, {\rm e}^{\, - \, n \, c_4 \, \delta^{\, 4}} \, \le \,
C \, {\rm e}^{\, - \, n \, \frac{c_4}{2} \, \delta^{\, 4}} \, ,
$$
where we used \eqref{restrictionsprop5} one last time. In the end, we have proved that each term in the decomposition of $\Gjn$ is exponentially small in
the considered regime $-2 \, M \le \omega \le -\omega_0$.
\end{proof}

\subsection{Conclusion}

We now discuss the proof of the bound \eqref{bound2} in Theorem \ref{thm:1}. For $| \, j \, | \le M \, n$, we use the definition \eqref{defprincipal} of $\GGjn$
to get:
$$
| \, \GGjn \, | \, \le \, \dfrac{1}{\pi} \, \exp \left( - \, \dfrac{c_4 \, (j \, - \, \alpha \, n)^{\, 2}}{9 \, c_3^{\, 2} \, n} \, \right) \,
\int_{-\sqrt{\frac{2 \, (M \, + \, |\alpha|)}{3\, c_3}}}^{\sqrt{\frac{2 \, (M \, + \, |\alpha|)}{3\, c_3}}} \, {\rm d}u
\, \le \, C \, \exp \left( - \, \dfrac{c_4 \, (j \, - \, \alpha \, n)^{\, 2}}{9 \, c_3^{\, 2} \, n} \, \right) \, .
$$
In particular, with $\omega_0>0$ given as in Proposition \ref{prop4}, there exist two constants $C>0$ and $c>0$ such that the approximate Green's
function in \eqref{defprincipal} satisfies:
$$
\big| \, \GGjn \, \big| \, \le \, C \, {\rm e}^{- \, c \, n} \, ,
$$
as long as $j$ and $n$ satisfy $-2 \, M \le \omega \le -\omega_0$ (recall that the parameter $\omega$ defined in \eqref{defomega}). Combining with
Proposition \ref{prop5}, we obtain:
$$
\big| \, \Gjn \, - \, \GGjn \, \big| \, \le \, C \, {\rm e}^{- \, c \, n} \, \le \,
\dfrac{C}{n \, | \, \omega \, |} \, \exp \left( - \, c \, n \, | \, \omega \, |^{\, 3/2} \, \right) \, ,
$$
for $-2 \, M \le \omega \le -\omega_0$. Hence the bound \eqref{estimprop4} of Proposition \ref{prop4} holds not only for $-\omega_0 \le \omega \le -n^{-\, 2/3}$
but for the wider interval $-2 \, M \le \omega \le -n^{-\, 2/3}$ (up to choosing suitable constants $C$ and $c$). This proves the validity of \eqref{bound2} for
$-2 \, M \, n \le j \, - \, \alpha \, n \le -n^{\, 1/3}$.

Let us now discuss the case $-n^{\, 1/3} \le j \, - \, \alpha \, n \le 0$. We go back to the definition \eqref{defprincipal} of $\GGjn$ and use the following
estimate from above:
$$
| \, \GGjn \, | \, \le \, \dfrac{1}{\pi} \, \int_{-\sqrt{\frac{2 \, |j \, - \, \alpha \, n|}{3\, c_3 \, n}}}^{\sqrt{\frac{2 \, |j \, - \, \alpha \, n|}{3\, c_3 \, n}}} \,
{\rm e}^{\, - \, \sqrt{3 \, c_3 \, n \, |j \, - \, \alpha \, n|} \, u^{\, 2}} \, {\rm d}u \, \le \, \dfrac{2}{\pi} \, \sqrt{\frac{2 \, |j \, - \, \alpha \, n|}{3\, c_3 \, n}}
\, \le \, \dfrac{C}{n^{\, 1/3}} \, ,
$$
since we have $|j \, - \, \alpha \, n| \le n^{\, 1/3}$. Adding with the uniform estimate of Proposition \eqref{prop1}, we get:
$$
| \, \Gjn \, - \, \GGjn \, | \, \le \, \dfrac{C}{n^{\, 1/3}} \, ,
$$
and this proves the validity of \eqref{bound2} for $-n^{\, 1/3} \le j \, - \, \alpha \, n \le 0$.

We now discuss the remaining values $j \, < \, - \, 2 \, M \, n$ for which the Green's function $\Gjn$ vanishes. For $j \, < \, - \, 2 \, M \, n$, we are in the case
$j \, - \, \alpha \, n  \, < \, - \, M \, n$ and $\Gjn=0$, hence the bound \eqref{bound2} reduces to proving:
\begin{equation}
\label{bound3}
\big| \, \GGjn \, \big| \, \le \, \dfrac{C}{|j \, - \, \alpha \, n|} \, \exp \left( - \, c \, \dfrac{|j \, - \, \alpha \, n|^{\, 3/2}}{n} \, \right) \, ,
\end{equation}
for some suitable constants $C$ and $c$. For $j \, - \, \alpha \, n  \, < \, - \, M \, n$, the definition \eqref{defprincipal} of $\GGjn$ gives:
\begin{align*}
| \, \GGjn \, | \, \le \, \dfrac{1}{\pi} \, \exp \left( - \, \dfrac{c_4 \, (j \, - \, \alpha \, n)^{\, 2}}{9 \, c_3^{\, 2} \, n} \, \right) \,
\int_{\R} \, {\rm e}^{\, - \, \sqrt{3 \, c_3 \, M} \, n \, u^{\, 2}} \, {\rm d}u \, \le \,
\dfrac{C}{\sqrt{n}} \, \exp \left( - \, c \, \dfrac{(j \, - \, \alpha \, n)^{\, 2}}{n} \, \right) \\
\le \, \dfrac{C}{|j \, - \, \alpha \, n|} \, \dfrac{|j \, - \, \alpha \, n|}{\sqrt{n}} \, \exp \left( - \, c \, \dfrac{(j \, - \, \alpha \, n)^{\, 2}}{n} \, \right)
\, \le  \, \dfrac{C}{|j \, - \, \alpha \, n|} \, \exp \left( - \, c \, \dfrac{(j \, - \, \alpha \, n)^{\, 2}}{n} \, \right) \, .
\end{align*}
Hence \eqref{bound3} is valid for $j \, < \, - \, 2 \, M \, n$ since the decay of $\GGjn$ is actually even faster than the right hand side of \eqref{bound3}.
This completes the proof of Theorem \ref{thm:1}.

\section{Consequences}
\label{section4}

This section is devoted to the proof of several consequences of Theorem \ref{thm:1}. We first prove Corollary \ref{coro:1}. We then state and prove two other
consequences, the first of which makes the divergence of $(\Gjn)_{j \in \Z}$ in the $\ell^1$ norm precise. Our last result deals with the action of the Laurent
operator $L_a$ on sequences with bounded variations and gives an alternative proof to one of the main results in \cite{ELR}.

\subsection{On the instability in the maximum norm. I}

We first prove Corollary \ref{coro:1}. In many occurrences below, we shall compare Riemann sums with integrals of nonincreasing functions. Namely, for $a \in \R$
and $f \in L^1((a,+\infty);\R^+)$ a nonincreasing function, for any nondecreasing sequence $(x_k)_{k \in \N}$ with values in $[a,+\infty)$, there holds:
$$
\sum_{k \in \N} \, (x_{k+1} \, - \, x_k) \, f(x_{k+1}) \, \le \, \int_a^{+\infty} \, f(x) \, {\rm d}x \, .
$$

\begin{proof}[Proof of Corollary \ref{coro:1}]
For $n \in \N^*$, we define the integer:
$$
J_n \, := \, \min \, \Big\{ \, j \in \Z \, / \, j \, - \, \alpha \, n \, \ge 1 \, \Big\} \, ,
$$
and our first goal is to show a uniform upper bound for the sum:
$$
\big| \, \mathcal{G}_{J_n-1}^{\, n} \, \big| \, + \, \sum_{j \, \ge \, J_n} \, \big| \, \Gjn \, \big| \, ,
$$
see \eqref{boundcoro1-1}. The uniform bound for the term $\mathcal{G}_{J_n-1}^{\, n}$ follows from Proposition \ref{prop1} (actually, this term is not only bounded
but tends to zero). We now use \eqref{bound1} to obtain (for suitable positive constants $C$ and $c$):
$$
\sum_{j \, \ge \, J_n} \, \big| \, \Gjn \, \big| \, \le \, C \, \sum_{j \, \ge \, J_n} \, (x_{j+1,n} \, - \, x_{j,n}) \, f(x_{j,n}) \, ,
$$
where the function $f$ is defined on $\R^+$ by:
$$
\forall \, x \, > \, 0 \, ,\quad f(x) \, := \, \min \, ( \, 1 \, , \, x^{-1/4} \, ) \, \exp \, \big( \, - \, c \, x^{\, - \, 3/2} \, \big) \, ,
$$
and the sampling sequence $(x_{j,n})_{j \ge J_n}$ is here defined as:
$$
\forall \, j \, \ge \, J_n \, ,\quad x_{j,n} \, := \, \dfrac{j \, - \, \alpha \, n}{n^{\, 1/3}} \, \ge \, \dfrac{1}{n^{\, 1/3}} \, .
$$
We thus get:
$$
\sum_{j \, \ge \, J_n} \, \big| \, \Gjn \, \big| \, \le \, C \, \int_{x_{J_n,n}-n^{\, -1/3}}^{+\infty} \, f(x) \, {\rm d}x \, \le \, C \, \| \, f \, \|_{L^1(\R^+)} \, .
$$
This completes the proof of the validity of \eqref{boundcoro1-1}. The proof of \eqref{boundcoro1-2} follows from \eqref{bound2} in a similar way.
\end{proof}

\subsection{On the instability in the maximum norm. II}

Our result is the following.

\begin{corollary}
\label{coro:2}
Let Assumption \ref{hyp:1} be satisfied with $c_3>0$. Then the Green's function $(\Gjn)_{(n,j) \in \N \times \Z}$ in \eqref{def:green} satisfies:
\begin{equation}
\label{limit}
\lim_{n \, \to \, + \, \infty} \, \, \dfrac{1}{n^{\, 1/8}} \, \, \sum_{j \in \Z} \, \big| \, \Gjn \, \big| \, = \, \dfrac{8 \, \Gamma (11/8)}{\sqrt{3} \, \pi^{\, 3/2}} \,
\dfrac{c_3^{\, 1/2}}{c_4^{\, 3/8}} \, ,
\end{equation}
where we use the standard notation for the $\Gamma$ function:
$$
\forall \, x \, > \, 0 \, ,\quad \Gamma (x) \, := \, \int_0^{+\infty} \, t^{\, x \, - \, 1} \, {\rm e}^{\, - \, t} \, {\rm d}t \, .
$$
\end{corollary}

\noindent As far as we know, the results in \cite{hedstrom1,hedstrom2} only give the existence of two constants $0 \, < \, \ell_1 \, < \, \ell_2$ such that the
following inequalities hold:
$$
\forall \, n \in \N^* \, ,\qquad \ell_1 \, < \, \dfrac{1}{n^{\, 1/8}} \, \, \sum_{j \in \Z} \, \big| \, \Gjn \, \big| \, < \, \ell_2 \, ,
$$
and these results do not seem to have been improved since then. Our result is more precise. We hope that our sharp bounds for the Green's function
$(\Gjn)_{(n,j) \in \N \times \Z}$ will help us get sharp stability results for discrete shock profiles associated with the Lax-Wendroff scheme, see \cite{SmyrlisYu}
for a preliminary analysis. The sharp description \eqref{limit} is a first step in this direction and a more general study is left to a future work.

At a formal level, the result of Corollary \ref{coro:2} is consistent with \cite{Thomee} since when $c_3$ vanishes in \eqref{hyp:stabilite2} and $c_4$ remains
positive, the main result of \cite{Thomee} shows that $(\Gjn)_{j \in \Z}$ is bounded in $\ell^1(\Z;\C)$. Hence the limit \eqref{limit} should be zero when $c_3$
vanishes, which is consistent with \eqref{limit}.

\begin{proof}[Proof of Corollary \ref{coro:2}]
We first apply Corollary \ref{coro:1} and get:
$$
\lim_{n \, \to \, + \, \infty} \, \, \dfrac{1}{n^{\, 1/8}} \, \, \sum_{j \, / \, j \, - \, \alpha \, n \, \ge \, 0} \, \big| \, \Gjn \, \big| \, = \,
\lim_{n \, \to \, + \, \infty} \, \, \dfrac{1}{n^{\, 1/8}} \, \, \sum_{j \, / \, j \, - \, \alpha \, n \, < \, 0} \, \big| \, \Gjn \, - \, \GGjn \, \big| \, = \, 0 \, .
$$
From now on, we let $\ell$ denote the value of the limit in \eqref{limit}, that is:
\begin{equation}
\label{defl}
\ell \, := \, \dfrac{8 \, \Gamma (11/8)}{\sqrt{3} \, \pi^{\, 3/2}} \, \dfrac{c_3^{\, 1/2}}{c_4^{\, 3/8}} \, .
\end{equation}
In order to show \eqref{limit}, it is therefore sufficient to prove:
\begin{equation}
\label{limit1}
\lim_{n \, \to \, + \, \infty} \, \, \dfrac{1}{n^{\, 1/8}} \, \, \sum_{j \, / \, j \, - \, \alpha \, n \, < \, 0} \, \big| \, \GGjn \, \big| \, = \, \ell \, ,
\end{equation}
where we recall that the approximate Green's function $(\GGjn)_{j \in \Z}$ is defined in \eqref{defprincipal}. In what follows, we are going
to prove that \eqref{limit1} holds, where the limit $\ell$ is defined in \eqref{defl}.
\bigskip

For $n \in \N^*$, we now define:
$$
J_n \, := \, \max \, \Big\{ \, j \in \Z \, / \, j \, - \, \alpha \, n \, < \, - \, 3 \, \Big\} \, .
$$
From the definition \eqref{defprincipal}, we get the trivial bound (just bound the absolute value of the cosine and the exponentials by $1$):
$$
\big| \, \GGjn \, \big| \, \le \, \dfrac{2}{\pi} \, \sqrt{\dfrac{2 \, | \, j \, - \, \alpha \, n \, |}{3 \, c_3 \, n}} \, ,
$$
which implies:
$$
\big| \, \mathbb{G}_{J_n+1}^{\, n} \, \big| \, + \, \big| \, \mathbb{G}_{J_n+2}^{\, n} \, \big| \, + \, \big| \, \mathbb{G}_{J_n+3}^{\, n} \, \big| \,
\le \, \dfrac{C}{\sqrt{n}} \, ,
$$
for some suitable constant $C>0$ that does not depend on $n \in \N^*$. Hence proving \eqref{limit1} amounts to proving:
\begin{equation}
\label{limit2}
\lim_{n \, \to \, + \, \infty} \, \, \dfrac{1}{n^{\, 1/8}} \, \, \sum_{j \, \le \, J_n} \, \big| \, \GGjn \, \big| \, = \, \ell \, .
\end{equation}

For $n \in \N^*$ and $j \le J_n$, we define the sampling point:
\begin{equation}
\label{defyjn}
y_{j,n} \, := \, \dfrac{| \, j \, - \, \alpha \, n \, |}{\sqrt{n}} \, > \, \dfrac{3}{\sqrt{n}} \, .
\end{equation}
In order to make the reading of some computations easier, we also define the positive parameters:
\begin{equation}
\label{defbeta01}
\beta_0 \, := \, \dfrac{c_4}{9 \, c_3^{\, 2}} \, ,\quad \beta_1 \, := \, \dfrac{2}{3 \, \sqrt{3 \, c_3}} \, .
\end{equation}
Then performing a change of variable in the integral in \eqref{defprincipal}, we obtain the formula:
\begin{multline}
\label{formuleyjn}
\forall \, n \in \N^* \, ,\quad \forall \, j \le J_n \, ,\quad \dfrac{1}{n^{\, 1/8}} \, \GGjn \, = \, \dfrac{\sqrt{3 \, \beta_1}}{\pi \, \sqrt{2}} \, \dfrac{1}{\sqrt{n}} \,
\dfrac{1}{y_{j,n}^{\, 1/4}} \, \exp \left( - \, \beta_0 \, y_{j,n}^{\, 2} \, \right) \, \times \\
\cos \left( \beta_1 \, n^{\, 1/4} \, y_{j,n}^{\, 3/2} \, - \, \dfrac{\pi}{4} \right) \,
\int_{-\sqrt{3 \, \beta_1} \, n^{\, 1/8} \, y_{j,n}^{\, 3/4}}^{\sqrt{3 \, \beta_1} \, n^{\, 1/8} \, y_{j,n}^{\, 3/4}} \, {\rm e}^{\, - \, u^{\, 2}} \, {\rm d}u \, .
\end{multline}
The formula \eqref{formuleyjn} clarifies the role of the scale $n^{\, -1/2}$ that corresponds to the step in a Riemann sum. Our first task is to
simplify the formula \eqref{formuleyjn} by approximating the integral of the Gaussian function by $\sqrt{\pi}$ (which corresponds to taking the
limit $n \to +\infty$ for any fixed $y_{j,n}>0$). A precise statement is the following.

\begin{lemma}
\label{lemA1}
With the sampling points $y_{j,n}$ given in \eqref{defyjn}, let us define the sequence:
$$
\forall \, n \in \N^* \, ,\quad \forall \, j \le J_n \, ,\quad H_j^n \, := \, \sqrt{\dfrac{3 \, \beta_1}{2 \, \pi}} \, \dfrac{1}{\sqrt{n}} \,
\dfrac{1}{y_{j,n}^{\, 1/4}} \, \exp \left( - \, \beta_0 \, y_{j,n}^{\, 2} \, \right) \,
\cos \left( \beta_1 \, n^{\, 1/4} \, y_{j,n}^{\, 3/2} \, - \, \dfrac{\pi}{4} \right) \, .
$$
Then there holds:
\begin{equation}
\label{limitlemA1}
\lim_{n \, \to \, + \, \infty} \, \sum_{j \le J_n} \, \Big| \, \dfrac{1}{n^{\, 1/8}} \, \GGjn \, - \, H_j^n \, \Big| \, = \, 0 \, .
\end{equation}
Consequently, \eqref{limit2} holds if and only if:
\begin{equation}
\label{limit3}
\lim_{n \, \to \, + \, \infty} \, \, \sum_{j \, \le \, J_n} \, \big| \, H_j^n \, \big| \, = \, \ell \, .
\end{equation}
\end{lemma}

\begin{proof}[Proof of Lemma \ref{lemA1}]
For later use, we introduce the positive constant:
$$
\frak{c} \, := \, \int_0^{+\infty} \, y^{\, - \, 1/4} \, {\rm e}^{\, - \, \beta_0 \, y^{\, 2}} \, {\rm d}y \, .
$$
We consider $\varepsilon>0$ and fix some positive real number $M_\varepsilon>0$ such that:
$$
2 \, \int_{M_\varepsilon}^{+\infty} \, {\rm e}^{\, - \, u^{\, 2}} \, {\rm d}u \, \le \, \dfrac{\varepsilon}{2 \, \frak{c}} \, \dfrac{\pi \, \sqrt{2}}{\sqrt{3 \, \beta_1}} \, .
$$
For $n \in \N^*$ and $j \le J_n$ that satisfies:
$$
\sqrt{3 \, \beta_1} \,  \, n^{\, 1/8} \, y_{j,n}^{\, 3/4} \, \ge \, M_\varepsilon \, ,
$$
we thus have:
\begin{align*}
\Big| \, \dfrac{1}{n^{\, 1/8}} \, \GGjn \, - \, H_j^n \, \Big| \, & \le \, \dfrac{\sqrt{3 \, \beta_1}}{\pi \, \sqrt{2}} \, \dfrac{1}{\sqrt{n}} \,
\dfrac{1}{y_{j,n}^{\, 1/4}} \, \exp \left( - \, \beta_0 \, y_{j,n}^{\, 2} \, \right) \times 2 \, \int_{M_\varepsilon}^{+\infty} \, {\rm e}^{\, - \, u^{\, 2}} \, {\rm d}u \\
& \le \, \dfrac{\varepsilon}{2 \, \frak{c}} \, \dfrac{1}{\sqrt{n}} \, \dfrac{1}{y_{j,n}^{\, 1/4}} \, \exp \left( - \, \beta_0 \, y_{j,n}^{\, 2} \, \right) \, .
\end{align*}
Summing with respect to $j$, we obtain (here we use again a comparison principle between a Riemann sum and the integral of a nonincreasing function):
$$
\sum_{j \in \mathcal{J}_{\varepsilon,n}} \, \Big| \, \dfrac{1}{n^{\, 1/8}} \, \GGjn \, - \, H_j^n \, \Big| \, \le \, \dfrac{\varepsilon}{2} \, ,
$$
where the set of indices $\mathcal{J}_{\varepsilon,n}$ is defined as:
$$
\mathcal{J}_{\varepsilon,n} \, := \, \Big\{ \, j \, \le \, J_n \, / \, \sqrt{3 \, \beta_1} \,  \, n^{\, 1/8} \, y_{j,n}^{\, 3/4} \, \ge \, M_\varepsilon \, \Big\} \, .
$$
It remains to bound from above the remaining sum:
$$
\sum_{j \not \in \mathcal{J}_{\varepsilon,n}} \, \Big| \, \dfrac{1}{n^{\, 1/8}} \, \GGjn \, - \, H_j^n \, \Big| \, ,
$$
and we observe that the condition $j \not \in \mathcal{J}_{\varepsilon,n}$ is equivalent to requiring:
$$
y_{j,n} \, < \, \left( \dfrac{M_\varepsilon}{\sqrt{3 \, \beta_1}} \right)^{\, 4/3} \, n^{\, - \, 1/6} \, .
$$
Applying the triangle inequality, we thus obtain:
\begin{align*}
\sum_{j \not \in \mathcal{J}_{\varepsilon,n}} \, \Big| \, \dfrac{1}{n^{\, 1/8}} \, \GGjn \, - \, H_j^n \, \Big| \, & \le \,
\sum_{j \not \in \mathcal{J}_{\varepsilon,n}} \, \Big| \, \dfrac{1}{n^{\, 1/8}} \, \GGjn \, \Big| \, + \, \Big| \, H_j^n \, \Big| \\
& \le \, 2 \, \sqrt{\dfrac{3 \, \beta_1}{2 \, \pi}} \, \dfrac{1}{\sqrt{n}} \, \dfrac{1}{y_{j,n}^{\, 1/4}} \, \le \,
2 \, \sqrt{\dfrac{3 \, \beta_1}{2 \, \pi}} \, \int_0^{\left( \frac{M_\varepsilon}{\sqrt{3 \, \beta_1}} \right)^{\, 4/3} \, n^{\, - \, 1/6}} \, \dfrac{{\rm d}y}{y^{\, 1/4}} \, .
\end{align*}
Choosing $n \ge N_\varepsilon$ for some appropriate $N_\varepsilon \in \N^*$, we thus obtain:
$$
\sum_{j \not \in \mathcal{J}_{\varepsilon,n}} \, \Big| \, \dfrac{1}{n^{\, 1/8}} \, \GGjn \, - \, H_j^n \, \Big| \, \le \, \dfrac{\varepsilon}{2} \, ,
$$
and the claim \eqref{limitlemA1} follows.

The conclusion of Lemma \ref{lemA1} on the equivalence between \eqref{limit2} and \eqref{limit3} follows from the triangle inequality.
\end{proof}

\noindent In view of Lemma \ref{lemA1}, we now wish to prove that \eqref{limit3} holds, where the sequence $(H_j^n)_{j \le J_n}$ is defined in
Lemma \ref{lemA1}. A final simplification amounts to substituting a discrete sum by an integral, which will allow us to perform change of variables
and other algebraic operations more easily. A precise statement is the following.

\begin{lemma}
\label{lemA2}
With the sampling points $y_{j,n}$ given in \eqref{defyjn} and the sequence $(H_j^n)_{j \le J_n}$ defined in Lemma \ref{lemA1}, there holds:
\begin{equation}
\label{limitlemA2}
\lim_{n \, \to \, + \, \infty} \, \left( \sum_{j \le J_n} \, \Big| \, H_j^n \, \Big| \, - \, \sqrt{\dfrac{3 \, \beta_1}{2 \, \pi}} \, \int_0^{+\infty} \,
\dfrac{{\rm e}^{\, - \, \beta_0 \, y^{\, 2}}}{y^{\, 1/4}} \, \Big| \, \cos \Big( \, \beta_1 \, n^{\, 1/4} \, y^{\, 3/2} \, - \, \dfrac{\pi}{4} \Big) \, \Big| \, {\rm d}y
\right) = \, 0 \, .
\end{equation}
Consequently, \eqref{limit3} holds if and only if:
\begin{equation}
\label{limit4}
\lim_{n \, \to \, + \, \infty} \, \sqrt{\dfrac{3 \, \beta_1}{2 \, \pi}} \, \int_0^{+\infty} \, \dfrac{{\rm e}^{\, - \, \beta_0 \, y^{\, 2}}}{y^{\, 1/4}} \,
\Big| \, \cos \Big( \, \beta_1 \, n^{\, 1/4} \, y^{\, 3/2} \, - \, \dfrac{\pi}{4} \Big) \, \Big| \, {\rm d}y \, = \, \ell \, .
\end{equation}
\end{lemma}

Let us assume for a moment that the conclusion of Lemma \ref{lemA2} holds and let us complete the proof of Corollary \ref{coro:2}. For ease of reading,
we define the sequence:
$$
\forall \, n \in \N^* \, ,\quad \mathbb{I}_n \, := \, \sqrt{\dfrac{3 \, \beta_1}{2 \, \pi}} \, \int_0^{+\infty} \, \dfrac{{\rm e}^{\, - \, \beta_0 \, y^{\, 2}}}{y^{\, 1/4}} \,
\Big| \, \cos \Big( \, \beta_1 \, n^{\, 1/4} \, y^{\, 3/2} \, - \, \dfrac{\pi}{4} \Big) \, \Big| \, {\rm d}y \, ,
$$
of which we aim at computing the limit. We perform a first change of variable $x = \beta_1 \, y^{\, 3/2}$ in the integral and obtain the relation:
$$
\mathbb{I}_n \, = \, \sqrt{\dfrac{2}{3 \, \pi}} \, \int_0^{+\infty} \, \dfrac{1}{x^{\, 1/2}} \, \exp \left( \, - \, \dfrac{\beta_0}{\beta_1^{\, 4/3}} \, x^{\, 4/3} \, \right) \,
\Big| \, \cos \Big( \, n^{\, 1/4} \, x \, - \, \dfrac{\pi}{4} \Big) \, \Big| \, {\rm d}x \, ,
$$
and we now use without proof the following classical result\footnote{The proof merely consists in approximating any $L^1$ function by a smooth compactly
supported function, then using the Fourier series expansion of $|\cos (u-\pi/4)|$ and integrating by parts.}:

\begin{lemma}
\label{lemA3}
Let $F \in L^1(\R;\C)$. Then there holds:
$$
\lim_{n \to +\infty} \, \int_\R \, F(x) \, \Big| \, \cos \Big( \, n^{\, 1/4} \, x \, - \, \dfrac{\pi}{4} \Big) \, \Big| \, {\rm d}x \, = \, \dfrac{2}{\pi} \, \int_\R \, F(x) \, {\rm d}x \, .
$$
\end{lemma}

We therefore obtain:
$$
\lim_{n \to +\infty} \, \mathbb{I}_n \, = \, \sqrt{\dfrac{2}{3 \, \pi}} \, \dfrac{2}{\pi} \, \int_0^{+\infty} \,
\dfrac{1}{x^{\, 1/2}} \, \exp \left( \, - \, \dfrac{\beta_0}{\beta_1^{\, 4/3}} \, x^{\, 4/3} \, \right) \, {\rm d}x \, ,
$$
and a final change of variable in the integral yields:
$$
\lim_{n \to +\infty} \, \mathbb{I}_n \, = \, \left( \dfrac{2}{\pi} \right)^{\, 3/2} \, \dfrac{1}{\sqrt{3}} \, \dfrac{\beta_1^{\, 1/2}}{\beta_0^{\, 3/8}} \, \dfrac{3}{4} \,
\Gamma (3/8) \, .
$$
Using the definition \eqref{defbeta01} of $\beta_0$ and $\beta_1$ and using the functional relation $x \, \Gamma(x) \, = \, \Gamma (x+1)$, we end up
proving (recall the definition \eqref{defl}):
$$
\lim_{n \to +\infty} \, \mathbb{I}_n \, = \, \ell \, ,
$$
which means that \eqref{limit4} holds. Going back all the way up, we have thus proved that \eqref{limit} holds. At this stage, it just remains to prove Lemma
\ref{lemA2}.

\begin{proof}[Proof of Lemma \ref{lemA2}]
The multiplicative factor $\sqrt{(3 \, \beta_1)/(2\, \pi)}$ in the definition of $H_j^n$ and in front of the integral in \eqref{limitlemA2} is harmless so we omit it from
now on. We let $\varepsilon>0$ and consider two positive numbers $A_\varepsilon,B_\varepsilon$ that satisfy:
$$
\int_0^{A_\varepsilon} \, \dfrac{{\rm e}^{\, - \, \beta_0 \, y^{\, 2}}}{y^{\, 1/4}} \, {\rm d}y \, + \,
\int_{B_\varepsilon}^{+\infty} \, \dfrac{{\rm e}^{\, - \, \beta_0 \, y^{\, 2}}}{y^{\, 1/4}} \, {\rm d}y \, \le \, \dfrac{\varepsilon}{3} \, .
$$
By comparing the Riemann sums below with integrals, we therefore also get:
$$
\sum_{j \le J_n \, / \, y_{j,n} \le A_\varepsilon} \, \dfrac{1}{\sqrt{n}} \, y_{j,n}^{\, - \, 1/4} \, {\rm e}^{\, - \, \beta_0 \, y_{j,n}^{\, 2}} \, + \,
\sum_{j \le J_n \, / \, y_{j,n} \ge B_\varepsilon+n^{\, -\, 1/2}} \, \dfrac{1}{\sqrt{n}} \, y_{j,n}^{\, - \, 1/4} \, {\rm e}^{\, - \, \beta_0 \, y_{j,n}^{\, 2}}
\, \le \, \dfrac{\varepsilon}{3} \, .
$$
In particular, we get:
$$
\sum_{j \le J_n \, / \, y_{j,n} \le A_\varepsilon} \, \big| \, H_j^n \, \big| \, + \,
\sum_{j \le J_n \, / \, y_{j,n} \ge B_\varepsilon+n^{\, -\, 1/2}} \, \big| \, H_j^n \, \big| \, \le \, \dfrac{\varepsilon}{3} \, .
$$
We therefore focus on the difference:
$$
\sum_{j \le J_n \, / \, A_\varepsilon < y_{j,n}<B_\varepsilon+n^{\, -\, 1/2}} \, \Big| \, H_j^n \, \Big| \, - \, \int_{A_\varepsilon}^{B_\varepsilon} \,
\dfrac{{\rm e}^{\, - \, \beta_0 \, y^{\, 2}}}{y^{\, 1/4}} \, \Big| \, \cos \Big( \, \beta_1 \, n^{\, 1/4} \, y^{\, 3/2} \, - \, \dfrac{\pi}{4} \Big) \, \Big| \, {\rm d}y \, ,
$$
since the remaining parts on the left hand side of \eqref{limitlemA2} are estimated (in absolute value) by $2\, \varepsilon/3$. We recall that throughout this
proof we omit the multiplicative factor $\sqrt{(3 \, \beta_1)/(2\, \pi)}$ in the definition of $H_j^n$.

We introduce the function:
$$
f_n \, : \, y \in (0,+\infty) \, \longmapsto \, \dfrac{{\rm e}^{\, - \, \beta_0 \, y^{\, 2}}}{y^{\, 1/4}} \,
\Big| \, \cos \Big( \, \beta_1 \, n^{\, 1/4} \, y^{\, 3/2} \, - \, \dfrac{\pi}{4} \Big) \, \Big| \, .
$$
Then $f_n$ is Lipshitzean on the interval $[A_\varepsilon,B_\varepsilon+2]$ and, more precisely, there exists a constant $C_\varepsilon>0$ that depends on
$\varepsilon$ but not on $n \in \N^*$, and such that:
$$
\forall \, x,y \in [A_\varepsilon,B_\varepsilon+2] \, ,\qquad \big| \, f_n(x) \, - \, f_n(y) \, \big| \, \le \, C_\varepsilon \, n^{\, 1/4} \, | \, x \, - \, y \, | \, .
$$
We thus get the bound:
\begin{equation}
\label{ineglemA2}
\Big| \,  | \, H_j^n \, | \, - \, \int_{y_{j,n}}^{y_{j-1,n}} \, f_n(y) \, {\rm d}y \, \Big| \, = \, \Big| \, \int_{y_{j,n}}^{y_{j-1,n}} \, f_n(y_{j,n}) \, - \, f_n(y) \, {\rm d}y \, \Big| \, \le \,
C_\varepsilon \, n^{\, - \, 3/4} \, ,
\end{equation}
as long as $y_{j,n}$ satisfies:
\begin{equation}
\label{conditionyjn}
A_\varepsilon \, < \, y_{j,n} \, < \, B_\varepsilon \, + \, n^{\, -\, 1/2} \, .
\end{equation}
Of course, the constant $C_\varepsilon$ in \eqref{ineglemA2} is independent of $n \in \N^*$. Summing \eqref{ineglemA2} over $j$, we get:
$$
\left| \, \sum_{j \le J_n \, / \, A_\varepsilon < y_{j,n}<B_\varepsilon+n^{\, -\, 1/2}} \, \Big| \, H_j^n \, \Big| \, - \,
\int_{y_{J_{\max},n}}^{y_{J_{\min}-1,n}} \, f_n(y) \, {\rm d}y \, \right| \, \le \, C_\varepsilon \, n^{\, - \, 1/4} \, ,
$$
where the interval of indices $[J_{\min},J_{\max}] \cap \Z$ corresponds to the set of integers $j$ such that \eqref{conditionyjn} is satisfied. It is then easy to
verify that the quantity:
$$
\int_{A_\varepsilon}^{B_\varepsilon} \, f_n(y) \, {\rm d}y \, - \, \int_{y_{J_{\max},n}}^{y_{J_{\min}-1,n}} \, f_n(y) \, {\rm d}y
$$
tends to zero as $n$ tends to infinity so overall, we can fix an integer $N_\varepsilon \in \N^*$ such that for any $n \ge N_\varepsilon$, there holds:
$$
\left| \, \sum_{j \le J_n \, / \, A_\varepsilon < y_{j,n}<B_\varepsilon+n^{\, -\, 1/2}} \, \Big| \, H_j^n \, \Big| \, - \,
\int_{A_\varepsilon}^{B_\varepsilon} \, f_n(y) \, {\rm d}y \, \right| \, \le \, \dfrac{\varepsilon}{3} \, .
$$
This gives the final bound:
$$
\left| \, \sum_{j \le J_n} \, \Big| \, H_j^n \, \Big| \, - \, \int_0^{+\infty} \, f_n(y) \, {\rm d}y \, \right| \, \le \, \varepsilon \, ,
$$
for $n \ge N_\varepsilon$ so Lemma \ref{lemA2} is proved.
\end{proof}
\end{proof}

Let us see what Corollary \ref{coro:2} gives in the case of the Lax-Wendroff scheme \eqref{LW}. Recalling \eqref{TaylorLW}, we have:
$$
c_3 \, = \, \dfrac{\lambda \, (1 \, - \, \lambda^{\, 2})}{6} \, ,\qquad c_4 \, = \, \dfrac{\lambda^{\, 2} \, (1 \, - \, \lambda^{\, 2})}{8} \, ,
$$
which means that for any $\lambda \in (0,1)$, the Green's function of the Lax-Wendroff scheme \eqref{LW} satisfies:
$$
\lim_{n \, \to \, + \, \infty} \, \, \dfrac{1}{n^{\, 1/8}} \, \, \sum_{j \in \Z} \, \big| \, \Gjn \, \big| \, = \, \dfrac{2^{\, 5/8} \, \Gamma (3/8)}{\pi^{\, 3/2}} \,
\dfrac{(1 \, - \, \lambda^{\, 2})^{\, 1/8}}{\lambda^{\, 1/4}} \, .
$$

\subsection{Uniform bounds for initial data of bounded variations}

In Corollary \ref{coro:3} below, we let $\text{\rm BV}(\Z;\C)$ denote the space of complex valued sequences that have bounded variations, that is
$u \in \text{\rm BV}(\Z;\C)$ if the quantity:
$$
\sum_{j \in \Z} \, | \, u_{j+1} \, - \, u_j \, |
$$
is finite. We use without proof that any sequence with bounded variations has a finite limit at $-\infty$ (the same property holds at $+\infty$). Our result
dates back to \cite{ELR} but our proof differs from the one in that reference since we have an accurate description of the Green's function at our disposal.

\begin{corollary}[Estep-Loss-Rauch]
\label{coro:3}
Let $a \in \ell^{\, 1}(\Z;\C)$ satisfy Assumption \ref{hyp:1}. Then there exists a constant $C>0$ such that for any sequence $u \in \text{\rm BV}(\Z;\C)$
that satisfies $\lim_{j \to -\infty} u_j=0$, the Laurent operator $L_a$ satisfies:
$$
\sup_{n \in \N} \, \, \Big\| \, L_a^{\, n} \, u \, \Big\|_{\ell^{\, \infty}} \, \le \, C \, \sum_{j \in \Z} \, | \, u_{j+1} \, - \, u_j \, | \, .
$$
\end{corollary}

\noindent Corollary \ref{coro:3} explains why on Figure \ref{fig:dispersion} the oscillating wave packets generated by the step function remain bounded in
the $\ell^{\, \infty}$ norm. For general $\text{\rm BV}$ initial data that have nonzero limit at $-\infty$, one should apply Corollary \ref{coro:3} to $u \, - \,
\lim_{-\infty}u$ (we recall that $L_a$ maps any constant sequence $v$ to itself since the $a_j$'s sum to $1$).

\begin{proof}[Proof of Corollary \ref{coro:3}]
The key point in the proof of Corollary \ref{coro:3} is the following bound which shows that the oscillations in the Green's function cancel after integration.

\begin{lemma}
\label{lemA4}
Let Assumption \ref{hyp:1} be satisfied. Then there exists a constant $C>0$ such that the Green's function $(\Gjn)_{(n,j) \in \N \times \Z}$ satisfies the uniform
bound:
\begin{equation}
\label{boundintegral}
\forall \, n \in \N^* \, ,\quad \forall \, j \in \Z \, ,\quad \left| \, \sum_{\ell \le j} \, \mathcal{G}_\ell^{\, n} \, \right| \, \le \, C \, .
\end{equation}
\end{lemma}

Let us assume for a moment that the claim of Lemma \ref{lemA4} holds and let us show that the bound \eqref{boundintegral} implies Corollary \ref{coro:3}.
We introduce the Heaviside sequence $\mathbb{H}$ defined by:
$$
\forall \, j \in \Z \, ,\quad \mathbb{H}_j \, := \, \begin{cases}
1 \, ,&\text{\rm if $j \, \ge \, 0$,}\\
0 \, ,&\text{\rm if $j \, < \, 0$,}
\end{cases}
$$
and for any $\ell \in \Z$, the notation $\mathbb{H}(\cdot \, - \, \ell)$ stands for the shift of the sequence $\mathbb{H}$ by an index $\ell$, namely:
$$
\forall \, j \in \Z \, ,\quad  \left( \mathbb{H}(\cdot \, - \, \ell) \right)_j \, := \, \mathbb{H}_{j-\ell} \, .
$$
We use the trick from \cite{ELR} to decompose any sequence $u \in \text{\rm BV}(\Z;\C)$ that satisfies $\lim_{j \to -\infty} u_j=0$ under the form:
$$
u \, = \, \sum_{\ell \in \Z} \, (u_\ell \, - \, u_{\ell-1}) \, \mathbb{H}(\cdot \, - \, \ell) \, ,
$$
which gives:
$$
\forall \, n \in \N^* \, ,\quad L_a^{\, n} \, u \, = \, \sum_{\ell \in \Z} \, (u_\ell \, - \, u_{\ell-1}) \, \Big( L_a^{\, n} \, \mathbb{H}(\cdot \, - \, \ell) \Big) \, .
$$
Since $L_a^{\, n}$ is the Laurent operator associated with the sequence $(\Gjn)_{j \in \Z}$, we have:
$$
\forall \, \ell \in \Z \, ,\quad L_a^{\, n} \, \mathbb{H}(\cdot \, - \, \ell) \, = \, \Big( L_a^{\, n} \, \mathbb{H} \Big) (\cdot \, - \, \ell) \, .
$$
Applying the triangle inequality, we see that Corollary \ref{coro:3} will follow from the uniform bound:
\begin{equation}
\label{boundcoro3}
\sup_{n \in \N} \, \Big\| \, L_a^{\, n} \, \mathbb{H} \, \Big\|_{\ell^{\, \infty}} \, < \, + \infty \, .
\end{equation}
From the expression of the Heaviside sequence, we compute:
$$
\forall \, j \in \Z \, ,\quad \Big( L_a^{\, n} \, \mathbb{H} \Big)_j \, = \, \sum_{\ell \le j} \, \mathcal{G}_\ell^{\, n} \, ,
$$
so Lemma \ref{lemA4} implies the validity of \eqref{boundcoro3} and the claim of Corollary \ref{coro:3} follows. We thus focus from now on on the
proof of Lemma \ref{lemA4}.

\begin{proof}[Proof of Lemma \ref{lemA4}]
Applying Corollary \ref{coro:1}, we see that the bound \eqref{boundintegral} amounts to proving that there exists a constant $C>0$ such that the analogous
bound for the approximate Green's function holds, namely:
\begin{equation}
\label{boundintegral'}
\forall \, n \in \N^* \, ,\quad \forall \, j \, < \, \alpha \, n \, ,\quad \left| \, \sum_{\ell \le j} \, \mathbb{G}_\ell^n \, \right| \, \le \, C \, ,
\end{equation}
where we recall that the expression of $\GGjn$ is given in \eqref{defprincipal}. We keep the notation of the proof of Corollary \ref{coro:2} for the integer $J_n$
and for the sampling points $y_{j,n}$ in \eqref{defyjn}. We also keep the definition \eqref{defbeta01} for the constants $\beta_0$ and $\beta_1$. We then use
a change of variable in the integral of \eqref{defprincipal} to obtain:
\begin{multline}
\label{formule'yjn}
\forall \, n \in \N^* \, ,\quad \forall \, j \le J_n \, ,\quad \GGjn \, = \, \dfrac{3 \, \beta_1}{\pi \, \sqrt{2}} \, \dfrac{1}{n^{\, 1/4}} \, y_{j,n}^{\, 1/2} \,
\exp \left( - \, \beta_0 \, y_{j,n}^{\, 2} \, \right) \, \times \\
\cos \left( \beta_1 \, n^{\, 1/4} \, y_{j,n}^{\, 3/2} \, - \, \dfrac{\pi}{4} \right) \,
\int_{-1}^1 \, \exp \left( \, - \, 3 \, \beta_1 \, n^{\, 1/4} \, y_{j,n}^{\, 3/2} \, u^{\, 2} \right) \, {\rm d}u \, .
\end{multline}

We introduce a function $g$ that is defined a follows:
$$
\forall \, x \, > \, 0 \, ,\quad g(x) \, := \, \cos \left( \beta_1 \, x \, - \, \dfrac{\pi}{4} \right) \, \int_{-1}^1 \, \exp \left( \, - \, 3 \, \beta_1 \, x \, u^{\, 2} \right) \, {\rm d}u \, ,
$$
which puts the formula \eqref{formule'yjn} in the more compact form:
\begin{equation}
\label{formule''yjn}
\forall \, n \in \N^* \, ,\quad \forall \, j \le J_n \, ,\quad \GGjn \, = \, \dfrac{\beta_1 \, \sqrt{2}}{\pi} \, \exp \left( - \, \beta_0 \, y_{j,n}^{\, 2} \, \right) \,
(y_{j,n} \, - \, y_{j+1,n}) \, \dfrac{3}{2} \, n^{\, 1/4} \, y_{j,n}^{\, 1/2} \, g \big( n^{\, 1/4} \, y_{j,n}^{\, 3/2} \big) \, .
\end{equation}
We introduce the primitive function $\tilde{g}$ of $g$:
$$
\forall \, x \, > \, 0 \, ,\quad \widetilde{g}(x) \, := \, \int_0^x \, g(y) \, {\rm d}y \, ,
$$
and we shall use without proof that $\widetilde{g}$ is bounded on $\R^+$ (the proof of this property is left to the reader).

Our goal is to apply an Abel transformation to the series of the $\GGjn$. This is suggested by the form \eqref{formule''yjn} where the factor:
$$
(y_{j,n} \, - \, y_{j+1,n}) \, \dfrac{3}{2} \, n^{\, 1/4} \, y_{j,n}^{\, 1/2} \, g \big( n^{\, 1/4} \, y_{j,n}^{\, 3/2} \big)
$$
arises as some kind of derivative. Based on that goal to achieve, we decompose $\GGjn$ under the form:
$$
\GGjn \, = \, \dfrac{\beta_1 \, \sqrt{2}}{\pi} \, \exp \left( - \, \beta_0 \, y_{j,n}^{\, 2} \, \right) \, \Big( \widetilde{g} \big( n^{\, 1/4} \, y_{j,n}^{\, 3/2} \big)
\, - \, \widetilde{g} \big( n^{\, 1/4} \, y_{j+1,n}^{\, 3/2} \big) \Big) \, - \, \varepsilon_{j,n} \, ,
$$
and the Abel transformation (or discrete integration by parts) directly gives the uniform bound:
$$
\sup_{n \in \N^*} \, \sup_{j \, \le \, J_n} \quad \left| \, \sum_{\ell \le j} \, \dfrac{\beta_1 \, \sqrt{2}}{\pi} \, \exp \left( - \, \beta_0 \, y_{\ell,n}^{\, 2} \, \right)
\, \Big( \widetilde{g} \big( n^{\, 1/4} \, y_{\ell,n}^{\, 3/2} \big) \, - \, \widetilde{g} \big( n^{\, 1/4} \, y_{\ell+1,n}^{\, 3/2} \big) \Big) \, \right| \, < \, + \infty \, .
$$
We thus focus on the bound for the remainder $\varepsilon_{j,n}$, and use Taylor's formula to get:
$$
\Big| \, \varepsilon_{j,n} \, \Big| \, \le \, C \, \left( \dfrac{y_{j,n}}{\sqrt{n}} \, + \, \dfrac{1}{n^{\, 3/4} \, y_{j+1,n}^{\, 1/2}} \right) \,
\exp \left( - \, \beta_0 \, y_{j,n}^{\, 2} \, \right) \, ,
$$
which, after summation, gives the uniform bound:
$$
\sup_{n \in \N^*} \, \sup_{j \, \le \, J_n} \quad \sum_{\ell \le j} \, \big| \, \varepsilon_{\ell,n} \, \big| \, < \, + \infty \, .
$$
This completes the proof of Lemma \ref{lemA4}.
\end{proof}
\end{proof}

\paragraph{Acknowledgements.} It is a pleasure to thank Gr\'egory Faye for most helpful and stimulating discussions on the subject.

\bibliographystyle{alpha}
\bibliography{Source_biblio}
\end{document}